\documentclass[11pt, reqno]{amsart} 

\usepackage{amsmath, amssymb, amsfonts, amstext, amsthm, amscd}
\usepackage[mathscr]{eucal}
\usepackage{enumerate}
\usepackage{tikz,color,soul}
\usetikzlibrary{matrix,arrows}
\usepackage[utf8]{inputenc}
\usepackage[english]{babel}

\usepackage{mathrsfs}

\usepackage{verbatim}

\usepackage[new]{old-arrows}

\usepackage{longtable}
\usepackage{hyperref}
\usepackage[all]{xy}
\usepackage{ulem}
\usepackage{xcolor}

\usetikzlibrary{arrows.meta, positioning,arrows}

\newtheorem{thm}{Theorem}[section]
\newtheorem{prop}[thm]{Proposition}
\newtheorem{lemma}[thm]{Lemma}
\newtheorem{cor}[thm]{Corollary}
\newtheorem{defn}[thm]{Definition}
\newtheorem{ex}[thm]{Example}
\newtheorem{rmk}[thm]{Remark}

\numberwithin{equation}{section}

\newcommand{\C}{\mathbb C}

\newcommand{\R}{\mathbb R}
\newcommand{\Z}{\mathbb Z}

\def \mc{\mathcal}
\allowdisplaybreaks

\begin{document}

\title[Toric principal bundles]{Classification, reduction and stability of toric principal bundles}

%\begin{comment}	
\author[J. Dasgupta]{Jyoti Dasgupta}
\address{Department of Mathematics, Indian Institute of Science Education and Research Pune, India}
\email{jdasgupta.maths@gmail.com}

\author[B. Khan]{Bivas Khan}
\address{Department of Mathematics, Indian Institute of Science Education and Research Pune, India}
\email{bivaskhan10@gmail.com}

\author[I. Biswas]{Indranil Biswas}
\address{School of Mathematics, Tata Institute of Fundamental Research, Mumbai, India }

\email{indranil@math.tifr.res.in}

\author[A. Dey]{Arijit Dey}

\address{Department of Mathematics, Indian Institute of Technology-Madras, Chennai, India }

\email{arijitdey@gmail.com}

\author[M. Poddar]{Mainak Poddar}
\address{Department of Mathematics, Indian Institute of Science Education and Research Pune, India}
\email{mainak@iiserpune.ac.in}
%\end{comment}

%\author{Jyoti Dasgupta \and Bivas Khan }

\subjclass[2010]{14M25, 14J60, 32L05.}

\keywords{Toric variety, equivariant principal bundle, stability, automorphism, Levi reduction}

\begin{abstract}
Let $X$ be a complex toric variety equipped with the action of an algebraic torus $T$, and let $G$ be a complex linear 
algebraic group.  We classify all $T$-equivariant principal $G$-bundles $\mathcal{E}$ over $X$ and the morphisms 
between them. When $G$ is connected and reductive, we characterize the equivariant automorphism group 
$\text{Aut}_T(\mathcal{E} )$ of $\mathcal{E}$ as the intersection of certain parabolic subgroups of $G$ that arise naturally from the 
$T$-action on $\mathcal{E}$. We then give a criterion for the equivariant reduction of the structure group of 
$\mathcal{E}$ to a Levi subgroup of $G$ in terms of $\text{Aut}_T(\mathcal{E} )$. We use it to prove a principal bundle 
analogue of Kaneyama's theorem on equivariant splitting of torus equivariant vector bundles of small rank over a 
projective space. When $X$ is projective and $G$ is connected and reductive, we show that the notions of stability and equivariant stability are equivalent for any $T$-equivariant principal $G$-bundle over $X$.
\end{abstract}

\maketitle

%\begin{center} {\bf Declarations} \end{center}

%\section*{Declarations}

%\noindent {\bf Funding:} The second-named author is supported by postdoctoral fellowship from the  National Board for Higher Mathemtics--Department of Atomic Energy, India.  The third-named author is supported in part by a J. C. Bose fellowship from the Science and Engineering Research Board, India. The last three authors are supported in part by their respective MATRICS research grants from the Science and Engineering Research Board, India.

%\noindent {\bf Conflicts of interest/Competing interests:} On behalf of all authors, the corresponding author states that there is no conflict of interest. 

%\noindent {\bf Code availability:} Not applicable.

%\noindent {\bf Data availability:} The manuscript has no associated data.

%\noindent {\bf Corresponding author:} Mainak Poddar, email: mainak@iiserpune.ac.in, ORCID ID: 0000-0002-6019-9862

\tableofcontents

\section{Introduction}

The study of torus equivariant vector bundles on toric varieties (or, in other words, toric vector bundles), was 
initiated by Kaneyama \cite{Kan} and Klyachko \cite{Kly}. This has led to many remarkable applications. For 
instance, Kaneyama proved the existence of an equivariant splitting for any torus equivariant vector bundle over 
$\mathbb{P}^n$ of rank $r < n$ \cite[Corollary 3.5]{Kan2}. This is closely related to the Hartshorne conjecture on 
the splitting of any rank two vector bundle over $\mathbb{P}^n$ for $n \ge 7$; see \cite {IS} for a recent 
development in this regard. Also, Klyachko used the stability of toric vector bundles in his proof of Horn's 
conjecture on eigenvalues of sums of Hermitian matrices \cite{Kly2}. In \cite{Pay}, Payne constructed the 
moduli stack of toric vector bundles with fixed equivariant total Chern class as a quotient of a fine moduli scheme of
framed toric vector bundles by a linear group action. Moreover, in the rank $3$ case, he  proved that 
Murphy's law for singularity holds for the moduli space of toric vector bundles for certain types of toric varieties. 

Several attempts have been made recently to study the more general category of toric principal bundles (cf. \cite{BDP, BDP2, BDP3, KM}). In this article, we prove some foundational results regarding the classification 
of torus equivariant principal bundles over complex toric varieties and their morphisms. We also study
 their equivariant reduction and stability. In particular, we prove a principal bundle analogue of Kaneyama's result mentioned above.

Let $X$ be a complex toric variety with respect to a given action of an algebraic torus $T$, and let $G$ be a complex 
linear algebraic group. A description of $T$-equivariant principal $G$-bundles when $X$ is nonsingular was given 
in the article \cite{BDP}. It is in terms of {\it admissible collections} of homomorphisms
$\rho_{\sigma}\,:\, T \,\longrightarrow\, G$ and group elements $P(\tau,\, \sigma)\,\in\, G$ that satisfy
some extension and cocycle conditions respectively (see Definition \ref{admissible}). Here, $\sigma$ and $\tau$ vary 
over the set of maximal cones in the fan of $X$. Two admissible collections are said to be equivalent if they 
satisfy certain conjugacy relations. The isomorphism classes of $T$-equivariant principal $G$-bundles over $X$ are 
claimed to be in bijective correspondence with the equivalence classes of admissible collections \cite[Theorem 
3.2]{BDP}. However, the definition of the equivalence relation for admissible collections in \cite{BDP} is insufficient 
to prove that equivalent admissible collections produce isomorphic torus equivariant principal bundles.  We note here
that extension conditions are required not only in the definition of admissible collections, but also in defining 
their equivalence (see Definition \ref{equiv_admissible}). This affects the statement and proof of the main 
classification theorems, Theorem 3.2, Theorem 4.5 and Theorem 4.6 of \cite{BDP}. The corrected versions of these 
appear herein as Theorem \ref{classifi}, Theorem \ref{K4} and Theorem \ref{K5} respectively.  We also extend the 
first two theorems to singular toric varieties using the recent trivialization results in \cite{ESP} or \cite{BDP3}. (The statement and proof of other results of \cite{BDP} remain unaltered.)

Central to the approach of \cite{BDP} is the notion of a {\it distinguished section} of a torus equivariant 
principal bundle over an affine toric variety $X_{\sigma}$. We recall that
a distinguished section is a section whose behaviour 
under the torus action is determined by a homomorphism $\rho_{\sigma}\,:\, T
\,\longrightarrow\, G$, see Definition \ref{ds} or 
\cite[Section 2]{BDP} for more details. Here, we give a complete characterisation of the distinguished sections 
through Lemmas \ref{dsclass} and \ref{dsclass2}, leading to the main classification result, namely Theorem 
\ref{classifi}. When we fix an embedding $G \,\hookrightarrow\, {\rm GL}(r, \mathbb{C})$, the information regarding the 
homomorphism $\rho_{\sigma}\,:\, T \,\longrightarrow\, G$ can be transformed into data comprising of certain collection of vectors in 
$\mathbb{Z}^r$, following Kaneyama \cite{Kan}. This gives classification results, namely
Theorems \ref{K4} and \ref{K5}, 
that have a more combinatorial flavour, generalizing the corresponding results of Kaneyama \cite[Theorem 1.3]{Kan2} 
in the case of vector bundles. We note that the data representing a principal bundle, obtained as outlined above, match with the data representing the associated vector bundle corresponding to the aforesaid embedding of $G$ in ${\rm GL}(r, \mathbb{C})$. However, the equivalence relations of the data that describe isomorphisms of these bundles are slightly different in the two categories.  

Next, we study the equivariant automorphism group $\text{Aut}_T(\mathcal{E})$ of a \(T\)-equivariant principal $G$-bundle 
$\mathcal{E}$. We obtain a very precise description of $\text{Aut}_T(\mathcal{E})$ as a subgroup of $G$, see Proposition 
\ref{automorphism}. An analogous description of the set of $T$-equivariant morphisms between two equivariant 
principal $G$-bundles is given by Proposition \ref{morph}. When $G$ is reductive, $\text{Aut}_T(\mathcal{E})$ is the intersection of parabolic subgroups of $G$ corresponding to certain one parameter subgroups that arise naturally 
from the fan of $X$ and the homomorphisms $\rho_{\sigma}$ mentioned above; see Theorem \ref{Aut}. It is easy to observe that $\text{Aut}_T(\mathcal{E})$ always contains the 
center $Z(G)$ of $G$. To illustrate our technique, in Section 8, we compute $\text{Aut}_T(\mathcal{E})$ when $\mathcal{E}$ is the tangent frame bundle 
of a nonsingular projective toric variety $X$, and $X$ has Picard number $\leq 2$. Under these assumptions, 
$\text{Aut}_T(\mathcal{E}) \,=\, Z(G)$ when $X$ is $\mathbb{P}^n$, but $\text{Aut}_T(\mathcal{E})$ can be strictly bigger than $Z(G)$ when Picard number is two; see Section \ref{pictwo}. These results may be recovered by using Klyachko filtrations of the tangent bundle. In Section \ref{extsg}, we use extension of structure group to produce an example of the equivariant automorphism group of a principal ${\rm SL}(3, \C)$-bundle over $\mathbb{P}^2$, whose computation is beyond the scope of vector bundle techniques. In particular, the example shows that the automorphism group can become strictly bigger under an extension of structure group.

Using the classification result, Theorem \ref{classifi}, we derive a criterion, namely Proposition \ref{reduction}, 
for the existence of a $T$-equivariant reduction of the structure group of $\mathcal{E}$ to a subgroup $H\,\subset 
\,G$. Then, in Proposition \ref{levi_reduction}, under the assumption that $G$ is reductive, we show that 
$\mathcal{E}$ admits an equivariant reduction of structure group to a Levi subgroup \(H\) of \(G\) if and only if 
\(Z^0(H)\), the component of the center of $H$ containing the
identity element, is contained in \(\text{Aut}_T(\mathcal{E})\).

We apply the above result on Levi reduction to study the question of equivariant splitting of $\mathcal{E}$ as 
follows. Let \(\phi\,:\,G \,\longrightarrow\, G'\) be an injective homomorphism of reductive linear algebraic groups over 
\(\C\).  Denote by \(\mathcal{E}_{\phi}\) the equivariant principal \(G'\)-bundle obtained from \(\mathcal{E}\) by 
extending its structure group  via \(\phi\). Recall that a principal $G$-bundle is said to be split if it admits a 
reduction of structure group to a torus. In Theorem \ref{extension split}, using Proposition \ref{levi_reduction}, 
we prove that if \(\mathcal{E}_{\phi}\) splits equivariantly, then \(\mathcal{E}\) also splits equivariantly. Now, 
consider $G'$ to be ${\rm GL}(r, \mathbb{C})$, with $r \,< \,n $, and $X$ to be $\mathbb{P}^n$. Then,
using Kaneyama's result 
\cite[Corollary 3.5]{Kan2} and Theorem \ref{extension split}, we prove that \(\mathcal{E}\) splits equivariantly;
see Theorem \ref{equivkan}. We also give an alternative proof of Theorem \ref{equivkan} using results in 
\cite{BGH}, \cite{BP}.
   
It is known that the notions of semistability (respectively, stability) and  equivariant semistability (respectively, equivariant stability) of a $T$-equivariant torsion-free sheaf over a complex projective toric variety  are equivalent (see \cite[Proposition 4.13]{Kool} and \cite[Theorem 2.1]{BDGP}).  In the final section, we study  a similar question for $T$-equivariant principal $G$-bundles over a projective toric 
variety when $G$ is a connected reductive affine algebraic group. We show that the notions of semistability (respectively, stability) and equivariant semistability (respectively, equivariant stability) coincide in this case as well;
see Theorem \ref{thmstable}.
   
It was observed in \cite{BDP} that the isomorphism classes of
holomorphic $T$-equivariant principal bundles on a nonsingular toric
variety are identified with the isomorphism classes of algebraic $T$-equivariant principal bundles,
essentially because a holomorphic 
homomorphism from the algebraic torus to a linear algebraic group is in fact algebraic. In view of this,
without any loss of generality, we will focus here on the algebraic case.

\subsection*{Acknowledgement} 
The fourth author would like to thank Sam Payne for a stimulating discussion. The second-named author is supported by postdoctoral fellowship from the  National Board for Higher Mathemtics--Department of Atomic Energy, India.  The third-named author is supported in part by a J. C. Bose fellowship from the Science and Engineering Research Board, India. The last three authors are supported in part by their respective MATRICS research grants from the Science and Engineering Research Board, India.

\section{Preliminaries}\label{prelim}

In this section, we introduce some definitions and notation to be used throughout the paper. We also recall some results from \cite{BDP} and \cite{ESP} for the convenience of the reader.
  
Let $T\,\cong\, \left(\C^*\right) ^n$ be an algebraic torus. Let $M\,:=\,\text{Hom}(T,\, \C^*) \,\cong\, \Z^n$ be its
character 
lattice and $N\,:=\,\text{Hom}_{\Z}(M, \,\Z)$ the dual lattice. Note that $N$ parametrizes the one parameter subgroups of $T$. The character of
$T$ corresponding to any $u \,\in\, M$ will be denoted 
by $\chi^u$. For any $v \in N$, we denote by $\lambda^v: \mathbb{C}^{\ast} \to T $ the corresponding one parameter subgroup (abbreviated as $1$-psg henceforth) of $T$. We denote by \(\langle \, , \rangle : M \times N \rightarrow \Z \) the natural pairing between \(M\) and \(N\). Let $\Xi$ be a fan in $N \otimes_{\Z} \R$ defining a toric variety $X\,=\,X(\Xi)$ under 
the action of the torus $T$. Let $\Xi(k)$ denote the set of all $k$-dimensional cones of $\Xi$, and let $\sigma(k)$ denote the set of all
$k$-dimensional faces (subcones) of a cone $\sigma$ in $\Xi$. For an element $\alpha$ of $\Xi(1)$, we denote its 
primitive integral generator by $v_{\alpha}$. For a cone \(\sigma \in \Xi\), the dual cone $\sigma^{\vee} $ and the orthogonal space $\sigma^{\perp}$ are defined respectively as follows:
\begin{equation*}
	\sigma^{\vee}=\{u \in M \otimes_{\Z} \R ~|~ \langle u , v \rangle \geq 0 \text{ for all } v \in \sigma \} \text{ and } 	\sigma^{\perp}=\{u \in M \otimes_{\Z} \R ~|~ \langle u , v \rangle = 0 \text{ for all } v \in \sigma \}. 
\end{equation*}
Let \(S_{\sigma}:=\sigma^{\vee} \cap M \) be the affine semigroup (which is finitely generated by Gordon's Lemma \cite[Proposition 1, p. 12]{Ful}) corresponding to a cone $\sigma \in \Delta$. Let \(\C[S_{\sigma}]:=\C[\chi^u ~|~ u \in S_{\sigma}]\) be the associated semigroup algebra and  \(X_{\sigma}:=\text{Spec }\C[S_{\sigma}]\) be the corresponding affine toric variety. The closed points of \(X_{\sigma}\) are in one-to-one correspondence with \(\text{Hom}_{\text{sg}}(S_\sigma, \C)\), the set of semigroup homomorphisms from $S_\sigma$ to $\C$, where $\C$ is considered as a semigroup under multiplication. The affine toric variety \(X_{\sigma}\) has a unique distinguished point, denoted by \(x_{\sigma}\) (see \cite[Section 2.1]{Ful}). This point is defined by a semigroup homomorphism

\begin{equation}\label{distinguished_point}
 S_{\sigma} \rightarrow \C, \,	
u  \mapsto
\begin{cases}
1, \text{  if } u \in \sigma^{\perp},\\
0, \text{  otherwise}.
\end{cases}
\end{equation}
The orbit containing the distinguished point \(x_{\sigma}\) under the action of the torus \(T\) is denoted by \(O_{\sigma}\)  (see \cite[Section 3.1]{Ful}). Note that $O_{\sigma} = \text{Spec }\C[\sigma^{\perp} \cap M] $. When \(\sigma \in \Xi(n)\), we have  \(O_{\sigma} = \{ x_{\sigma} \} \). For the trivial cone $\{0\}$, we have the principal orbit \(O:= O_{\{0\}} =X_{\{0\}}\). Moreover, there is a canonical
identification of $O$ with $T$ as both are defined as $\text{Spec }\C[M] $ (cf. \cite[Proposition 1.6]{Oda} or \cite[p. 17, 52]{Ful}). We denote the closed point of $O$ corresponding to the identity element $1_T$ of $T$ by $x_0$.

For any inclusion of cones $\tau \,\preceq\, \sigma$ in the fan, we have a canonical inclusion of affine toric varieties 
$X_{\tau} \,\subseteq\, X_{\sigma}$. Thus we regard $T$ as an open subset
\begin{equation}\label{eqinc}
T\, \subseteq\, X_{\sigma}
\end{equation}
of every affine toric variety $X_{\sigma}$.  We use 
this convention in the rest of the paper.

Let \(v \in N\) and \(\sigma \in \Xi\). Then 
\[v \in \sigma \text{ if and only if } \lim\limits_{t \rightarrow 0} \lambda^v(t) x_0 \text{ exists in } X_{\sigma}.\] Furthermore, if \(v \) belongs to the  relative interior of the cone $\sigma$, then we have \(\lim\limits_{t \rightarrow 0} \lambda^v(t) x_0 = x_{\sigma} \) (see \cite[Claim 1, p.~38]{Ful}). 

Let \(T_{\sigma}\) denote the stabilizer 
of a point in \(O_{\sigma}\) under the action of $T$.  Then \(T_{\sigma}\) is independent of the choice of the point in \(O_{\sigma}\) as \(T\) is abelian. Note that \(T_{\sigma}\) 
is a subtorus of \(T\) with cocharacter lattice \(N_{\sigma}\), where \(N_{\sigma}\) is
the sublattice of \(N\) spanned by the elements of 
$\sigma \cap N$. We recall that each \(O_{\sigma}\) may be given a group structure using a canonical identification 
with \(T/ T_{\sigma}\); see  \cite[Proposition 1.6]{Oda} or \cite[p.~52]{Ful}. 

For any $\sigma \in \Xi$, there is an exact sequence
\begin{eqnarray}\label{Split1}
0 \,\longrightarrow\, N_{\sigma} \,\longrightarrow\,N\,\longrightarrow\, N/N_{\sigma}\,\longrightarrow\, 0\, .
\end{eqnarray}
Then applying \(Hom_{\Z}( \, \cdot \, ,\, \Z)\), we get the following exact sequence
\begin{equation*}
	0 \longrightarrow \sigma^{\perp} \cap M \longrightarrow M \longrightarrow M / (\sigma^{\perp} \cap M) \longrightarrow 0.
\end{equation*}
Now, applying \(Hom_{\Z}( \, \cdot \, ,\, \C^*)\), we get the exact sequence
\begin{eqnarray}\label{Split2}
1\,\longrightarrow\, T_{\sigma}\,\longrightarrow\, T\,\longrightarrow\, O_{\sigma}\,\longrightarrow\, 1\, .
\end{eqnarray}
Since $\sigma$ is saturated, the quotient
\(N/N_{\sigma}\) is a torsion free, and hence, free \(\Z\)-module (see \cite[p. 29]{Ful}). Thus \eqref{Split1} splits,
and a splitting of it induces a splitting of \eqref{Split2}. For each \(\sigma\), fix once and for all, a splitting of \eqref{Split2},
and let 
\begin{equation}\label{projection}
\pi_{\sigma} \,:\,T\,\longrightarrow \,T_{\sigma}
\end{equation}
be the projection associated to it. 

Let $G$ denote a complex linear algebraic group. Suppose \(\mathcal{E}\) is a $T$-equivariant 
algebraic principal $G$-bundle over $X$ which admits an algebraic trivialization of the underlying principal \(G\)-bundle over 
$X_{\sigma}$. Let \(s\,:\,X_{\sigma}\,\longrightarrow \,\mathcal{E}\) be any algebraic section. 
We encode the $T$-action on \(\mathcal{E}|_{X_{\sigma}}\) in terms of the section $s$ as follows:

\begin{defn}[{\cite[Definition 2.1]{BDP}}]\label{laf}
For any \(x \,\in\, X_{\sigma}\) and \(t \,\in\, T\), define $\rho_s(x, t) \in G$ by
\[ts(x)\, =\, s(tx) \cdot \rho_s(x,\, t) .\]
Since the action of $G$ on each fiber of \(\mathcal{E}\) is free and transitive, it follows that \(\rho_s(x,\, t)\)
is well defined and it is algebraic in the variables $x$ and $t$. We say that
$\rho_s \,:\, X_{\sigma} \times T\,\longrightarrow\, G$ is the local action function associated to $s$.
\end{defn}

\begin{defn}[{\cite[Definition 2.5]{BDP}}]\label{ds}
We say that a section $s$ of $\mathcal{E}|_{X_{\sigma}}$ is distinguished if
\begin{itemize}
\item $ \rho_s(x,\, \cdot)$ is independent of $x$, and
		
\item  $\rho_s(x, \,\cdot)$ factors through the projection $\pi_{\sigma}\,:\, T \,\longrightarrow\, T_{\sigma}$ defined in \eqref{projection}.
\end{itemize}
\end{defn}

If $s$ is a distinguished section, then $\rho_s$ induces a homomorphism from \(T\) to \(G\), which is also denoted 
by $\rho_s$ (see \cite[Lemma 2.3]{BDP}).

\begin{lemma}[{\cite[Corollary 2.3]{ESP}}]\label{lem:cone1} Let \(\mathcal{E}\) be a $T$-equivariant 
	algebraic principal $G$-bundle over $X$. For any cone $\sigma$, the restriction $\mathcal{E}|_{X_{\sigma}}$   is trivial and admits a
distinguished section.
\end{lemma}

\begin{lemma}[{\cite[Lemma 2.9]{BDP}}]\label{distconj} 
The homomorphisms \(\rho_s\,:\, T \,\longrightarrow\, G\), induced by different distinguished sections \(s\)  of $\mc{E}|_{X_{\sigma}}$, are equal up to conjugation by elements
of $G$. More precisely, if \(s\) and \(s'\) are two distinguished sections such that $s'(x_{\sigma}) \,=\, s(x_{\sigma})\cdot g$ for some \(g \in G\), then $\rho_{s'} \,=\, g^{-1} \rho_s g$. 
\end{lemma}

\begin{defn}[{\cite[Definition 3.1]{BDP}}]\label{admissible}
Let $\Xi^{*}$ denote the set of all maximal cones in the fan $\Xi$.
An admissible collection \(\{\rho_{\sigma},\, P(\tau,\, \sigma)\}\) consists of a collection of homomorphisms
\[\{\rho_{\sigma}\,:\,T \,\longrightarrow\, G ~\,\mid\,~ \sigma \,\in\, \Xi^{*} \}\] and a collection of elements
\(\{P(\tau,\, \sigma)\,\in\, G ~\,\mid\,~ \tau, \sigma \,\in\, \Xi^{*} \}\) satisfying the following conditions:
\begin{enumerate}[$($i$)$]
\item $\rho_{\sigma}$ factors through \(\pi_{\sigma}\,:\,T \,\longrightarrow\, T_{\sigma} \) $($see \eqref{projection}$)$.

\item For every pair \((\tau, \sigma)\) of cones in $\Xi^{*}$, the function $$T\, \longrightarrow G,\, \ \
t\, \longmapsto\, \rho_{\tau}(t)P(\tau, \sigma)\rho_{\sigma}(t)^{-1}$$
extends to a \(G\)-valued function \(\phi_{ \tau \sigma}: X_{\sigma} \cap X_{\tau}\, \rightarrow G\) $($see \eqref{eqinc}$)$.

\item \(P(\sigma, \,\sigma)\,=\,1_G\) $($the identity element of $G)$ for all $\sigma \,\in\, \Xi^{*}$.

\item For every triple \((\tau, \sigma, \delta)\) of cones in $\Xi^{*}$, the cocycle condition
\[P(\tau,\, \sigma) P(\sigma,\, \delta) P(\delta, \,\tau)\,=\,1_G\] holds.
\end{enumerate}
\end{defn}

We now briefly recall from \cite[Section 3]{BDP}, how admissible collections are 
associated to $T$-equivariant principal \(G\)-bundles.

Let $\mc{E}$ be a $T$-equivariant algebraic principal $G$-bundle on $X$.
For any $\sigma \,\in\, 
\Xi^{*}$, set $\mathcal{E}_{\sigma}\,:=\,\mc{E}|_{X_{\sigma}}$, and let \(s_{\sigma}\) be a 
distinguished section of \(\mc{E}_{\sigma}\). Let $\rho_{\sigma}\,:=\,\rho_{s_{\sigma}}\,:\, 
T\, \longrightarrow \,G$ be the corresponding homomorphism. Then, we have a \(T\)-equivariant 
trivialization of $\mathcal{E}_{\sigma}$,
\begin{equation}\label{triv} 
\psi_{\sigma} \,:\, \mathcal{E}_{\sigma}\,\longrightarrow\, X_{\sigma} \times G, \,\,\, ~ s_{\sigma}(x)
\cdot g \,\longmapsto\, (x,\, g), 
\end{equation}
where the action of \(T\) on \(X_{\sigma} \times G\) is given by
\[t(x,\, g)\, =\, (tx,\, \rho_{\sigma}(t)g),\,\,\,\,\, \text{ for all } \,\,\, x \,\in\, X_{\sigma}, \,
g \,\in\, G,\, t \,\in\, T.\]
For $\sigma,\, \tau \,\in\, \Xi^{*}$, the transition function $\phi_{\tau \sigma}\,:\,
X_{\sigma} \cap X_{\tau}\,\longrightarrow\, G$ is defined by the relation 
\begin{equation}\label{transition}
s_{\sigma}(x)\,=\, s_{\tau}(x) \cdot \phi_{\tau \sigma}(x)\,\,\,\,\,\text{ for all }\,\,\, x \,\in\, X_{\sigma} \cap X_{\tau}\, .
\end{equation}
Then, using the \(T\)-equivariance property of $\mathcal{E}$, we have 
\begin{equation}\label{transprop}
\phi_{\tau \sigma}(tx)\,=\,\rho_{\tau}(t) \phi_{\tau \sigma}(x) \rho_{\sigma}(t)^{-1}
\end{equation} 
for all $x \,\in\, X_{\sigma} \cap X_{\tau}$ and $t \,\in\, T$ (see \cite[Section 3]{BDP}). Define
\[P(\tau, \,\sigma)\,:=\,\phi_{\tau \sigma}(x_0) \,.\]
Then the collection \(\{\rho_{\sigma},\, P(\tau, \,\sigma)\}\) gives rise to an admissible collection
as defined in Definition \ref{admissible}. 

\section{Classification of equivariant principal bundles}

In this section, we derive a characterisation of distinguished sections of a $T$-equivariant principal $G$-bundle 
over an affine toric variety \(X_{\sigma}\). This is used in order to obtain a description of isomorphism classes of $T$-equivariant principal 
bundles over an arbitrary toric variety \(X\).

\begin{lemma}\label{dsclass} Suppose $s_{\sigma}$ is a distinguished section
for a $T$-equivariant principal $G$-bundle $\mathcal{E}$ over an affine toric variety $X_{\sigma}$. Let
$ \varphi_{\sigma} \,:\, X_{\sigma}\,\longrightarrow\, G $ be an algebraic function. Then 
$$s'_{\sigma}(x) \,:= \,s_{\sigma}(x) \cdot \varphi_{\sigma}(x)$$ is a distinguished section for $\mathcal{E}$
over $X_{\sigma}$ if and only if $\varphi_{\sigma}$ satisfies the condition
\begin{equation}\label{phiprop}
\varphi_{\sigma}(tx_0)\, =\, \rho_{s_{\sigma}}(t) \, \varphi_{\sigma}(x_0)\,
\varphi_{\sigma}(x_{\sigma})^{-1} \rho_{s_{\sigma}}(t)^{-1} \varphi_{\sigma}(x_{\sigma})
\end{equation}
for all \,$t \,\in\, T$.
\end{lemma} 

\begin{proof}
Take a distinguished section $s_{\sigma}'= \,s_{\sigma} \cdot \varphi_{\sigma}$. From Lemma \ref{distconj} we have 
	\begin{equation}\label{rhoconj} 
	\rho_{s'_{\sigma}}(t) \,=\, \varphi_{\sigma}(x_{\sigma})^{-1} \rho_{s_{\sigma}}(t)\, \varphi_{\sigma}(x_{\sigma}) \,.
	\end{equation}
	\noindent
	Substituting $tx_0$ for $x$ in the definition of $s'_{\sigma}$, we have 
	\begin{equation}\label{sectionrlnp}
	s_{\sigma}'(t x_0)\,=\,{s}_{\sigma}(t x_0) \cdot \varphi_{\sigma}(t x_0)
\end{equation}
for all\, $t \,\in\, T$. In view of Definitions \ref{laf} and \ref{ds} it is concluded that
\begin{equation}\label{basicprop}
s(t x_0) \,=\, t s(x_0) \cdot \rho_s(t)^{-1}
\end{equation}
for all distinguished section $s$.
Then, applying \eqref{basicprop} on both sides of \eqref{sectionrlnp} we get that
\begin{equation*}
t s_{\sigma}'(x_0)\cdot \rho_{s_{\sigma}'}(t)^{-1} \,= \,
t {s}_{\sigma}(x_0)\cdot \rho_{s_{\sigma}}(t)^{-1} \varphi_{\sigma}(t x_0) \,.
	\end{equation*}

Now, by \eqref{rhoconj} and \eqref{sectionrlnp}, 
\begin{equation*}
t s_{\sigma}(x_0)\cdot \varphi_{\sigma}(x_0) \varphi_{\sigma}(x_{\sigma})^{-1}
\rho_{s_{\sigma}}(t)^{-1} \varphi_{\sigma}(x_{\sigma})      
\,=\, t {s}_{\sigma}(x_0)\cdot   \rho_{s_{\sigma}}(t)^{-1} \varphi_{\sigma}(t x_0) \,.
	\end{equation*}
Therefore, it follows that 
	\begin{equation*}
	\varphi_{\sigma}(tx_0)\, =\, \rho_{s_{\sigma}}(t) \, \varphi_{\sigma}(x_0)\,\varphi_{\sigma}(x_{\sigma})^{-1}
\rho_{s_{\sigma}}(t)^{-1} \varphi_{\sigma}(x_{\sigma})
\end{equation*}
for all\, $t \,\in\, T$. This completes one direction of the proof. 

To prove the converse, assume that $\varphi$ satisfies the condition \eqref{phiprop}. It can be shown that
$x_0$ may be replaced by an arbitrary point $x \,\in\, X_{\sigma}$ in \eqref{phiprop}. To prove this, first
consider an arbitrary point $x\,=\, t' x_0 \,\in\, O$, where $t' \,\in\, T$.  Then, by \eqref{phiprop}, we have 
\begin{equation}\label{overT}\begin{split}
	\varphi_{\sigma}(tx)= & \varphi_{\sigma}(tt'x_0) \\
	=& \rho_{s_{\sigma}}(tt') \, \varphi_{\sigma}(x_0)\, \varphi_{\sigma}(x_{\sigma})^{-1} \rho_{s_{\sigma}}(tt')^{-1} \varphi_{\sigma}(x_{\sigma})   \\
	=&  \rho_{s_{\sigma}}(t) \, \rho_{s_{\sigma}}(t') \, \varphi_{\sigma}(x_0)\, \varphi_{\sigma}(x_{\sigma})^{-1} \rho_{s_{\sigma}}(t')^{-1}\rho_{s_{\sigma}}(t)^{-1}  \varphi_{\sigma}(x_{\sigma})   \\
	=&  \rho_{s_{\sigma}}(t) \, \rho_{s_{\sigma}}(t') \, \varphi_{\sigma}(x_0)\, \varphi_{\sigma}(x_{\sigma})^{-1} \rho_{s_{\sigma}}(t')^{-1} \varphi_{\sigma}(x_{\sigma}) \,
	 \varphi_{\sigma}(x_{\sigma})^{-1}  \rho_{s_{\sigma}}(t)^{-1}  \varphi_{\sigma}(x_{\sigma})  \\
	=&  \rho_{s_{\sigma}}(t) \,  \varphi_{\sigma}(t' x_0) \,
	\varphi_{\sigma}(x_{\sigma})^{-1}  \rho_{s_{\sigma}}(t)^{-1}  \varphi_{\sigma}(x_{\sigma}) \\
	 =&  \rho_{s_{\sigma}}(t) \,  \varphi_{\sigma}(x) \,
	 \varphi_{\sigma}(x_{\sigma})^{-1}  \rho_{s_{\sigma}}(t)^{-1}  \varphi_{\sigma}(x_{\sigma}).	 
	\end{split}
	\end{equation}
	As $O$ is Zariski--open dense in $X_{\sigma}$, and the two sides of  \eqref{overT} are algebraic
in $x$, it follows that
\begin{equation}\label{phipropgen}
\varphi_{\sigma}(tx)\, =\, \rho_{s_{\sigma}}(t) \, \varphi_{\sigma}(x)\, \varphi_{\sigma}(x_{\sigma})^{-1}
\rho_{s_{\sigma}}(t)^{-1} \varphi_{\sigma}(x_{\sigma})
\end{equation} 
for all\, $x \,\in\, X_{\sigma}$ and $t\,\in\, T$.

We need to show that the local action function
$ \rho_{s'_{\sigma}} (x,\,t) $ is independent of $x$.
As $s_{\sigma}$ is distinguished, using \cite[Equation (2.1)]{BDP},
\begin{equation*} \rho_{s'_{\sigma}} (x,\, t)\, =\, \varphi_{\sigma}(tx)^{-1} \rho_{s_{\sigma}}(t)\,
\varphi_{\sigma}(x) \,.
\end{equation*}
Then, applying \eqref{phipropgen} we obtain that
\begin{equation*}
\begin{split}
\rho_{s'_{\sigma}} (x,\, t) \,=\,& \varphi_{\sigma}(x_{\sigma})^{-1} \rho_{s_{\sigma}}(t)\, \varphi_{\sigma}(x_{\sigma})\, \varphi_{\sigma}(x)^{-1} \, \rho_{s_{\sigma}}(t)^{-1} 
\rho_{s_{\sigma}}(t)\, \varphi_{\sigma}(x) \\
=\,&  \varphi_{\sigma}(x_{\sigma})^{-1} \rho_{s_{\sigma}}(t)\, \varphi_{\sigma}(x_{\sigma}) \,.
\end{split}
\end{equation*}
This shows that $\rho_{s'_{\sigma}}(x,\,t)$ is independent of $x$. Moreover, it factors through the 
projection $\pi_{\sigma}\,:\, T \,\longrightarrow\, T_{\sigma}$ as $\rho_{s_{\sigma}}(t)$ does so (see \eqref{projection}). 
Consequently, $s'_{\sigma}$ is a distinguished section.
\end{proof}

\begin{lemma}\label{dsclass2} Suppose $s_{\sigma}$ and $s_{\sigma}'$ are two arbitrary distinguished sections for a $T$-equivariant principal $G$-bundle $\mathcal{E}$ over an affine toric variety $X_{\sigma}$. 
Let $ \varphi_{\sigma} : X_{\sigma} \,\longrightarrow\, G $ be the algebraic function defined by
$s_{\sigma}'(x) \,=\, s_{\sigma}(x) \cdot \varphi_{\sigma}(x)$. Then,
\begin{enumerate}[$($a$)$]
\item $\rho_{s'_{\sigma}}(t) \,=\, \varphi_{\sigma}(x_{\sigma})^{-1} \rho_{s_{\sigma}}(t)\, \varphi_{\sigma}(x_{\sigma})$.

\item $\rho_{s_{\sigma}}(t)\varphi_{\sigma}(x_0) \varphi_{\sigma}(x_{\sigma})^{-1}\rho_{s_{\sigma}}(t)^{-1}$ extends to a \(G\)-valued function over \(X_{\sigma} \,.\) $ \text{In particular, }\\
\lim\limits_{t x_0 \rightarrow x_{\sigma}} \rho_{s_{\sigma}}(t)\varphi_{\sigma}(x_0)
\varphi_{\sigma}(x_{\sigma})^{-1}\rho_{s_{\sigma}}(t)^{-1}\,=\,1_G\,. $
\end{enumerate} 
\end{lemma}	

\begin{proof}  Suppose $s_{\sigma}$ and $s_{\sigma}'$ are two arbitrary distinguished sections.
Then part $(a)$ follows immediately from Lemma \ref{distconj} or equation \eqref{rhoconj}.

Note that by Lemma \ref{dsclass} we have
	 \[\varphi_{\sigma}(tx_0) \,=\, \rho_{s_{\sigma}}(t) \, \varphi_{\sigma}(x_0)\,
\varphi_{\sigma}(x_{\sigma})^{-1} \rho_{s_{\sigma}}(t)^{-1} \varphi_{\sigma}(x_{\sigma}) 
\] 
for all \,$t \,\in\, T$.

As $\varphi_{\sigma}$ is regular on $X_{\sigma}$, it follows that \(\rho_{s_{\sigma}}(t) 
\varphi_{\sigma}(x_0)\varphi_{\sigma}(x_{\sigma})^{-1} \rho_{s_{\sigma}}(t)^{-1} \) extends to an algebraic 
function on \(X_{\sigma}\).

Recall equation \eqref{basicprop},
\begin{equation*}
s(t x_0) \,=\, t s(x_0) \cdot \rho_s(t)^{-1} \,. \end{equation*}
Now setting $s \,=\, s_{\sigma}'$, and taking limit as \(t  x_0 \rightarrow x_{\sigma}\), 
we get that
\[s'_{\sigma}(x_{\sigma})\,=\,\lim\limits_{t x_0 \rightarrow x_{\sigma}} t  s'_{\sigma}(x_0) \cdot \rho_{s'_{\sigma}}(t)^{-1}. \]
Then, using the definition of $\varphi_{\sigma}$, and part $(a)$ of this lemma, we see that 
	\[ {s}_{\sigma}(x_{\sigma}) \cdot \varphi_{\sigma}(x_{\sigma})\,=\,\lim\limits_{t  x_0 \rightarrow x_{\sigma}} t  {s}_{\sigma}(x_0) \cdot\varphi_{\sigma}(x_0) \, \varphi_{\sigma}(x_{\sigma})^{-1} \rho_{s_{\sigma}}(t)^{-1} \varphi_{\sigma}(x_{\sigma}).\] 
		\noindent 
	Cancelling \(\varphi_{\sigma}(x_{\sigma})\) on both sides,
	\begin{equation*}\begin{split}
	{s}_{\sigma}(x_{\sigma})= &\lim\limits_{t  x_0 \rightarrow x_{\sigma}} t  {s}_{\sigma}(x_0) \cdot \varphi_{\sigma}(x_0) \,\varphi_{\sigma}(x_{\sigma})^{-1} \rho_{s_{\sigma}}(t)^{-1} \\
	=&\lim\limits_{t  x_0 \rightarrow x_{\sigma}} t  {s}_{\sigma}(x_0) \cdot \rho_{s_{\sigma}}(t)^{-1}\rho_{s_{\sigma}}(t)\,\varphi_{\sigma}(x_0)\, \varphi_{\sigma}(x_{\sigma})^{-1} \rho_{s_{\sigma}}(t)^{-1} \\
	=&\lim\limits_{t  x_0 \rightarrow x_{\sigma}}   {s}_{\sigma}(t  x_0) \cdot \rho_{s_{\sigma}}(t)\,\varphi_{\sigma}(x_0)\, \varphi_{\sigma}(x_{\sigma})^{-1} \rho_{s_{\sigma}}(t)^{-1} \\
	=& \; {s}_{\sigma}(x_{\sigma}) \cdot \lim\limits_{t  x_0 \rightarrow x_{\sigma}}  \rho_{s_{\sigma}}(t)\,\varphi_{\sigma}(x_0) \,\varphi_{\sigma}(x_{\sigma})^{-1} \rho_{s_{\sigma}}(t)^{-1} .
	\end{split}
	\end{equation*}
	Note that to derive the last equality, we have used the fact that
\(\rho_{s_{\sigma}}(t) \, \varphi_{\sigma}(x_0) \,\varphi_{\sigma}(x_{\sigma})^{-1} \rho_{s_{\sigma}}(t)^{-1} \)
extends to an algebraic function on \(X_{\sigma}\). Finally, observe from the last equation
that \[\lim\limits_{t  x_0 \rightarrow x_{\sigma}} \rho_{s_{\sigma}}(t)\,\varphi_{\sigma}(x_0)\,
\varphi_{\sigma}(x_{\sigma})^{-1} \rho_{s_{\sigma}}(t)^{-1} \,=\, 1_G\,.\]
This concludes the proof.
\end{proof}

\begin{prop}\label{iso_admi}
Let \(\mathcal{E}_1\) and \(\mathcal{E}_2\) be two \(T\)-equivariant principal \(G\)-bundles on a toric variety \(X\), which are \(T\)-equivariantly isomorphic. For  \(k=1,2 \), let \(\{s^k_{\sigma}\}_{\sigma \in \Xi^*}\) be a collection of distinguished sections which associates an admissible collection \(\{\rho_{s^k_{\sigma}}, \ P^k(\tau, \, \sigma)\}\)  to \(\mathcal{E}_k\). Then for each $\sigma \in \Xi^{*}$, there exists $\alpha_{\sigma}, \  \beta_{\sigma} \in G$ such that the following conditions hold:
\begin{enumerate}[$($a$)$]
\item $\rho_{s^2_{\sigma}}=\alpha_{\sigma}^{-1} \rho_{s^1_{\sigma}} \alpha_{\sigma}$.

\item \(P^2(\tau, \, \sigma)\,=\,\beta_{\tau}^{-1} P^1(\tau, \, \sigma) \beta_{\sigma}\) for every
pair of cones \((\tau,\, \sigma)\) in $\Xi^{*}$, 

\item $\rho_{s^1_{\sigma}}(t)\beta_{\sigma} \alpha_{\sigma}^{-1}\rho_{s^1_{\sigma}}(t)^{-1}$ extends to a 
\(G\)-valued function over \(X_{\sigma} \,.\) $ \text{In particular, }\\ \lim\limits_{t x_0 \rightarrow 
x_{\sigma}} \rho_{s^1_{\sigma}}(t)\beta_{\sigma} \alpha_{\sigma}^{-1}\rho_{s^1_{\sigma}}(t)^{-1}\,=\,1_G\,. $
\end{enumerate}
\end{prop}

\begin{proof}
Let \(\Phi\,:\,\mathcal{E}_1\,\longrightarrow\, \mathcal{E}_2\) be a \(T\)-equivariant bundle isomorphism.
For $\sigma \in \Xi^*$, consider the section
$$ \widehat{s}^2_{\sigma}\,:= \,\Phi \circ s^1_{\sigma} $$ of $\mathcal{E}_2$ over \(X_{\sigma}\). By
\cite[Lemma 2.4]{BDP}  we have \begin{equation}\label{rhos2hat}
\rho_{\widehat{s}^2_{\sigma}}(x, t) = \rho_{s^1_{\sigma}}(x, t) = \rho_{s^1_{\sigma}}(t)
\end{equation}
for all \,$x \,\in\, X_{\sigma}$.
This shows that \(\{\widehat{s}^2_{\sigma}\}\) defines a collection of distinguished sections of \(\mathcal{E}_2\).
Then using \eqref{transition} and \(G\)-equivariance property of $\Phi$, we have
\[\widehat{s}^2_{\sigma}(x_0)\,=\,\Phi(s^1_{\sigma}(x_0))
\,=\, \Phi(s^1_{\tau}(x_0) \cdot \phi^1_{\tau \sigma}(x_0) )
=\Phi(s^1_{\tau}(x_0)) \cdot \phi^1_{\tau \sigma}(x_0) 
\,=\, \widehat{s}^2_{\tau}(x_0) \cdot P^1(\tau, \,\sigma),\] 
where $\phi^1_{\tau \sigma}$ is the transition functions for $\mathcal{E}_1$. 
Thus the admissible collection associated to $\mathcal{E}_2$ that arises from \(\{\widehat{s}^2_{\sigma}\}\) is same as \(\{\rho_{s^1_{\sigma}}, P^1(\tau, \sigma)\}\). 

Define \(\varphi_{\sigma}\,:\,X_{\sigma}\,\longrightarrow\, G\) such that 
\begin{equation}\label{sectionrln}
s^2_{\sigma}(x)\,=\,\widehat{s}^2_{\sigma}(x) \cdot \varphi_{\sigma}(x)
\end{equation}
for all \,$x \,\in\, X_{\sigma}$.
In addition, define \[ \alpha_{\sigma}\,:=\,\varphi_{\sigma}(x_{\sigma}) \quad {\rm and} \quad \beta_{\sigma}
\,:=\, \varphi_{\sigma}(x_0) \,. \] 
Then $(a)$ and $(c)$ follow immediately from Lemma \ref{dsclass2} by setting $\mathcal{E}
\,=\, \mathcal{E}_2$, $s_{\sigma} \,=\, \widehat{s}_{\sigma}^2 $ and $s'_{\sigma} \,=\, s^2_{\sigma} $.

Let $\phi^2_{\tau \sigma}$ and $\widehat{\phi}^2_{\tau \sigma}$ be the transition functions for $\mathcal{E}_2$ determined by the collections  \(\{s^2_{\sigma}\}\) and $\{\widehat{s}^2_{\sigma}\}$ 
of distinguished sections respectively. In other words, for all \(x \in X_{\sigma} \cap X_{\tau}\) we have
\begin{equation}\label{p2}
s^2_{\sigma}(x)=s^2_{\tau}(x) \cdot \phi^2_{\tau \sigma}(x) \quad \text{ and } \quad \widehat{s}^2_{\sigma}(x)
\,=\,\widehat{s}^2_{\tau}(x) \cdot \widehat{\phi}^2_{\tau \sigma}(x). 
\end{equation}
Using \eqref{sectionrln} and  \eqref{p2} we get that
\begin{align*}
& {s}^2_{\sigma}(x) = {s}^2_{\tau}(x) \cdot  \phi^2_{\tau \sigma}(x)\\
\Rightarrow \; &\widehat{s}^2_{\sigma}(x) \cdot \varphi_{\sigma}(x) \,=\,\widehat{s}^2_{\tau}(x) \cdot \varphi_{\tau}(x) \phi^2_{\tau \sigma}(x)\\
\Rightarrow \; & \widehat{s}^2_{\tau}(x) \cdot \widehat{\phi}^2_{\tau \sigma}(x) \varphi_{\sigma}(x)=\widehat{s}^2_{\tau}(x) \cdot \varphi_{\tau}(x) \phi^2_{\tau \sigma}(x) \\
\Rightarrow \;  & \phi^2_{\tau \sigma}(x)=\varphi_{\tau}(x)^{-1}  \widehat{\phi}^2_{\tau \sigma}(x) \varphi_{\sigma}(x).
\end{align*}
In particular, putting \(x=x_0\) in the last equation,
\[P^2(\tau, \sigma)\,=\,\beta_{\tau}^{-1} P^1(\tau, \sigma) \beta_{\sigma}\, \] where \(\beta_{\sigma}
\,:=\,\varphi_{\sigma}(x_0)\). Thus $(b)$ holds
as well.
\end{proof}

\begin{defn}\label{equiv_admissible}
Two admissible collections \(\{\rho^1_{\sigma}, P^1(\tau, \sigma)\}\) and \(\{\rho^2_{\sigma}, P^2(\tau, 
\sigma)\}\) are said to be equivalent if for each $\sigma \in \Xi^*$, there exists $\alpha_{\sigma},\ 
\beta_{\sigma} \in G$ such that the following conditions hold:

\begin{enumerate}[$($a$)$]
	\item $\rho^2_{\sigma}=\alpha_{\sigma}^{-1}\ \rho^1_{\sigma}\ \alpha_{\sigma}$.
	\item For every pair \((\tau, \sigma)\) of cones in $\Xi^{*}$, \(P^2(\tau, \sigma)=\beta_{\tau}^{-1}\ P^1(\tau, \sigma)\ \beta_{\sigma}\).
	\item $\rho^1_{\sigma}(t)\ \beta_{\sigma} \alpha_{\sigma}^{-1}\ (\rho^1_{\sigma}(t))^{-1}$ extends to a \(G\)-valued function over \(X_{\sigma} \,.\) 
\end{enumerate}

The equivalence class of an admissible collection  \(\{\rho_{\sigma}, P(\tau, \sigma)\}\) will be denoted by \([\{\rho_{\sigma}, P(\tau, \sigma)\}]\).
\end{defn}

\begin{thm}\label{classifi}
	Let \(X\) be a toric variety and \(G\) be a complex linear algebraic group. Then the isomorphism classes of \(T\)-equivariant principal \(G\)-bundles on \(X\) are in one-to-one correspondence with the equivalence classes of admissible collections \([\{\rho_{\sigma}, P(\tau, \sigma)\}]\).
\end{thm}

\begin{proof}
 Let us define a map \(\bf{A}\) from the set of isomorphism classes of $T$-equivariant principal \(G\)-bundles to the set of equivalence classes of admissible collections given by \[{\bf A}([\mathcal{E}])  = [\{\rho_{\sigma}, P(\tau, \sigma)\}] \] as described in Section \ref{prelim}.
\noindent	 
Note that \(\bf{A}\) is well defined by Proposition \ref{iso_admi}.
	 
	Conversely, given an admissible collection \(\{\rho,P\}=\{\rho_{\sigma}, P(\tau, \sigma)\}\), a $T$-equivariant principal \(G\)-bundle \(E(\{\rho,P\})\) can be constructed as follows (see proof of \cite[Theorem 3.2]{BDP}): Let $\phi_{\tau \sigma}: X_{\sigma} \cap X_{\tau} \rightarrow G$ denote the regular extension of the function $\rho_{\tau}(t)P(\tau, \sigma) \rho_{\sigma}(t)^{-1}$. Since $\{\phi_{\tau \sigma}\}$ satisfies the cocycle conditions, we can construct a $T$-equivariant principal \(G\)-bundle \(E(\{\rho,P\})\) over \(X\) with $\{\phi_{\tau \sigma}\}$ as transition functions: 
	 \begin{equation}\label{p3}
	 E(\{\rho,P\})=\left(\bigsqcup_{\sigma \in \Xi^*}  X_{\sigma} \times G\right) / \sim,
	 \end{equation} 
	 where \((x,g) \sim (y,h)\) for \((x,g) \in X_{\sigma} \times G\) and \((y,h) \in X_{\tau} \times G\) if and only if 
	 \begin{equation}\label{p4}
	 x=y, ~ x \in X_{\sigma} \cap X_{\tau} ~ \text{ and } ~ h=\phi_{\tau \sigma}(x)g.
	 \end{equation}
	 The action of \(G\) on \(X_{\sigma} \times G\) is given by right multiplication and the action of \(T\) on \(X_{\sigma} \times G\) is given by \(t (x, g)=(tx, \rho_{\sigma}(t)g)\).

	 Let $\{\rho',P'\}=\{\rho'_{\sigma}, P'(\tau, \sigma)\}$ be another admissible collection equivalent to \(\{\rho_{\sigma}, P(\tau, \sigma)\}\), and let \(E(\{\rho',P'\})\) be the principal bundle associated to \(\{\rho'_{\sigma}, P'(\tau, \sigma)\}\). To get a well defined map from the set of equivalence classes of admissible collections to the set of isomorphism classes of $T$-equivariant principal \(G\)-bundles, we need to show that \(E(\{\rho,P\})\) and \(E(\{\rho',P'\})\) are \(T\)-equivariantly isomorphic principal \(G\)-bundles. 
	 Let \(\phi'_{\tau \sigma}:X_{\sigma} \cap X_{\tau} \rightarrow G\) be the algebraic extension of \(\rho'_{\tau}(t)P'(\tau, \sigma) \rho'_{\sigma}(t)^{-1}\). Then for each $\sigma \in \Xi^*$, there exists $\alpha_{\sigma}, \ \beta_{\sigma} \in G$ such that the conditions of Definition \ref{equiv_admissible} hold. Define \(r_{\sigma}: O \rightarrow G\) by 
	 \begin{equation}\label{relndefn}
	 r_{\sigma}(t  x_0)=\rho_{\sigma}(t)\beta_{\sigma} \alpha_{\sigma}^{-1}\rho_{\sigma}(t)^{-1} \alpha_{\sigma} \text{ for  all } t \in T.
	 \end{equation}
Then \(r_{\sigma}(x_0)= \beta_{\sigma}\)	and by Definition \ref{equiv_admissible} $(c)$, \(r_{\sigma}\) extends to a regular function \(X_{\sigma} \rightarrow G\). Now, we also have 
	 \begin{equation*}
	 \begin{split}
	 \phi'_{\tau \sigma}(t  x_0)= & \rho'_{\tau }(t)P'(\tau, \sigma) \rho'_{\sigma}(t)^{-1}\\
	 =&\alpha_{\tau}^{-1} \rho_{\tau}(t) \alpha_{\tau} \, \beta_{\tau}^{-1} P(\tau, \sigma) \beta_{\sigma} \, \alpha_{\sigma}^{-1} \rho_{\sigma}(t)^{-1} \alpha_{\sigma}\\
	 =&  \left( \alpha_{\tau}^{-1} \rho_{\tau}(t) \alpha_{\tau} \beta_{\tau}^{-1} \rho_{\tau}(t)^{-1}\right) \left(  \rho_{\tau}(t) P(\tau, \sigma)\rho_{\sigma}(t)^{-1} \right) \left( \rho_{\sigma}(t)  \beta_{\sigma}  \alpha_{\sigma}^{-1} \rho_{\sigma}(t)^{-1} \alpha_{\sigma}\right) \\
	 =&r_{\tau}(t  x_0)^{-1} \phi_{\tau \sigma}(t  x_0) r_{\sigma}(t  x_0).
	 \end{split}
	 \end{equation*}
	 By density of \(O\) in $X_{\sigma} \cap X_{\tau }$, this implies that 
	 \begin{equation}\label{D5}
	 \phi'_{\tau \sigma}(x)=r_{\tau}(x)^{-1} \phi_{\tau \sigma}(x)\, r_{\sigma}(x)\text{ for all } x \in X_{\sigma} \cap X_{\tau }.
	 \end{equation}
Let us define a map 
\begin{equation*}
\begin{split}
\Phi:E(\{\rho,P\}) & \rightarrow E(\{\rho',P'\})\\
[x,g] & \mapsto [x, r_{\sigma}(x)^{-1}g], \ x \in X_{\sigma}, \ g \in G.
\end{split}
\end{equation*}
\noindent	 
	Note that, if \((x,g) \in X_{\sigma} \times G\) is equivalent to \((y,h) \in X_{\tau} \times G\), then \(x=y\) and 
	\begin{equation}\label{MTeq}
	h=\phi_{\tau \sigma}(x)g.
	\end{equation} Now \(\Phi([y,h])=[y, r_{\tau}(y)^{-1}h]\). To show that \(\Phi\) is well defined we need to show that \((x, r_{\sigma}(x)^{-1}g)\) is equivalent to \((y, r_{\tau}(y)^{-1}h)\), i.e. we need to show that 
	 \begin{equation*}
	 r_{\tau}(x)^{-1}h= \phi'_{\tau \sigma}(x)r_{\sigma}(x)^{-1}g \text{ for  all } x \in X_{\tau} \cap X_{\sigma}.
	 \end{equation*}
	 To see this, note that by \eqref{D5},
	 \begin{equation*}
	 \begin{split}
	 \phi'_{\tau \sigma}(x)r_{\sigma}(x)^{-1}g= &r_{\tau}(x)^{-1} \phi_{\tau \sigma}(x) r_{\sigma}(x)r_{\sigma}(x)^{-1}g\\
	 =&r_{\tau}(x)^{-1} \phi_{\tau \sigma}(x)g\\
	 =&r_{\tau}(x)^{-1}h \ (\text{by \eqref{MTeq}}).
	 \end{split}
	 \end{equation*}
	 Thus \(\Phi\) is well defined. Also it is \(G\)-equivariant. To see that $\Phi$ is \(T\)-equivariant, note that 
	  \[\Phi(t \, [x,g])=\Phi([tx, \rho_{\sigma}(t)g])= [tx, r_{\sigma}(tx)^{-1}\rho_{\sigma}(t)g] \text{ for }  x \in X_{\sigma}, \ t \in T .\] 
	  On the other hand, \[t \, \Phi([x,g])=t \, [x, r_{\sigma}(x)^{-1}g]=[tx, \rho'_{\sigma}(t)r_{\sigma}(x)^{-1}g].\] Now, by \eqref{relndefn}, we see that
	 \begin{equation}\label{theq1}
	 \begin{split}
	 r_{\sigma}(tx_0)^{-1}\rho_{\sigma}(t)g=&\alpha_{\sigma}^{-1}\rho_{\sigma}(t) \alpha_{\sigma}\beta_{\sigma} ^{-1}\rho_{\sigma}(t)^{-1}\rho_{\sigma}(t)g\\
	 =&\rho'_{\sigma}(t)\beta_{\sigma}^{-1}g\\
	 =&\rho'_{\sigma}(t)r_{\sigma}(x_0)^{-1}g.
	 \end{split}
	 \end{equation}
 \noindent
 Moreover, for \(t_1,\ t_2 \in T\) we have,
 \begin{equation}\label{theq2}
 \begin{split}
 r_{\sigma}(t_2\, t_1\,x_0)^{-1}\rho_{\sigma}(t_2)g=& \alpha_{\sigma}^{-1}\rho_{\sigma}(t_2 t_1) \alpha_{\sigma}\beta_{\sigma} ^{-1}\rho_{\sigma}(t_2 t_1)^{-1}\rho_{\sigma}(t_2)g  ~ (\text{by } \eqref{relndefn})\\
  =& \alpha_{\sigma}^{-1}\rho_{\sigma}(t_2) \rho_{\sigma}(t_1)  \alpha_{\sigma}\beta_{\sigma} ^{-1}\rho_{\sigma}(t_1)^{-1}g\\
  =&\alpha_{\sigma}^{-1}\rho_{\sigma}(t_2)  \alpha_{\sigma} (\alpha_{\sigma}^{-1} \rho_{\sigma}(t_1)  \alpha_{\sigma}\beta_{\sigma} ^{-1}\rho_{\sigma}(t_1)^{-1})g\\
  =& \rho_{\sigma}'(t_2)  r_{\sigma}( t_1x_0)^{-1} g ~(\text{by }\eqref{relndefn}).
 \end{split}
 \end{equation}
 Since every point of \(O\) is of the form \(t_1 x_0\) for some \(t_1 \in T\), from \eqref{theq2} we have 
 \[r_{\sigma}(t_2 x)^{-1}\rho_{\sigma}(t_2)g = \rho_{\sigma}'(t_2)  r_{\sigma}( x )^{-1} g, \text{ for  all } x \in O, \; \text{and}\; t_2 \in T.\]  
 \noindent
	 Then, by density of \(O\) it follows that 
	 $$ r_{\sigma}(tx)^{-1}\rho_{\sigma}(t)g=\rho'_{\sigma}(t)r_{\sigma}(x)^{-1}g$$ for all \(x \in X_{\sigma},\) and \(t \in T\). This proves that \(\Phi\) is a \(T\)-equivariant morphism, and hence an isomorphism. Thus, we can define a  map \(\bf{B}\) from the set of equivalence classes of admissible collections to the set of isomorphism classes of $T$-equivariant principal \(G\)-bundles given by \[{\bf B} ([\{\rho_{\sigma}, P(\tau, \sigma)\} ]) = [E(\{\rho,P\}) ].\]
	 Now we check that \({\bf A \circ B}=Id\). Let \(\{\rho, P\}\) be an admissible collection and \(E(\{\rho, P\})\) be the associated principal \(G\)-bundle with transition functions $\{\phi_{\tau \sigma}\}$ (as constructed in \eqref{p3}). For $\sigma \in \Xi^{*}$, note that there is a section 
	 \begin{equation}\label{Rmk_sec}
	 s_{\sigma}: X_{\sigma} \rightarrow E(\{\rho, P\}), \text{ given by } s_{\sigma}(x)=[(x, 1_G)], \ x \in X_{\sigma}, 
	 \end{equation}
	 which clearly satisfies \(t s_{\sigma}(x)=s_{\sigma}(tx) \cdot \rho_{\sigma}(t).\) This shows that \(s_{\sigma}\) is indeed a distinguished section. Finally, note that 
	 \begin{equation*}
	 s_{\sigma}(x)=[(x, 1_G)]=[(x, \phi_{\tau \sigma}(x))]=[(x, 1_G)] \cdot \phi_{\tau \sigma}(x)=s_{\tau}(x) \cdot \phi_{\tau \sigma}(x),
	 \end{equation*}
	 which implies that \(\{\rho, P\}\) is an admissible collection associated to \(E(\{\rho, P\})\). Thus \({\bf A \circ B}=Id\).
	 
	 The equality \({\bf B \circ A}=Id\) follows from the last part of the proof of \cite[Theoem 3.2]{BDP}. To see this, let $\mathcal{E}\stackrel{\pi} \rightarrow X$ be an equivariant principal \(G\)-bundle. Consider a representative \(\{\rho, P\}\) of \(\bf{A}([\mathcal{E}])\).  Suppose \(\{\rho,P\}=\{\rho_{\sigma}, P(\tau, \sigma)\}\) is associated to a collection of distinguished sections \(\{s_{\sigma}\}\). Given \(e \in \mathcal{E}\), suppose \(e \in \mathcal{E}_{\sigma}\) for some $\sigma \in \Xi^{\ast}$. Then \(e=s_{\sigma}(\pi(e)) \cdot g_e^{\sigma}\) for unique \(g_e^{\sigma} \in G\). Then there is an isomorphism $$\Phi: \mathcal{E} \rightarrow E(\{\rho, P\}), \text{ given by } \Phi(e)=[(\pi(e), g_e^{\sigma})].$$
	 Thus we get \({\bf B \circ A}=Id\).
\end{proof}

\begin{rmk}{\rm 
Let \(\{\rho^1_{\sigma}, P^1(\tau, \sigma)\}\) and \(\{\rho^2_{\sigma}, P^2(\tau, \sigma)\}\) be two admissible collections which are equivalent in the sense of Definition \ref{equiv_admissible} with $\alpha_{\sigma},\ \beta_{\sigma} \in G$ . Then in view of the proof of Theorem \ref{classifi}, \(\{\rho^1_{\sigma}, P^1(\tau, \sigma)\}\) and  \(\{\rho^2_{\sigma}, P^2(\tau, \sigma)\}\)  give rise to equivariantly isomorphic principal bundles \(E(\{\rho^1, P^1\})\) and \(E(\{\rho^2, P^2\})\) respectively. Then \(\rho^1_{\sigma}\) (respectively, \(\rho^2_{\sigma}\)) is the homomorphism corresponding to the distinguished section \(s_{\sigma}\) (respectively, \(s'_{\sigma}\)) by \eqref{Rmk_sec}. Hence by Proposition \ref{iso_admi} , we see that $\alpha_{\sigma}$, $\beta_{\sigma}$ additionally satisfy the following
\begin{equation*}
\lim\limits_{t x_0 \rightarrow x_{\sigma}} \rho^1_{\sigma}(t)\beta_{\sigma} \alpha_{\sigma}^{-1}\rho^1_{\sigma}(t)^{-1}=1_G.
\end{equation*} 
}
\end{rmk}

\section{Kaneyama type description}\label{Kan description}

In this section, we obtain classification results for equivariant principal \(G\)-bundles having more 
combinatorial flavour, generalizing the classification of equivariant vector bundles given by Kaneyama 
\cite[Theorem 1.3]{Kan2}. Let us fix an embedding of \(G\) in \({\rm GL}(r, \C)\) such that \(K_0\,:=\,D(r, \C) \cap 
G\) is a maximal torus of $G$, where $D(r, \C)$ is the subgroup of diagonal matrices of \({\rm GL}(r, \C)\). We restrict our attention to a smaller class of distinguished sections.

\begin{defn}
Let $\mc{E}$ be a $T$-equivariant principal $G$-bundle on $X$. For any $\sigma \,\in\, \Xi^{*}$, set $\mathcal{E}_{\sigma}\,:=\,\mc{E}|_{X_{\sigma}}$, and let \(s_{\sigma}\) be a 
distinguished section of \(\mc{E}_{\sigma}\). Let $\rho_{\sigma}\,:=\,\rho_{s_{\sigma}}\,:\, 
T\, \longrightarrow \,G$ be the corresponding homomorphism. The distinguished section \(s_{\sigma}\) is called a Kaneyama section if
the image of $\rho_{\sigma}$ lies in \(K_0\).
\end{defn}

Every $\mathcal{E}_{\sigma}$ admits a Kaneyama section. When \(X\) is nonsingular, a proof of this fact can be found in  \cite[Lemma 4.2]{BDP}. The same proof works in the case of singular toric varieties. For any \(t \in T\) and a Kaneyama section \(s_{\sigma}\), the image $\rho_{\sigma}(t)$ is a diagonal matrix which we
denote by 
\begin{equation}\label{K1}
\xi^{\sigma}(t)= \text{diag}(\xi^{\sigma}_1(t), \ldots, \xi^{\sigma}_r(t)).
\end{equation}
Note that \(\xi^{\sigma}_i\) are characters of the torus \(T\). 

\begin{lemma} For any $\gamma \in \Xi(1)$, let $v_{\gamma}$ denote the primitive integral generator of $\gamma$.
	Consider $\sigma, \tau \in \Xi$. Then, there exists a permutation \(\epsilon \in S_r\)  of \(\{1, \ldots, r\}\) such that \[\langle \xi^{\sigma}_i, v_{\gamma} \rangle= \langle \xi^{\tau}_{\epsilon(i)}, v_{\gamma} \rangle \text{ for all } \gamma \in (\sigma \cap \tau)(1).\]
\end{lemma}

\begin{proof}
	The statement has been proved in the nonsingular case (see \cite[Lemma 4.4]{BDP}\label{K2}). The same proof works in general.
\end{proof}
Write \[\xi^\sigma:=(\xi^{\sigma}_1, \ldots, \xi^{\sigma}_r) \in M^{\oplus r}.\]
Recall that $\rho_{\sigma}(t)$, and therefore $\xi^{\sigma}(t)$ must factor through the projection \(\pi_{\sigma}: T \rightarrow T_{\sigma}\) (see \eqref{projection}). Associated to $\pi_{\sigma}$, there is a map \(\pi_{\sigma}^*: M_{\sigma } \rightarrow M\) from the character group \(M_{\sigma }:=M/(\sigma ^{\perp} \cap M)\) of \(T_{\sigma}\) to
that of \(T\). Then $\xi^{\sigma}$ lies in the image of \(M_{\sigma}^{\oplus r}\) in \(M^{\oplus r}\) under the map induced by
\(\pi_{\sigma}^*\). Note that this image is \(M^{\oplus r}\) if $\sigma$ is of top dimension, as \(T_{\sigma}=T\) in this
case.

\begin{defn}\label{K3}
	 Consider the following abstract data.
\begin{enumerate}
	\item A map \[\xi: \Xi^{*} \rightarrow \left( \pi_{\sigma}^*(M_{\sigma})\right)^{\oplus r},\; \sigma \mapsto \xi^{\sigma}=(\xi^{\sigma}_1, \ldots, \xi^{\sigma}_r) \]
	such that for every pair of maximal cones $\sigma$ and $\tau$ there exists a permutation \(\epsilon \in S_r\) satisfying \[\langle \xi^{\sigma}_i, v_{\gamma} \rangle= \langle \xi^{\tau}_{\epsilon(i)}, v_{\gamma} \rangle \text{ for all } i \text{ and } \gamma \in (\sigma \cap \tau)(1).\]
	
	\item A map \[P: \Xi^{*} \times \Xi^{*} \rightarrow G \subseteq {\rm GL}(r, \C),\; (\tau, \sigma) \mapsto P(\tau, \sigma)\] such that \(P(\tau, \sigma)_{ij} \neq 0\) only if \(\langle \xi^{\tau}_i, v_{\gamma} \rangle \geq \langle \xi^{\sigma}_j, v_{\gamma} \rangle \text{ for all } \gamma \in (\sigma \cap \tau)(1), \) and
	which satisfies the cocycle conditions: \[P(\sigma, \sigma)={\rm Id}_r, ~ P(\tau, \sigma) P(\sigma, \delta) P(\delta , \tau)={\rm Id}_r \text{ for every } \sigma, \tau, \delta \in \Xi^{*}.\]
	
	\item Write \((\xi, P)\) for a collection \( \{\xi^{\sigma}, P(\tau, \sigma)\}\), where $\sigma$, $\tau$ vary over $\Xi^{*}$. Two pairs \((\xi, P)\) and \((\xi', P')\) are said to be equivalent if there exists a permutation \(\eta=\eta(\sigma) \in S_r\), depending on $\sigma$, such that \[(\xi^{\sigma}_1, \ldots, \xi^{\sigma}_r)=(\xi'^{\sigma}_{\eta(1)}, \ldots, \xi'^{\sigma}_{\eta(r)}) \text{ for every } \sigma \in \Xi^{*},\]  
	and if there exists a map \[\beta: \Xi^{*} \rightarrow G, \; \sigma \mapsto \beta_{\sigma}\] such that  \((\beta_{\sigma})_{ij} \neq 0\) only if \(\langle \xi^{\sigma}_i, v_{\gamma} \rangle \geq \langle \xi^{\sigma}_{\eta^{-1}(j)}, v_{\gamma} \rangle \text{ for all } \gamma \in \sigma(1), \) and  \[P'(\tau, \sigma)=\beta_{\tau}^{-1} P(\tau, \sigma) \beta_{\sigma} \text{ for every } \tau, \sigma \in \Xi^{*}.\]
\end{enumerate}
\end{defn}

We have the following combinatorial classification of $T$-equivariant principal \(G\)-bundles over toric varieties. 
\begin{thm}\label{K4}
Let \(X\) be a  toric variety defined by a fan $\Xi$. The set of isomorphism classes of $T$-equivariant principal $G$-bundles on $X$ is in bijective correspondence with the set of data \((1)\) and \((2)\) up to the equivalence \((3)\).
\end{thm}

\begin{proof}
	Given a principal \(G\)-bundle $\mathcal{E}$, choose a collection \(\{s_{\sigma}\}\) of Kaneyama sections  to get the data \((\xi, P)\) which arises from an admissible collection \(\{\rho_{\sigma}, P(\tau, \sigma)\}\) by Theorem \ref{classifi} and \eqref{K1}. Due to Lemma \ref{K2}, the data \((\xi, P)\) satisfies the condition \((1)\) of Definition \ref{K3}. To see condition \((2)\) holds, recall from Definition \ref{admissible} \((ii)\) and Theorem \ref{classifi} that  \(\xi^{\tau}(t)P(\tau, \sigma) \xi^{\sigma}(t)^{-1}\) extends to a \(G\)-valued function over \(X_{\sigma} \cap X_{\tau} \) for every pair \((\tau, \sigma)\) of cones in $\Xi^{*}$. Now, observe that 
	\begin{align*}
	\xi^{\tau}(t)P(\tau, \sigma) \xi^{\sigma}(t)^{-1} & =\text{diag}(\xi^{\tau}_1(t), \ldots, \xi^{\tau}_r(t)) \left( P(\tau, \sigma)_{ij}\right) \text{diag}(\xi^{\sigma}_1(t)^{-1}, \ldots, \xi^{\sigma}_r(t)^{-1} ) \\
	&=\left( P(\tau, \sigma)_{ij}\, \xi^{\tau}_i(t) \, \xi^{\sigma}_j(t)^{-1} \right). 
	\end{align*}
	Hence, \(P(\tau, \sigma)_{ij}\, \xi^{\tau}_i(t)\, \xi^{\sigma}_j(t)^{-1} \) are regular functions on \(X_{\sigma} \cap X_{\tau} \). This is  equivalent to the following:
	\begin{center}
		\(P(\tau, \sigma)_{ij} \neq 0\) only if \(\langle \xi^{\tau}_i, v_{\gamma} \rangle \geq \langle \xi^{\sigma}_j, v_{\gamma} \rangle \text{ for all } \gamma \in (\sigma \cap \tau)(1). \)
	\end{center} 
	Also note that \(P(\tau, \sigma)\) satisfies the cocycle conditions.
	
Let \(\{s'_{\sigma}\}\)  be another collection of Kaneyama sections of $\mathcal{E}$  giving an admissible collection \(\{\rho'_{\sigma}, P'(\tau, \sigma)\}\) and corresponding
data \((\xi', P')\).    
Then \(\{\rho'_{\sigma}, P'(\tau, \sigma)\}\) is in the equivalence class of \(\{\rho_{\sigma}, P(\tau, \sigma)\}\).
 Note that \(\{\xi^{\sigma}_1, \ldots, \xi^{\sigma}_r\}\) is the set of characters of the representation \(\rho_{\sigma} : T \rightarrow G \subseteq {\rm GL}(r, \C)\). Thus, it is invariant under conjugation. Hence, from Definition \ref{equiv_admissible} \((a)\) and Theorem \ref{classifi} we get a permutation \(\eta=\eta(\sigma) \in S_r\) such that \[(\xi^{\sigma}_1, \ldots, \xi^{\sigma}_r)=(\xi'^{\sigma}_{\eta(1)}, \ldots, \xi'^{\sigma}_{\eta(r)}) \text{ for every } \sigma \in \Xi^{*}.\] 
Again from Definition \ref{equiv_admissible} \((b)\) and Theorem \ref{classifi}, we can define a map \[\beta: \Xi^{*} \rightarrow G, \ \sigma \mapsto \beta_{\sigma}\]  such that \[P'(\tau, \sigma)=\beta_{\tau}^{-1} P(\tau, \sigma) \beta_{\sigma} \text{ for every } \tau, \sigma \in \Xi^{*}.\] Finally, note that 
$$ \rho_{\sigma}(t) \beta_{\sigma} \rho_{\sigma}'(t)^{-1} = (\rho_{\sigma}(t) \beta_{\sigma}
\alpha_{\sigma}^{-1} \rho_{\sigma}(t)^{-1})\, \alpha_{\sigma} \,.   $$
 Therefore, using Definition \ref{equiv_admissible} \((c)\) and Theorem \ref{classifi} we get that \(\rho_{\sigma}(t) \beta_{\sigma} \rho_{\sigma}'(t)^{-1}\) extends to a \(G\)-valued function over \(X_{\sigma}\) which is equivalent to the following:
  \[(\beta_{\sigma})_{ij} \neq 0\; \text{  only if}\; \langle \xi^{\sigma}_i, v_{\gamma} \rangle \geq \langle \xi^{\sigma}_{\eta^{-1}(j)}, v_{\gamma} \rangle \; \text{ for all } \;\gamma \in (\sigma )(1). \]
  Hence, we conclude that the data $(\xi, P)$ and $(\xi', P')$ are equivalent. 
  
  Conversely, given a set of data \((\xi, P)\) satisfying the conditions of Definition \ref{K3}, we regard $\xi^{\sigma}$ as $\rho_{\sigma}$ and \(P(\tau, \sigma)\) as the transition function at the point \(x_0 \in O\) to construct a principal \(G\)-bundle \(E(\{\rho, P\})\) as in Theorem \ref{classifi}. Equivalence of \((\xi, P)\) and \((\xi', P')\) imply the equivalence of 
  $\{\rho_{\sigma}, P(\tau, \sigma)\}$ and  $\{\rho_{\sigma}', P'(\tau, \sigma)\}$.
    Then the theorem follows from Theorem \ref{classifi}.
\end{proof}

Assume further that \(X\) is nonsingular and complete. Given $\xi$ as above, for every $\sigma \in \Xi^{*}\,=\, \Xi(n)$,
where \(n=\text{dim}(X)\), we have a map \[m_{\sigma}\,:\, \sigma(1)\, \longrightarrow\, \Z^r 
\]
such that $m_{\sigma}({\gamma})_i\,=\,\langle \xi^{\sigma}_i, v_{\gamma} \rangle$ for all
$\gamma \,\in\, \sigma(1)$.
Moreover, as $\sigma$ is nonsingular, every map \(m_{\sigma}\,:\, \sigma(1) \,\longrightarrow\, \Z^r\) corresponds to a $\xi^{\sigma}$ by the nondegeneracy of the pairing \(\langle ~ , ~ \rangle\). A strategy of
Kaneyama \cite{Kan, Kan2} is to reformulate the equivalence \((3)\), Definition \ref{K3}, in terms of \(m_{\sigma}\)'s. 

\begin{defn}\label{K6}
Consider the following abstract data.
\begin{enumerate}
	\item[$(1')$]  A map \[m \,:\, \Xi(1) \,\longrightarrow\, \Z^r,\,\ \gamma \,\longmapsto\,
m(\gamma)\,=\,(m(\gamma)_1, \ldots, m(\gamma)_r) \] such that for every \(\sigma \,\in\, \Xi(n)\) there is another
map \(m_{\sigma} \,:\, \sigma(1) \,\longrightarrow\, \Z^r\) so that there is a permutation $\epsilon \in S_r$ satisfying 
	\begin{align*}
	m_{\sigma}(\gamma)=(m_{\sigma}(\gamma)_1, \ldots, m_{\sigma}(\gamma)_r)=(m(\gamma)_{\epsilon(1)}, \ldots, m(\gamma)_{\epsilon(r)})
	\end{align*}
	for all $\gamma \in \sigma(1)$.

	\item[$(2')$] A map \[P\,\,:\,\, \Xi(n) \times \Xi(n)\,\longrightarrow\, G \,
\subseteq\, {\rm GL}(r, \C),\,\ (\tau,\, \sigma)\,\longmapsto\, P(\tau,\, \sigma)\] such that \(P(\tau, \sigma)_{ij} \neq 0\) only if \(m_{\tau}({\gamma})_i  \geq m_{\sigma}({\gamma})_j \text{ for all } \gamma \in (\sigma \cap \tau)(1), \) and which satisfies the following cocycle conditions: 
	 \[P(\sigma, \sigma)={\rm Id}_r, ~ P(\tau, \sigma) P(\sigma, \delta) P(\delta , \tau)={\rm Id}_r \text{ for every } \sigma, \tau, \delta \in \Xi(n).\]
	
	\item[$(3')$]  Two pairs \((m,\, P)\) and \((m', \,P')\) are said to be equivalent if there exists a
permutation \(\eta\,=\,\eta(\sigma) \,\in\, S_r\), that depends on $\sigma$, such that \[(m_{\sigma}({\gamma})_1, \ldots, m_{\sigma}({\gamma})_r)=(m'_{\sigma}({\gamma})_{\eta(1)}, \ldots, m'_{\sigma}({\gamma})_{\eta(r)}) \text{ for every } \sigma \in \Xi(n),\]  
	and if there exists a map \[\beta: \Xi(n) \rightarrow G,\, \sigma \mapsto \beta_{\sigma}\] such that  \((\beta_{\sigma})_{ij} \neq 0\) only if \( m_{\sigma}({\gamma})_i  \geq  m_{\sigma}({\gamma})_{\eta^{-1}(j)} \text{ for all } \gamma \in \sigma(1), \) and  \[P'(\tau, \sigma)=\beta_{\tau}^{-1} P(\tau, \sigma) \beta_{\sigma} \text{ for every } \tau, \sigma \in \Xi(n).\]
\end{enumerate}
\end{defn} 

For a nonsingular and complete toric variety, Theorem \ref{K4} can be reformulated as follows.
\begin{thm}\label{K5}
	Let \(X\) be a nonsingular complete toric variety defined by a fan $\Xi$ of dimension $n$. The set of isomorphism classes of $T$-equivariant principal $G$-bundles on $X$ is in bijective correspondence with the set of data \((1')\) and \((2')\) up to the equivalence \((3')\).
\end{thm}

\begin{ex}{\rm  (cf. \cite[Remark 1.4]{Kan2})
\noindent
Let \(D_{\gamma}\) be the invariant prime divisor corresponding to the ray $\gamma \in \Xi(1)$. Put $a_{\gamma}=m(\gamma)$ where \(m\) is defined by Definition \ref{K6} \((1')\) in the case \(r=1\). Let \(G=\C^*\) and  \(P(\tau, \sigma)=1\) for every $\tau$ and $\sigma$ in $\Xi(n)$. Then the associated principal bundle \(E(\{m, P\})\) is the frame bundle associated to the line bundle \(\mathcal{O}_X\left( -\sum_{\gamma \in \Xi(1)} a_{\gamma} D_{\gamma}\right) \). 
}	
\end{ex}

\section{Morphisms of equivariant principal bundles}

In this section, we give a combinatorial description of equivariant morphisms between equivariant principal bundles. Let $X$ be a complex toric variety. Let us recall that a morphism between two principal \(G\)-bundles $\mathcal{E}$ and $\mathcal{E'}$ on \(X\) is a morphism of varieties $\Phi : \mathcal{E} \rightarrow \mathcal{E'}$ such that \(\Phi\) commutes with the bundle projections and \(\Phi\) is  \(G\)-equivariant. An equivariant morphism between two \(T\)-equivariant principal \(G\)-bundles $\mathcal{E}$ and $\mathcal{E'}$ on a toric variety \(X\) is a morphism $\Phi : \mathcal{E} \rightarrow \mathcal{E'}$ of principal \(G\)-bundles which is also \(T\)-equivariant. We note here that any morphism of principal bundles is in fact an isomorphism. 
\begin{prop}\label{automorphism}
	Let $\mathcal{E}$ be a \(T\)-equivariant principal \(G\)-bundle over \(X\). Let \(e \in \mathcal{E}_{x_0} \), the fiber of $\mathcal{E}$ at \(x_0\), and for each $\sigma \in \Xi^*$, take a distinguished section \(s_{\sigma}: X_{\sigma} \rightarrow \mathcal{E} \) such that \(s_{\sigma}(x_0)=e\). Then \(\text{Aut}_T(\mathcal{E})\), the group of \(T\)-equivariant automorphisms of $\mathcal{E}$, is given by
	\begin{align*}
	\text{Aut}_T(\mathcal{E}) &=\{g \in G ~|~ \text{for all } \sigma \in \Xi^*, ~ \rho_{\sigma}(t) g \rho_{\sigma}(t)^{-1} \text{ extends to $X_{\sigma}$ } \}
	\end{align*}
\end{prop}

\begin{proof} Note that,
	 by Lemma \ref{dsclass} or \cite[Lemma 2.6]{BDP}, if  $s_{\sigma}$  is a distinguished section of $\mathcal{E}$ over $X_{\sigma}$, so is $s_{\sigma}\cdot g$ for any $g \in G$. Thus, given $e \in  \mathcal{E}_{x_0} $, we can (and do) choose  a distinguished section \(s_{\sigma}\)  for each $\sigma \in \Xi^*$, such that \(s_{\sigma}(x_0)=e\).

	Consider the map  $\xi : \text{Aut}_T(\mathcal{E}) \rightarrow G$, uniquely determined by the relation 
	\begin{align*}
	\Phi(e)=e \cdot \xi(\Phi),\; {\rm for \; all\; } \Phi \in \text{Aut}_T(\mathcal{E}).
	\end{align*}
	For any \( \Phi \in \text{Aut}_T(\mathcal{E}) \), by  continuity of \(\Phi\) and  density of \(O\) in $X$, $\Phi$ is determined by  \(\Phi|_{\mathcal{E}_O}\). Here \(\mathcal{E}_O\) denotes the restriction of $\mathcal{E}$ to the open orbit \(O\). By \(T\)-equivariance, \(\Phi |_{\mathcal{E}_O}\) is determined by \(\Phi |_{\mathcal{E}_{x_0}}\), whereas \(\Phi|_{\mathcal{E}_{x_0}}\) is determined by $\Phi(e)$ using \(G\)-equivariance. This shows that the map $\xi$ is injective.
	
	Let $\Phi_1, \Phi_2 \in \text{Aut}_T(\mathcal{E})$. Then, 
	\begin{align*}
	(\Phi_2 \circ \Phi_1)(e)=\Phi_2(e \cdot \xi(\Phi_1))=\Phi_2(e) \cdot \xi(\Phi_1)=e \cdot \xi(\Phi_2) \xi(\Phi_1).
	\end{align*}
	Thus, $\xi(\Phi_2 \circ \Phi_1)=\xi(\Phi_2) \xi(\Phi_1)$, which shows that $\xi$ is a group homomorphism.
	
	Let $\Phi \in \text{Aut}_T(\mathcal{E})$. For any $\sigma \in \Xi^*$, define a map  \(\varphi_{\sigma}:X_{\sigma} \rightarrow G \) by \[\Phi(s_{\sigma}(x))=s_{\sigma}(x) \cdot \varphi_{\sigma}(x) \; {\rm for\; all\; } x \in X_{\sigma}.\]
	Then, by the definition of $\xi$, we have \(\varphi_{\sigma}(x_0)=\xi(\Phi)\). Observe that
	\begin{align*}
	&t \,\Phi(s_{\sigma}(x)) = t s_{\sigma}(x) \cdot \varphi_{\sigma}(x)\\
	&\Rightarrow \Phi(t s_{\sigma}(x))= t s_{\sigma}(x) \cdot \varphi_{\sigma}(x)   \\
	&\Rightarrow \Phi(s_{\sigma}(tx) \cdot \rho_{\sigma}(t))=s_{\sigma}(tx) \cdot \rho_{\sigma}(t) \, \varphi_{\sigma}(x)\\
	&\Rightarrow \Phi(s_{\sigma}(tx) ) \cdot \rho_{\sigma}(t)=s_{\sigma}(tx) \cdot \rho_{\sigma}(t)\,  \varphi_{\sigma}(x)\\
	&\Rightarrow s_{\sigma}(tx) \cdot \varphi_{\sigma}(tx)\, \rho_{\sigma}(t) = s_{\sigma}(tx) \cdot \rho_{\sigma}(t) \, \varphi_{\sigma}(x)\\
	&\Rightarrow \varphi_{\sigma}(tx) =\rho_{\sigma}(t) \, \varphi_{\sigma}(x)\, \rho_{\sigma}(t)^{-1}.
	\end{align*}
	For \(x=x_0\), we have
	 $$\varphi_{\sigma}(tx_0) =\rho_{\sigma}(t)  \varphi_{\sigma}(x_0) \rho_{\sigma}(t)^{-1}=\rho_{\sigma}(t) \xi(\Phi) \rho_{\sigma}(t)^{-1}\,.$$  Thus, \(\rho_{\sigma}(t) \xi(\Phi) \rho_{\sigma}(t)^{-1}\) extends to \(X_{\sigma}\).

	Conversely, let \(g \in G \) be such that $\rho_{\sigma}(t) g \rho_{\sigma}(t)^{-1} \text{ extends to $X_{\sigma}$ for all } \sigma \in \Xi^*$. Denote the said extension of  $\rho_{\sigma}(t) g \rho_{\sigma}(t)^{-1}$ by	$\varphi_{\sigma}: X_{\sigma} \to G$.
	First we define $\Phi \in \text{Aut}_T(\mathcal{E}|_O)$ by setting $\Phi(e)=e \cdot g$. Then $\Phi$ determines a \(T\)-equivariant morphism (hence an automorphism) of $\mathcal{E}|_O$ as follows:
	\[\Phi(te \cdot h)=te \cdot gh \text{ for all } t \in T \text{ and } h \in G.\]
	It remains to check that $\Phi$ extends to give a \(T\)-equivariant automorphism of $\mathcal{E}$.
	
	Note that, for $\sigma \in \Xi^*$, we have
	 $$\Phi(te)=t e \cdot g=t s_{\sigma}(x_0) \cdot g=s_{\sigma}(tx_0) \cdot \rho_{\sigma}(t) g.$$
	On the other hand, $$\Phi(te)=\Phi(t s_{\sigma}(x_0))=\Phi(s_{\sigma}(tx_0)  \cdot \rho_{\sigma}(t))=\Phi(s_{\sigma}(tx_0)) \cdot \rho_{\sigma}(t).$$ 
	Comparing the right hand sides of the last two equations, we have
	$$\Phi(s_{\sigma}(t x_0))=s_{\sigma}(t x_0)\, \cdot \rho_{\sigma}(t)\, g \, \rho_{\sigma}(t)^{-1}\,.$$
	 Since \(s_{\sigma}\) is regular on \(X_{\sigma}\) and \(\rho_{\sigma}(t) g \rho_{\sigma}(t)^{-1}\) extends to  \(\varphi_{\sigma}\) on  \(X_{\sigma} \), we have an extension of $\Phi$ to $X_{\sigma}$, for every $\sigma$, given by $$\Phi(s_{\sigma}(x)) := s_{\sigma}(x) \cdot \varphi_{\sigma}(x) \; $$ \(\text{ for all }  x \in X_{\sigma}.\)
	 Finally, these maps $\Phi|_{\mathcal{E}_{X_{\sigma}}}$ glue together to give a morphism of varieties $\widetilde{\Phi} : \mathcal{E} \rightarrow \mathcal{E}$ as the dense orbit \(O\) is contained in \(X_{\sigma}\) for all $\sigma \in \Xi^*$.  For each \(t \in T\), consider the subset
	 \[Y_t=\{y \in \mathcal{E} ~ |~ \widetilde{\Phi}(t y )=t \, \widetilde{\Phi}(y)\} \subseteq \mathcal{E}.\]
	 Clearly, \(Y_t\) is a closed subset of $\mathcal{E}$ containing \(\mathcal{E}|_O\). Thus \(Y=\bigcap\limits_{t \in T} Y_t\) is also closed and contains \(\mathcal{E}|_O\). By density of \(\mathcal{E}|_O\), we see that \(Y=\mathcal{E}\). So $\widetilde{\Phi}$ is \(T\)-equivariant. Similar arguments will show that $\widetilde{\Phi}$ commutes with the bundle projection  \(\mathcal{E} \rightarrow X\) and that $\widetilde{\Phi}$ is \(G\)-equivariant. Thus, we have $\widetilde{\Phi} \in \text{Aut}_T(\mathcal{E})$. 
		\end{proof}

  Some remarks are in order. First, the above description of \(\text{Aut}_T(\mathcal{E})\) depends on the choice of \(e \in \mathcal{E}_{x_0}\) up to conjugation by an element of $G$.  Consider an element \(e' \in \mathcal{E}_{x_0} \), possibly different from \(e\). Let \(h \in G\) be such that \(e'=e \cdot h\).  
  	Consider the map  $\xi' : \text{Aut}_T(\mathcal{E}) \rightarrow G$,  determined by the relation 
  	\begin{align*}
  	\Phi(e')=e' \cdot \xi'(\Phi).
  	\end{align*}
  	Then, we see that $\xi'(\Phi)=h^{-1} \xi(\Phi) h$. So, \(\text{Im }\xi'= h^{-1} \text{Im } \xi \ h \).
  
	Secondly, the above description of \(\text{Aut}_T(\mathcal{E})\) is independent of the choice of the distinguished sections $\{ s_{\sigma}\}$.	 For $\sigma \in \Xi^*$, consider another distinguished section \(s_{\sigma}': X_{\sigma} \rightarrow \mathcal{E}\) such that \(s'_{\sigma}(x_0)=e\). Then there exists $\alpha_{\sigma} \in G$ such that \(s'_{\sigma}(x_{\sigma})=s_{\sigma}(x_{\sigma}) \cdot \alpha_{\sigma}\). By Lemma \ref{distconj},  the homomorphism $\rho_{\sigma}'$ corresponding to the distinguished section \(s'_{\sigma}\) is given by 
	\begin{equation}\label{Rmk3.8_1}
	\rho_{\sigma}'(t)= \alpha_{\sigma}^{-1}  \rho_{\sigma}(t) \alpha_{\sigma} ~ \text{ for all } t \in T \,.
	\end{equation}
	We will show that for \(g \in G\),
	\begin{equation}\label{Rmk3.8_4}
	\rho_{\sigma}(t)g \rho_{\sigma}(t)^{-1} \text{ extends to } X_{\sigma} \text{ if and only if } \rho'_{\sigma}(t)g \rho'_{\sigma}(t)^{-1} \text{ extends to } X_{\sigma}.
	\end{equation} 
	Note that 
	\begin{equation}\label{Rmk3.8_2}
	\begin{split}
	\rho'_{\sigma}(t)g \rho'_{\sigma}(t)^{-1}= & \alpha_{\sigma}^{-1}  \rho_{\sigma}(t) \alpha_{\sigma} \,g\, \alpha_{\sigma}^{-1}  \rho_{\sigma}(t)^{-1} \alpha_{\sigma}  \\
	=& \alpha_{\sigma}^{-1} \,( \rho_{\sigma}(t) \alpha_{\sigma} \rho_{\sigma}(t)^{-1})\, (\rho_{\sigma}(t) g \rho_{\sigma}(t)^{-1})\, (\rho_{\sigma}(t) \alpha_{\sigma}^{-1} \rho_{\sigma}(t)^{-1})\, \alpha_{\sigma}.
	\end{split}
	\end{equation}
	For $\sigma \in \Xi^*$, define \(\varphi_{\sigma}: X_{\sigma} \rightarrow G\) by the relation \[s'_{\sigma}(x)=s_{\sigma}(x) \cdot \varphi_{\sigma}(x)  \text{ for all } x \in X_{\sigma}.\]
	Then we have 
	\begin{equation*}
	\begin{split}
	& t s'_{\sigma}(x)=t s_{\sigma}(x) \cdot \varphi_{\sigma}(x)\\
	& \Rightarrow s'_{\sigma}(tx) \cdot \rho_{\sigma}'(t)=s_{\sigma}(tx) \cdot \rho_{\sigma}(t)  \varphi_{\sigma}(x)\\
	& \Rightarrow s_{\sigma}(tx) \cdot \varphi_{\sigma}(tx) \rho_{\sigma}'(t)=s_{\sigma}(tx) \cdot \rho_{\sigma}(t)  \varphi_{\sigma}(x)\\
	& \Rightarrow \varphi_{\sigma}(tx)=\rho_{\sigma}(t) \varphi_{\sigma}(x) \rho_{\sigma}'(t)^{-1} \\
	& \Rightarrow \varphi_{\sigma}(tx)=\rho_{\sigma}(t) \varphi_{\sigma}(x) \alpha_{\sigma}^{-1}\rho_{\sigma}(t)^{-1} \alpha_{\sigma}.
	\end{split}
	\end{equation*}
	\noindent	
	As \( \varphi_{\sigma}(x_0)=1_G \), putting \(x=x_0\), we have \[\varphi_{\sigma}(tx_0)=\rho_{\sigma}(t) \alpha_{\sigma}^{-1}\rho_{\sigma}(t)^{-1} \alpha_{\sigma} .\] 
	Thus, \(\rho_{\sigma}(t) \alpha_{\sigma}^{-1}\rho_{\sigma}(t)^{-1}\) extends to \(X_{\sigma}\). A similar calculation  shows that \(\rho_{\sigma}(t) \alpha_{\sigma}\rho_{\sigma}(t)^{-1}\) also extends to \(X_{\sigma}\). Now from \eqref{Rmk3.8_2}, we see that if \(\rho_{\sigma}(t)g \rho_{\sigma}(t)^{-1}\) extends to \(X_{\sigma}\) for \(g \in G\), then \(\rho'_{\sigma}(t)g \rho'_{\sigma}(t)^{-1}\) also extends to \(X_{\sigma}\). Finally note that using \eqref{Rmk3.8_1}, similar arguments as above give the reverse implication of \eqref{Rmk3.8_4}.

	Now, we wish to develop a more precise characterisation of $\text{Aut}_T(\mathcal{E})$ when the group $G$ is connected and reductive.
	 Given a $1$-psg $\mu$ of $G$, the parabolic subgroup $P(\mu)$ of $G$ associated to $\mu$ is given  by (cf. \cite[Section 8.4]{Springer})
	  $$P(\mu) := \{g \in G \mid \lim_{z\to 0}\,  \mu(z) g \mu(z)^{-1} \; {\rm exists} \}\,. $$  

\begin{lemma}\label{7}
	Assume \(G\) to be a connected reductive linear algebraic group. Let \(\{s_{\sigma}\}_{\sigma \in \Xi^{\ast} }\) be a collection of distinguished sections of  a \(T\)-equivariant principal \(G\)-bundle $\mathcal{E}$ over \(X\) such that \(s_{\sigma}(x_0)=s_{\tau}(x_0)\) for all \(\tau, \sigma \in \Xi^*\). Then, for any ray \(\alpha \in (\sigma \cap \tau)(1)\), we have \[P(\rho_{\sigma} \circ \lambda^{v_{\alpha}}) =P(\rho_{\tau} \circ \lambda^{v_{\alpha}}).\] $($Recall that for any $\alpha \in \Sigma(1)$, $v_{\alpha} \in N$ denotes the primitive integral generator of $\alpha$ and $\lambda^{v_{\alpha}}: \mathbb{C}^{\ast} \to T $ denotes the corresponding 1-psg of $T$.$)$
\end{lemma}

\begin{proof}	Recall the transition function \(\phi_{\tau \sigma}: X_{\sigma} \cap X_{\tau} \rightarrow G\) 
	\eqref{transition},    that  satisfies 
	$$
	s_{\sigma}(x)= s_{\tau}(x) \cdot \phi_{\tau \sigma}(x), \, \text{ for all }\, x \in X_{\sigma} \cap X_{\tau}\, .
	$$
	By hypothesis, \(\phi_{\tau \sigma}(x_0)=1_G\). So, by \eqref{transprop},  we have 
	\begin{equation}\label{transition1}
	\phi_{\tau \sigma}(t  x_0)=\rho_{\tau}(t)\rho_{\sigma}(t)^{-1}.
	\end{equation}
	Let $\gamma$ be any subcone of $\sigma \cap \tau$. Then, 
	\[t  s_{\sigma}(x_{\gamma}) =t s_{\tau}(x_{\gamma}) \cdot \phi_{\tau \sigma}(x_{\gamma})=s_{\tau}(x_{\gamma})  \cdot \rho_{\tau }(t) \phi_{\tau \sigma}(x_{\gamma}).\]
	\noindent
	On the other hand, \[t  s_{\sigma}(x_{\gamma}) =s_{\sigma}(x_{\gamma}) \cdot \rho_{\sigma}(t)= s_{\tau}(x_{\gamma}) \cdot \phi_{\tau \sigma}(x_{\gamma}) \rho_{\sigma}(t).\] 
	So, we obtain, 
	\begin{equation}\label{1.1}
	\rho_{\sigma}(t)=\phi_{\tau \sigma}(x_{\gamma})^{-1}  \rho_{\tau }(t) \phi_{\tau \sigma}(x_{\gamma}).
	\end{equation}
	In particular, for \(\alpha \in (\sigma \cap \tau)(1)\), putting \(t=\lambda^{v_{\alpha}}(z)\) where \(z \in \C^*\) and $\gamma=\alpha$ in \eqref{1.1}, we have 
	\begin{equation}\label{action}
	(	\rho_{\sigma} \circ \lambda^{v_{\alpha}})(z) = \phi_{\tau \sigma}(x_{\alpha})^{-1} (\rho_{\tau } \circ \lambda^{v_{\alpha}})(z) \phi_{\tau \sigma}(x_{\alpha}).
	\end{equation} 
	\noindent
	Using \eqref{transition1}, we have \begin{equation}\label{eqtemp}
	 \phi_{\tau\sigma} (x_{\alpha}) =  \lim_{z \to 0} \phi_{ \tau \sigma} (\lambda^{v_{\alpha}}(z) x_0) =
	\lim_{z \to 0}\,
	(\rho_{\tau} \circ \lambda^{v_{\alpha}}) (z) \, (\rho_{\sigma}\circ \lambda^{v_{\alpha}}) (z) ^{-1}\, .  
	\end{equation}
	Now, applying \eqref{action} to the right hand side of \eqref{eqtemp}, we have
	$$ \phi_{\tau \sigma}(x_{\alpha})= \lim\limits_{z \rightarrow 0} \, (\rho_{\tau } \circ \lambda^{v_{\alpha}})(z) \, \phi_{\tau \sigma}(x_{\alpha})^{-1}( \rho_{\tau } \circ \lambda^{v_{\alpha}})(z)^{-1} \phi_{\tau \sigma}(x_{\alpha}) $$
	which implies that 
	\begin{equation}\label{unipotent}
	\lim\limits_{z \rightarrow 0} \, (\rho_{\tau } \circ \lambda^{v_{\alpha}})(z)\, \phi_{\tau \sigma}(x_{\alpha})^{-1} (\rho_{\tau } \circ \lambda^{v_{\alpha}})(z)^{-1} =1_G \,.
	\end{equation}
	Therefore, as $G$ is reductive, 
	\(\phi_{\tau \sigma}(x_{\alpha}) \in P(\rho_{\tau} \circ \lambda^{v_{\alpha}})\).
	 Thus, from \eqref{action}, we see that 
	\[P(\rho_{\sigma} \circ \lambda^{v_{\alpha}}) =\phi_{\tau \sigma}(x_{\alpha})^{-1} P(\rho_{\tau} \circ \lambda^{v_{\alpha}}) \phi_{\tau \sigma}(x_{\alpha})=P(\rho_{\tau} \circ \lambda^{v_{\alpha}}).\]
	\end{proof}

\begin{rmk}{\rm
	From \eqref{unipotent}, we in fact get that \(\phi_{\tau \sigma}(x_{\alpha}) \in R_uP(\rho_{\tau} \circ \lambda^{v_{\alpha}})\), the unipotent radical of the parabolic subgroup \(P(\rho_{\tau} \circ \lambda^{v_{\alpha}})\) (see \cite[Exercise 8.4.6(5)]{Springer}).}
\end{rmk}

\begin{defn}\label{Palpha}
	Let $\alpha \in \Xi(1)$. We set \(P^{\alpha}:=P(\rho_{\sigma} \circ \lambda^{v_{\alpha}}) \), the parabolic subgroup of \(G\) corresponding to the  $1$-psg \(\rho_{\sigma} \circ \lambda^{v_{\alpha}}\) of $G$.
\end{defn}
\noindent
Note that, by Lemma \ref{7}, the definition of \(P^{\alpha}\) is independent of the choice of the cone $\sigma$ containing \(\alpha\) as a ray.

\begin{thm}\label{Aut}
	Let \(G\) be a connected reductive linear algebraic group. Let \(e \in \mathcal{E}_{x_0}\). For each $\sigma \in \Xi^*$, fix a distinguished section \(s_{\sigma}: X_{\sigma} \rightarrow \mathcal{E} \) such that \(s_{\sigma}(x_0)=e\). Then, for any \(\sigma \in \Xi^*\), \(\rho_{\sigma}(t) g \rho_{\sigma}(t)^{-1}\) extends to \(X_{\sigma}\) if and only if $g \in \bigcap\limits_{\alpha \in \sigma(1)} P^{\alpha} $. In particular,
	\begin{align*}
	{\text{Aut}_T(\mathcal{E}) = \bigcap\limits_{\alpha \in \Xi(1)} P^{\alpha} \subset G}.
	\end{align*}
\end{thm}

\begin{proof}
	Suppose \(\rho_{\sigma}(t) g \rho_{\sigma}(t)^{-1}\) extends to \(X_{\sigma}\). In particular, it extends to  \(x_{\alpha}\) for any $\alpha \in \sigma(1)$. Since \(x_{\alpha}=\lim\limits_{z \rightarrow 0} \lambda^{v_{\alpha}}(z) x_0\), it follows that the limit  
	\(\lim\limits_{z \rightarrow 0} \, (\rho_{\sigma} \circ \lambda^{v_{\alpha}})(z)\, g \,(\rho_{\sigma} \circ \lambda^{v_{\alpha}})(z)^{-1}\)  exists.  Hence, \( g \in P(\rho_{\sigma} \circ \lambda^{v_{\alpha}}) = P^{\alpha}.\) Thus $g \in \bigcap\limits_{\alpha \in \sigma(1)}  P^{\alpha} $.

	Conversely, if $ g \in \bigcap\limits_{\alpha \in \sigma(1)} P^{\alpha} $, we show that $\rho_{\sigma}(t) g \rho_{\sigma}(t)^{-1} \text{ extends to }X_{\sigma}$. Let $\gamma \preceq \sigma$. Write $\gamma=Cone (v_1, \ldots, v_r)$ where \(v_i \in \sigma(1) \). Define \(v_{\gamma}:=v_1+\ldots+v_r \). Then \(v_{\gamma}\) is in the relative interior of the cone \(\gamma\). Since 
	\[(\rho_{\sigma} \circ \lambda^{v_i})(z) (\rho_{\sigma} \circ \lambda^{v_j})(z)=(\rho_{\sigma} \circ \lambda^{v_j})(z) (\rho_{\sigma} \circ \lambda^{v_i})(z),\] 
	by \cite[Page 62 equation (**)]{GIT}, we have 
	\[\bigcap\limits_{\alpha \in \gamma(1) } P(\rho_{\sigma} \circ \lambda^{v_{\alpha}}) \subseteq P(\rho_{\sigma} \circ \lambda^{v_{\gamma}}) .\]
	 This shows that $\rho_{\sigma}(t) g \rho_{\sigma}(t)^{-1}$ extends to the point \(x_{\gamma} \in O_{\sigma}\). Let \(y_{\gamma}\) be the value of the extended function at \(x_{\gamma}\). Any point of \(O_{\gamma}\) is of the form \(t  x_{\gamma}\) for some \(t \in T\). Thus we can extend $\rho_{\sigma}(t) g \rho_{\sigma}(t)^{-1}$ to \(O_{\gamma}\) by sending 
	 $$t  x_{\gamma} \mapsto \rho_{\sigma}(t)  y_{\gamma} \rho_{\sigma}(t)^{-1} .$$ 
	 Since $\rho_{\sigma}(t) g \rho_{\sigma}(t)^{-1}$ extends to \(O_{\gamma}\) for any subcone $\gamma \preceq \sigma$, it extends to the whole of \(X_{\sigma}(=\bigcup\limits_{\gamma \preceq \sigma} O_{\gamma})\). Thus by Proposition \ref{automorphism}, we have 
	 $$ {\text{Aut}_T(\mathcal{E}) = \bigcap\limits_{\alpha \in \Xi(1)} P^{\alpha} \subset G}\, .$$ 	\end{proof}

Next we give a description of \(T\)-equivariant morphisms between \(T\)-equivariant principal \(G\)-bundles over \(X\).

\begin{prop}\label{morph}
	Let $\mathcal{E}, \mathcal{E}'$ be \(T\)-equivariant principal \(G\)-bundles over \(X\). Let \(e \in \mathcal{E}_{x_0}\) $($respectively, \(e' \in \mathcal{E}'_{x_0}\)$)$ and for each $\sigma \in \Xi^*$, take a distinguished section \(s_{\sigma}: X_{\sigma} \rightarrow \mathcal{E}\) $($respectively, \(s'_{\sigma}: X_{\sigma} \rightarrow \mathcal{E'}\)$)$ such that \(s_{\sigma}(x_0)=e\) $($respectively, \(s'_{\sigma}(x_0)=e'\)$)$. Let $\rho_{\sigma}: T \to G $
	and $\rho_{\sigma}':T \to G$ be the group homomorphisms corresponding to $s_{\sigma}$ and $s_{\sigma}'$ respectively.  Then \(\text{Mor}_T(\mathcal{E}, \mathcal{E'})\), the set of all \(T\)-equivariant morphisms from $\mathcal{E}$ to $\mathcal{E'}$, admits a bijective correspondence with the set 
	\[\{(g_0, \{g_{\sigma}\}_{\sigma \in \Xi^*})~|~ g_0, g_{\sigma} \in G, ~ \rho_{\sigma}'= g_{\sigma} \rho_{\sigma} g_{\sigma}^{-1},\; \rho_{\sigma} g_{\sigma}^{-1} g_0 \rho_{\sigma}^{-1} \text{ extends to } X_{\sigma} \text{ for  all } \sigma \in \Xi^*\} .\]
	Moreover, if \(G\) is a connected reductive linear algebraic group, the above set is equal to the following
	\[\{ (g_0, \{g_{\sigma}\}_{\sigma \in \Xi^*})~|~ g_0, g_{\sigma} \in G, ~ \rho_{\sigma}'= g_{\sigma} \rho_{\sigma} g_{\sigma}^{-1},\; g_{\sigma}^{-1} g_0 \in \bigcap_{\alpha \in \sigma(1) } P^{\alpha}  \text{ for  all } \sigma \in \Xi^* \} .\]
\end{prop}

\begin{proof}
	Let  \(\Psi \in \text{Mor}_T(\mathcal{E}, \mathcal{E'})\). Note that, by \cite[Lemma 2.4]{BDP},  \( \widehat{s}'_{\sigma} := \Psi \circ s_{\sigma} \) is a distinguished section of $\mathcal{E'}$ over $X_{\sigma}$. Moreover, the 
	homomorphism $\widehat{\rho}_{\sigma}': T \to G$ associated to $ \widehat{s}'_{\sigma} $, satisfies 
	\begin{equation}\label{rhohatprime}
	\widehat{\rho}_{\sigma}' = \rho_{\sigma}\,.
	\end{equation}
	\noindent
	For any $\sigma \in \Xi^*$, define a map \(\varphi_{\sigma} : X_{\sigma} \rightarrow G\) by
	\begin{equation}\label{Mor1}
	\widehat{s}_{\sigma}'(x)=s'_{\sigma}(x) \cdot \varphi_{\sigma}(x) \text{ for all } x \in X_{\sigma}.
	\end{equation}
	Note that by the same argument in Proposition \ref{automorphism}, \(\Psi\) is determined by its value at \(e \in \mathcal{E}_{x_0}\). Now, using \eqref{Mor1}, we have 
	\begin{equation*}
	\Psi(e)=\Psi(s_{\sigma}(x_0))= \widehat{s}'_{\sigma}(x_0) =s'_{\sigma}(x_0) \cdot \varphi_{\sigma}(x_0) =e' \cdot  \varphi_{\sigma}(x_0).
	\end{equation*}
Thus, \(\varphi_{\sigma}(x_0)\) does not depend on the choice of $\sigma$.
	\noindent
	Note that, by \eqref{Mor1} and  Lemma \ref{distconj},
	\begin{equation}\label{Mor3}
	\widehat{\rho}_{\sigma}'(t)=\varphi_{\sigma}(x_{\sigma})^{-1} \rho_{\sigma}'(t) \varphi_{\sigma}(x_{\sigma}) \,.
	\end{equation}

	\noindent
	Using Lemma \ref{dsclass} and \eqref{Mor3}, we get
	\[\varphi_{\sigma}(tx_0)= \varphi_{\sigma}(x_{\sigma}) \rho_{\sigma}(t) \varphi_{\sigma}(x_{\sigma})^{-1} \varphi_{\sigma}(x_0) \rho_{\sigma}(t)^{-1}.\] Thus \(\rho_{\sigma}(t) \varphi_{\sigma}(x_{\sigma})^{-1} \varphi_{\sigma}(x_0) \rho_{\sigma}(t)^{-1}\) extends to \(X_{\sigma}\). Thus, taking $g_0 = \varphi_{\sigma}(x_0)$ and $g_{\sigma} = \varphi_{\sigma}(x_{\sigma})$, we obtain the 
	set theoretic description of \(\text{Mor}_T(\mathcal{E}, \mathcal{E'})\) as claimed in the statement of the proposition.

	Conversely, given \(g_{\sigma}\) for $\sigma \in \Xi^*$, and  \(g_0\) satisfying the conditions in the statement, we define a morphism \(\Psi: \mathcal{E} \to \mathcal{E}'\) by setting \(\Psi(e)=e' \cdot g_0\). By \(T\) and \(G\) equivariance, \(\Psi\) determines a morphism between $\mathcal{E}|_O$ and $\mathcal{E'}|_O$ as follows: \[\Psi(te \cdot g)=t \, \Psi(e) \cdot g=t e' \cdot g_0\,g, \text{ for all } t \in T, g \in G.\]
	Now for $\sigma \in \Xi^*$ and \(t \in T\), we have \[\Psi(te)=t \, \Psi(s_{\sigma}(x_0))=t s_{\sigma}'(x_0) \cdot g_0=s_{\sigma}'(tx_0) \cdot \rho_{\sigma}'(t)g_0.\] 
	On the other hand, we have \[\Psi(te)=\Psi(t s_{\sigma}(x_0))=\Psi(s_{\sigma}(tx_0)  \cdot \rho_{\sigma}(t))=\Psi(s_{\sigma}(tx_0))\cdot \rho_{\sigma}(t).\] 
	Comparing the right hand sides of the last two equations, we obtain 
	\begin{equation*}
	\Psi(s_{\sigma}(tx_0)) =s_{\sigma}'(tx_0) \cdot \rho_{\sigma}'(t)g_0 \rho_{\sigma}(t)^{-1}=s_{\sigma}'(tx_0) \cdot g_{\sigma} \rho_{\sigma}(t) g_{\sigma} ^{-1} \, g_0 \rho_{\sigma}(t)^{-1}\,.
	\end{equation*}
	Since \(s_{\sigma}\) and  \(s_{\sigma}'\)  are regular on \(X_{\sigma}\), and  \(\rho_{\sigma}(t) g_{\sigma} ^{-1}g_0 \rho_{\sigma}(t)^{-1}\)  extends to \(X_{\sigma}\) by hypothesis, \(\Psi\) extends to give a morphism between $\mathcal{E}|_{X_{\sigma}}$ and $\mathcal{E'}|_{X_{\sigma}}$. These morphisms glue to produce a morphism \(\Psi: \mathcal{E} \rightarrow \mathcal{E'}\).
\end{proof}

\section{Equivariant reduction of structure group}

Let $X$ be a complex toric variety. Let \(\phi : H \rightarrow G\) be homomorphism between complex linear algebraic groups \(H\) and \(G\). 

\begin{lemma}\label{asso}
	Let $p_H : \mathcal{E}_H \rightarrow X$ be a $T$-equivariant principal \(H\)-bundle. Let \(\{s_{\sigma}\}_{\sigma \in \Xi^*}\) be a collection of distinguished sections for \(\mathcal{E}_H\) which associates an admissible collection \(\{\rho_{\sigma}, P(\tau, \sigma)\}\) to \(\mathcal{E}_H\). Then  the associated bundle \(\mathcal{E}_H \times_H  G\) is a \(T\)-equivariant principal \(G\)-bundle. An admissible collection associated to \(\mathcal{E}_H \times_H  G\) is given by \(\{\phi \circ \rho_{\sigma}, \phi(P(\tau, \sigma))\}\).  
\end{lemma}

\begin{proof}
	The associated principal \(G\)-bundle \(\mathcal{E}_H \times_H  G\) is constructed as the quotient \[\left( \mathcal{E}_H \times  G \right)/ \sim \] where \[(e,g) \sim (eh,\phi(h)^{-1}g) \text{ for } h \in H.\] The projection map \[\pi: \mathcal{E}_H \times_H  G \rightarrow X  \] is given by \[\pi([e,g])=p_H(e).\] Note that, for any \(\sigma \in \Xi^*\)  \[\pi^{-1}(X_{\sigma})=p_H^{-1}(X_{\sigma})\times_H G.\]
	  Recall from \eqref{triv} that, for each $\sigma \in \Xi^*$, we have the following trivialization data for \(\mathcal{E}_H\)
	\begin{equation*}
	\begin{split}
	\psi_{\sigma}: \mathcal{E}_H|_{X_{\sigma}} & \rightarrow X_{\sigma} \times H,\\
	s_{\sigma}(x)h & \mapsto (x, h), \text{ where } x \in X_{\sigma}, \ h \in H.
	\end{split}
	\end{equation*}	
	This induces a trivialization for the associated bundle \(\mathcal{E}:=\mathcal{E}_H \times_H  G\) as follows:
	\begin{equation}\label{4}
	\begin{split}
	\pi^{-1}(X_{\sigma})=\mathcal{E}_H|_{X_{\sigma}}\times_H G & \longrightarrow (X_{\sigma} \times H ) \times_H G  \longrightarrow X_{\sigma} \times G\\
	[(s_{\sigma}(x)h, g)] & \mapsto [(x, h), g]   \mapsto (x, \phi(h)g).
	\end{split}
	\end{equation}
	The induced action of the torus \(T\) on $\mathcal{E}$ is given by 
	\[t [(e,g)]=[(te, g)], \text{ where } t \in T,~ e \in \mathcal{E}_H,~ g \in G.\]
	Thus $\mathcal{E}$ becomes a \(T\)-equivariant principal \(G\)-bundle. Note that \begin{equation}\label{associated section}
	s'_{\sigma}(x)=[(s_{\sigma}(x), 1_G)]
	\end{equation}
	defines a section of $\mathcal{E}$ over \(X_\sigma\). We show that it is indeed a distinguished section. Consider the group homomorphism \[\rho_{\sigma}':T \rightarrow G\text{ defined by }\rho_{\sigma}'=\phi \circ \rho_{\sigma}.\] Then, it is easily verified that 
	\begin{equation*}
	t s'_{\sigma}(x)=s'_{\sigma}(tx) \cdot \rho_{\sigma}'(t), \text{ where } t \in T,~ x \in X_{\sigma}.
	\end{equation*}
	This shows that \(\{s'_{\sigma}\}_{\sigma \in \Xi^*}\) is a collection of distinguished sections of \(\mathcal{E}\). Then, \eqref{4} can be written as
	\begin{equation*}
	\begin{split}
	\mathcal{E}|_{X_\sigma}  \longrightarrow X_{\sigma} \times G, ~ s'_{\sigma}(x)g  \mapsto (x, g), 
	\end{split}
	\end{equation*}
	where the action of \(T\) on $X_{\sigma} \times G$ is given by \(t \cdot (x, g)=(tx, \rho_{\sigma}'(t)g)\). The transition functions for $\mathcal{E}$ are given by \[\phi_{\tau \sigma}'=\phi \circ \phi_{\tau \sigma}: X_{\sigma} \cap X_{\tau} \rightarrow G,\] where \[\phi_{\tau \sigma}:X_{\sigma} \cap X_{\tau} \rightarrow  H \] are the transition functions for \(\mathcal{E}_H\). It can be verified that the following relation holds:
	\begin{equation*}
	s_{\sigma}'(x)=s_{\tau}'(x) \cdot \phi_{\tau \sigma}'(x), \text{ where } x \in X_{\sigma}.
	\end{equation*}
	Thus, we have $$\phi_{\tau \sigma}'(x_0)=\phi (\phi_{\tau \sigma}(x_0))=\phi(P(\tau, \sigma)).$$ Hence, \(\{\phi \circ \rho_{\sigma}, \phi(P(\tau, \sigma))\}\) is a class of admissible collection associated to $\mathcal{E}$.
	\end{proof}

We say that a \(T\)-equivariant principal \(G\)-bundle has an equivariant reduction of structure group to a closed subgroup \(H\) of \(G\) if there exists a \(T\)-equivariant principal \(H\)-bundle \(\mathcal{E}_H\) such that \(\mathcal{E}_H \times_H G \) is isomorphic to \(\mathcal{E}_G\) as \(T\)-equivariant principal \(G\)-bundles.

\begin{prop}\label{reduction}
	Let $X(\Xi)$ be a toric variety and \(G\) be a linear algebraic group over \(\C\). Let \(H\) be a closed subgroup of the linear algebraic group \(G\) and let \(\iota:H \hookrightarrow G\) be the inclusion map. 	Let \(\mathcal{E}_G\) be an equivariant principal \(G\)-bundle on \(X\) with an associated class of admissible collection \(\{\rho_{\sigma}, P(\tau, \sigma)\}\). Then \(\mathcal{E}_G\) has an equivariant reduction of structure group to \(H\) if and only if for each $\sigma \in \Xi^*$, there exists $\alpha_{\sigma}, \beta_{\sigma} \in G$ such that the following hold:
	\begin{enumerate}
		\item $\alpha_{\sigma}^{-1} \rho_{\sigma}(t) \alpha_{\sigma} \in H$ for all \(t \in T\),
		\item for every pair \((\tau, \sigma)\) of cones in $\Xi^*$, $\beta_{\tau}^{-1} P(\tau, \sigma) \beta_{\sigma} \in H$, 
		\item $\rho_{\sigma}(t)\beta_{\sigma} \alpha_{\sigma}^{-1}\rho_{\sigma}(t)^{-1}$ extends to a \(G\)-valued function over \(X_{\sigma}\).
	\end{enumerate}
\end{prop}

\begin{proof}
	Assume that \(\mathcal{E}_G\) admits an equivariant reduction of structure group to \(H\), denoted by \(\mathcal{E}_H\). Let \(\{\rho'_{\sigma}, P'(\tau, \sigma)\}\) be a class of admissible collection associated to the \(T\)-equivariant principal \(H\)-bundle \(\mathcal{E}_H\). By Lemma \ref{asso}, an admissible collection associated to \(\mathcal{E}_H \times_H  G\) is given by \(\{\iota \circ \rho'_{\sigma},~ \iota(P'(\tau, \sigma))\}\). Since \(\mathcal{E}_H \times_H G \cong \mathcal{E}_G\) as \(T\)-equivariant principal \(G\)-bundles, the admissible collection associated to them are equivalent. Hence, by Theorem \ref{classifi}, for each $\sigma \in \Xi^*$, we get $\alpha_{\sigma}, \, \beta_{\sigma} \in G$ such that the conditions $(1)-(3)$ hold. This completes the proof of the forward direction.
	
	Conversely, suppose that for each $\sigma \in \Xi^*$, there exist $\alpha_{\sigma}, \beta_{\sigma} \in G$ such that the conditions $(1)-(3)$ hold. Define a collection of homomorphisms $\rho_{\sigma}':T \rightarrow H$ by setting
	\begin{equation}\label{redeq1}
	\rho_{\sigma}'(t)=\alpha_{\sigma}^{-1} \rho_{\sigma}(t) \alpha_{\sigma} \text{ for  all } t \in T.
	\end{equation}
	Consider the elements \[P'(\tau, \sigma)=\beta_{\tau}^{-1} P(\tau, \sigma) \beta_{\sigma} \in H.\] We show that \(\{\rho'_{\sigma}, P'(\tau, \sigma)\}\) satisfies the conditions of Definition \ref{admissible}. We see that conditions \((i), \ (iii), \text{ and } (iv)\) of Definition \ref{admissible} are immediate. Let us verify condition \((ii)\) of Definition \ref{admissible}. We have that 
	\begin{equation*}
	\begin{split}
	&\rho_{\tau }'(t) P'(\tau, \sigma) \rho_{\sigma}'(t)^{-1}\\
	&= \alpha_{\tau}^{-1} \rho_{\tau }(t) \alpha_{\tau} \beta_{\tau}^{-1} P(\tau, \sigma) \beta_{\sigma} \alpha_{\sigma}^{-1} \rho_{\sigma}(t)^{-1} \alpha_{\sigma}\\
	&= \alpha_{\tau}^{-1} [\left( \rho_{\tau }(t) \alpha_{\tau} \beta_{\tau}^{-1} \rho_{\tau }(t)^{-1}\right) \left( \rho_{\tau }(t) P(\tau, \sigma)  \rho_{\sigma}(t)^{-1}\right) \left( \rho_{\sigma}(t) \beta_{\sigma} \alpha_{\sigma}^{-1} \rho_{\sigma}(t)^{-1}\right)   ] \alpha_{\sigma}.
	\end{split}
	\end{equation*}
	By condition (3) of the hypothesis, we see that \(\rho_{\tau }'(t) P'(\tau, \sigma) \rho_{\sigma}'(t)^{-1}\) extends to a \(G\)-valued function over \(X_{\sigma} \cap X_{\tau}\). Hence the collection \(\{\rho'_{\sigma}, P'(\tau, \sigma)\}\) satisfies all conditions of Definition \ref{admissible}. Thus the admissible collection \(\{\rho'_{\sigma}, P'(\tau, \sigma)\}\) corresponds to an equivariant principal \(H\)-bundle, say \(\mathcal{E}_H\).  By Lemma \ref{asso}, we have an admissible collection associated to the equivariant principal \(G\)-bundle \(\mathcal{E}_H \times_H G \). By construction, this admissible collection is equivalent to the admissible collection \(\{\rho_{\sigma}, P(\tau, \sigma)\}\). Hence, by Theorem \ref{classifi}, the \(T\)-equivariant principal \(G\)-bundles \(\mathcal{E}_H \times_H G \) and \(\mathcal{E}_G \) are isomorphic.	
\end{proof}

\noindent
As an immediate corollary, we have the following:
\begin{cor}\label{6} 
	Let \(\mathcal{E}_G\) be a \(T\)-equivariant principal \(G\)-bundle. Then \(\mathcal{E}_G\) has an equivariant reduction of structure group to \(H\) if and only if there is an admissible collection \(\{\rho_{\sigma}, P(\tau, \sigma)\}\) associated to \(\mathcal{E}_G\) with $\text{Im}(\rho_{\sigma}) \subseteq H$ and \(P(\tau, \sigma) \in H \) for all 
	$\sigma, \, \tau  $ in $\Xi^{\ast}$. 	
\end{cor}

\subsection{Equivariant reduction of structure group to a Levi subgroup}
Let \(G\) be a connected reductive linear algebraic group. A Levi subgroup of $G$ is the centralizer in $G$ of some torus in $G$. In particular, Levi subgroups are centralizers of their central torus, i.e. for a Levi subgroup $H$ of $G$, we have $C_G(Z^0(H))=H,$
where $C_G(Z^0(H)) $ denotes the centralizer of $Z^0(H)$ in $G$ (see \cite[Proposition 11.23]{Borel}). Here \(Z^0(H)\) denotes the connected component of the center \(Z(H)\) of \(H\) that contains the identity element.  

Now we give a criterion of a \(T\)-equivariant principal bundle over a toric variety to admit an equivariant reduction of structure group to a Levi subgroup (cf. \cite[Proposition 1.2]{BP_Levi}). 

\begin{prop}\label{levi_reduction} 
Let \(G\) be a connected reductive linear algebraic group. Let $\mathcal{E}_G$ be a \(T\)-equivariant principal \(G\)-bundle over \(X\). Then $\mathcal{E}_G$ has an equivariant reduction of structure group to a Levi subgroup \(H\) of \(G\) if and only if \[Z^0(H) \subseteq 
	\text{Aut}_T(\mathcal{E}_G).\]
	
\end{prop}

\begin{proof}
	Let \(\mathcal{E}_G\) admit an equivariant reduction of structure group to \(H\), say \(\mathcal{E}_H\). This implies that \( \text{Aut}_T(\mathcal{E}_H) \subseteq \text{Aut}_T(\mathcal{E}_G).\) For an element \(h\) in \(Z(H)\), define the map \[\mathcal{E}_H \rightarrow  \mathcal{E}_H , \text{ given by } z \mapsto zh.\] This is clearly a \(T\)-equivariant automorphism of \(\mathcal{E}_H\). Hence,
	we have \[Z^0(H) \subseteq Z(H) \subseteq \text{Aut}_T(\mathcal{E}_H). \] Thus, we obtain \(Z^0(H) \subseteq \text{Aut}_T(\mathcal{E}_G).\)
	
	To see the reverse direction, fix \(e \in (\mathcal{E}_G)_{x_0}\). As in the proof of Theorem \ref{Aut}, for each $\sigma \in \Xi^*$, we choose a distinguished section \(s_{\sigma}\) of \(\mathcal{E}_G\) over \(X_{\sigma}\) such that \(s_{\sigma}(x_0)=e\). Let $\rho_{\sigma} : T \rightarrow G$ denote the corresponding homomorphism. Note that an admissible collection associated to \(\mathcal{E}_G \) is given by \(\{\rho_{\sigma}, P(\tau, \sigma)\}\), where \(P(\tau, \sigma)=1_G\) for all $\tau, \sigma \in \Xi^*$. By Proposition \ref{reduction}, we need to find $\alpha_{\sigma}, \beta_{\sigma} \in G$ for all $\sigma \in \Xi^*$ such that the conditions \((1)-(3)\) are satisfied. Choose $$\beta_{\sigma}=1_G \text{ for all } \sigma \in \Xi^*,$$ then condition \((2)\) is satisfied. Fix $\sigma \in \Xi^*$. Let \(K_{\sigma}\) be a maximal torus of \(G\) such that 
	\begin{equation}\label{torus}
	\text{Im}(\rho_{\sigma})\subseteq K_{\sigma}.
	\end{equation}
\noindent	
  Define  \[P_{\sigma}:=\{g \in G ~ |~ \rho_{\sigma}(t) g \rho_{\sigma}(t)^{-1} \text{ extends to $X_{\sigma}$ }\} \,. \]	
		By Proposition \ref{automorphism}, we have \(P_{\sigma}=\text{Aut}_T({\mathcal{E}_G|_{X_{\sigma}}})\). Note that the action of \(T\) on \(\mathcal{E}_G|_{X_{\sigma}}\) induces an action of \(T\) on \(\text{Aut}({\mathcal{E}_G|_{X_{\sigma}}})\). We see that \(\text{Aut}_T({\mathcal{E}_G|_{X_{\sigma}}})\) is the subgroup that is fixed pointwise by the action of \(T\) on \(\text{Aut}({\mathcal{E}_G|_{X_{\sigma}}})\). Thus $P_{\sigma}$ is an algebraic group by \cite[Proposition (c), page 52]{Borel}. Let \( P^{0}_{\sigma}\) denote the connected component of 
		$P_{\sigma}$ containing the identity element.

	We claim that \(K_\sigma \subseteq P^{0}_{\sigma}\). 
	Let \(g \in K_{\sigma}\). By \eqref{torus}, $\rho_{\sigma}(t)$ and \(g\) commute for all \(t \in T\), which implies that \(g \in P_{\sigma}\).  Since \(K_{\sigma}\) is connected, the claim holds.

	Also, we have \[Z^0(H) \subseteq \text{Aut}_T(\mathcal{E}_G) \subseteq P_{\sigma},\] where the last inclusion follows from Proposition \ref{automorphism}. Thus \(Z^0(H) \subseteq P^0_{\sigma}\). Note that \(Z^0(H)\) is a torus of  \(P^0_{\sigma}\) by \cite[Proposition 11.23(i)]{Borel}. Let \(K'_{\sigma}\) be a maximal torus of \(P^0_{\sigma}\) containing \(Z^0(H)\). Then the maximal tori \(K_{\sigma}\) and \(K'_{\sigma}\) of the common connected algebraic group \(P^0_{\sigma}\) are conjugate by an element of \(P^0_{\sigma}\). Let $\alpha_{\sigma} \in P^0_{\sigma}$ be such that \[K'_{\sigma}=\alpha_{\sigma}^{-1} K_{\sigma} \alpha_{\sigma}.\] Therefore, we have $$\alpha_{\sigma}^{-1}\, \rho_{\sigma}(t) \, \alpha_{\sigma} \subseteq K'_{\sigma} \subseteq C_G(Z^0(H))=H.$$ 
Thus condition \((1)\) holds. Finally, condition \((3)\) is immediate, since $\alpha_{\sigma} \in P_{\sigma}$.
	\end{proof}

\begin{rmk}{\rm
		Let $\mathcal{E}_G$ be a $T$-equivariant principal \(G\)-bundle on a toric variety \(X=X(\Xi)\). Assume that \(G\) is reductive and recall the definition of \(P^{\alpha}\) from Definition \ref{Palpha}. Then we have seen in the proof of Proposition \ref{levi_reduction} that \( P_{\sigma} \) contains a maximal torus of \(G\) for all $\sigma \in \Xi^*$. Also, by Theorem \ref{Aut}, we have \( P_{\sigma} = \bigcap\limits_{\alpha \in \sigma(1)} P^{\alpha} \).
		
		Now consider the complex $\mathscr{P}(G)$ whose vertices are the	parabolic subgroups of \(G\) and whose simplices are of the form \(\{P^{\alpha} ~|~  \alpha \in A \text{ and $\bigcap\limits_{\alpha \in A}P^{\alpha}$  contains a common maximal torus of $G$}\} \) for some subset \(A\) of $\Xi(1)$.
		
		Thus, we have shown that an equivariant principal \(G\)-bundle $\mathcal{E}$ gives rise to simplicial maps from the fan $\Xi$ to the complex $\mathscr{P}(G)$ (cf. \cite[Section 6.1]{Kly}).	}
\end{rmk}

\subsection{Equivariant splitting of an equivariant principal bundle}

We say that a \(T\)-equivariant principal $G$-bundle $\mathcal{E}_G$ splits equivariantly if it admits an equivariant reduction of structure group to a maximal torus of \(G\). We give a criterion of equivariant splitting of an equivariant principal bundle using the Kaneyama type description given in Section \ref{Kan description}. As before, let us fix an embedding of \(G\) in \({\rm GL}(r, \C)\) such that \(K_0:=D(r, \C) \cap G\) is a maximal torus of $G$. 

\begin{prop}
	Let $X = X(\Xi)$ be a toric variety and \(\mathcal{E}_G\) be a $T$-equivariant principal \(G\)-bundle on \(X\). Then \(\mathcal{E}_G\)  splits equivariantly if and only if it has a Kaneyama data associated to it,  given by \((\xi, P)\) $($see Theorem \ref{K4}$)$, such that  \(P(\tau, \sigma)=1_G\) for all $\tau, \sigma \in \Xi^*$.
\end{prop}

\begin{proof}
	Let \(\mathcal{E}_G\) split equivariantly, i.e., it admits an equivariant reduction of structure group to a maximal torus of \(G\), say \(K\). Then by Corollary \ref{6}, there is an admissible collection \(\{\widehat{\rho}_{\sigma}, \widehat{P}(\tau,\, \sigma)\}\) associated to \(\mathcal{E}_G\) with $\text{Im}(\widehat{\rho}_{\sigma}) \subseteq K$ and \(\widehat{P}(\tau, \sigma)  \subseteq K \). Since any two maximal torus are conjugate, there exists \(g \in G\) such that \(K_0=g^{-1} K g\). Now consider
	\begin{equation*}
	\alpha_{\sigma}=\beta_{\sigma}=g \text{ for all } \sigma \in \Xi^*.
	\end{equation*}
	Note that these elements satisfy the condition in Definition \ref{equiv_admissible} (c) and replace \(\{\widehat{\rho}_{\sigma}, \widehat{P}(\tau, \sigma)\}\) by an equivalent collection \(\{\rho_{\sigma}, P(\tau, \sigma)\}\) using these elements. Then we have $\text{Im}(\rho_{\sigma}) \subseteq K_0$ and \(P(\tau, \sigma)  \subseteq K_0 \).

	Fix a cone \(\sigma \in \Xi^*\). Choose the elements
	\begin{equation*}
	\begin{split}
	&\beta_{\sigma}=1_G, \text{ and }\\
	&\beta_{\gamma}=P(\gamma, \ \sigma) \text{ for all maximal cones } \gamma \neq \sigma,\\
	&\alpha_{\tau}=\beta_{\tau} \text{ for  all } \tau \in \Xi^*.	
	\end{split}
	\end{equation*}
	Then we have 
	\begin{equation*}
	P'(\gamma, \sigma)=\beta_{\gamma}^{-1} P(\gamma, \sigma) \beta_{\sigma}=1_G
	\end{equation*}
	for all maximal cones \(\gamma \neq \sigma\). Let \(\tau, \ \tau'\) be any two maximal cones. Then, by the cocycle condition, we have
	\begin{equation*}
	P'(\tau, \tau')= P'(\tau, \sigma)P'(\sigma, \tau')=1_G.
	\end{equation*} Again, for any \(\sigma \in \Xi^*\), we have  \[\rho'_{\sigma}=\alpha_{\sigma}^{-1} \rho_{\sigma} \alpha_{\sigma}= \rho_{\sigma}.\]
	Therefore,  we can replace the admissible collection \(\{\rho_{\sigma}, P(\tau, \sigma)\}\) by the equivalent admissible collection \(\{\rho_{\sigma}, P'(\tau, \sigma)\}\) where \(P'(\tau, \sigma)=1_G\) for all $\tau, \sigma \in \Xi^*$. Finally, as in the proof of Theorem \ref{K4}, \(\{\rho_{\sigma}, P'(\tau, \sigma)\}\) gives rise to a Kaneyama data \((\xi, P')\) with \(P'(\tau, \sigma)=1_G\) for all $\tau, \sigma \in \Xi^*$.

	The converse follows from Corollary \ref{6}.
	\end{proof}

\noindent
As an immediate corollary, we have 
\begin{cor}(cf \cite[Lemma 4.3]{BDP})
	If the maximal torus \(K_0\) is normal in \(G\), then any $T$-equivariant principal $G$-bundle over a toric variety $X = X(\Xi)$ admits an equivariant reduction of structure group to \(K_0\).
\end{cor}

We remark that a $T$-equivariant principal $G$-bundle on a toric variety splits equivariantly if
$G$ is an abelian complex linear algebraic group. The proof is similar to \cite[Theorem 5.1]{DP}.

\begin{rmk}\label{choice of splitting}{\rm
		Let \(G\) be a connected algebraic group over \(\C\) and let \(\mathcal{E}_G\) be a \(T\)-equivariant principal \(G\)-bundle on a  toric variety \(X\). Let \(K\) be a maximal torus of \(G\) such that \(\mathcal{E}_G\) admits an equivariant reduction of structure group to \(K\), say \(\mathcal{E}_K\) such that \(\mathcal{E}_K \times_{K} G \cong \mathcal{E}_G\). Let \(\widetilde{K}\) be another maximal torus of \(G\). Since any two maximal tori of a connected algebraic group are conjugate, there exists an element \(g\) of \(G\) such that \(\widetilde{K}=g^{-1}Kg\). Then consider \(\mathcal{E}_K \times_{\widetilde{K}} G\), where \(\widetilde{K}\) acts on both \(\mathcal{E}_K \) and \(G\) via the inclusion \(g \widetilde{K} g^{-1}=K \subset G\). It is easy to see that \(\mathcal{E}_K \times_{\widetilde{K}} G \cong \mathcal{E}_K \times_{K} G \cong \mathcal{E}_G\). Hence, \(\mathcal{E}_K\) itself is a reduction of structure group for \(\widetilde{K} \subset G\).
	
}
\end{rmk}

Let \(\phi:G \rightarrow G'\) be an injective homomorphism of connected reductive linear algebraic groups over \(\C\). Then we have the following theorem, which is an equivariant version of \cite[Theorem 1.2]{BGH}, without any assumption on the Picard number of the underlying toric variety.

\begin{thm}\label{extension split}
	Let $X=X(\Xi)$ be a toric variety and \(\mathcal{E}_G\) be a \(T\)-equivariant principal \(G\)-bundle on \(X\).  Denote by \(\mathcal{E}_{\phi}\) the \(T\)-equivariant principal \(G'\)-bundle obtained from \(\mathcal{E}_G\) by extending the structure group via \(\phi\). Suppose that \(\mathcal{E}_{\phi}\) is equivariantly split, then \(\mathcal{E}_G\) itself splits equivariantly.
\end{thm}

\begin{proof}
	Fix \(e \in (\mathcal{E}_G)_{x_0}\). As in the proof of Theorem \ref{Aut}, for each $\sigma \in \Xi^*$, we choose a distinguished section \(s_{\sigma}\) of \(\mathcal{E}_G\) over \(X_{\sigma}\) such that \(s_{\sigma}(x_0)=e\). Let $\rho_{\sigma} : T \rightarrow G$ denote the corresponding homomorphism. Note that an admissible collection associated to \(\mathcal{E}_G \) is given by \(\{\rho_{\sigma}, 1_G\}\). For each $\sigma \in \Xi^*$, we have a distinguished section of \(\mathcal{E}_{\phi}\) over \(X_{\sigma}\) given by \(s'_{\sigma}(x)=[(s_{\sigma}(x), 1_G)]\) (see \eqref{associated section}). The admissible collection associated to \(\mathcal{E}_{\phi}\) is given by \(\{\phi \circ \rho_{\sigma}, 1_{G'}\}\). Now, for each \(\sigma \in \Xi^*\), we have the subgroups \[P_{\sigma}=\{g \in G ~ |~ \rho_{\sigma}(t) g \rho_{\sigma}(t)^{-1} \text{ extends to $X_{\sigma}$ }\} \] and \[P'_{\sigma}=\{g' \in G' ~ |~ (\phi \circ \rho_{\sigma})(t) g' (\phi \circ \rho_{\sigma})(t)^{-1} \text{ extends to $X_{\sigma}$ }\} .\]
	We claim that 
	\begin{equation}\label{extension parabolic}
		\phi^{-1}(P'_{\sigma})=P_{\sigma}.
	\end{equation}
	To see this, let \(g \in P_{\sigma}\). Let us denote the extension of \[ T \rightarrow G, \  t \mapsto \rho_{\sigma}(t) g \rho_{\sigma}(t)^{-1}\] to \(X_{\sigma}\) by \[\psi_{\sigma}:X_{\sigma} \rightarrow G.\] Then the map \[\phi \circ \psi_{\sigma}:X_{\sigma} \rightarrow G'\] extends the map \[ T \rightarrow G', \  t \mapsto (\phi \circ \rho_{\sigma})(t) \phi(g) (\phi \circ \rho_{\sigma})(t)^{-1}.\] This shows that \(P_{\sigma} \subset \phi^{-1}(P'_{\sigma}) \). For the reverse direction, consider an element \(g' \in P'_{\sigma}\) such that \(g'=\phi(g)\) for some \(g \in G\). We have the map \[\eta_{\sigma}: T \rightarrow G', \    t \mapsto (\phi \circ \rho_{\sigma})(t) \phi(g) (\phi \circ \rho_{\sigma})(t)^{-1}.\] Since \(g' \in P'_{\sigma}\), we have an extension of \(\eta_{\sigma}\), say \(\widetilde{\eta}_{\sigma}:X_{\sigma} \rightarrow G'\). Note that \[\widetilde{\eta}_{\sigma}(T)=\eta_{\sigma}(T) \subset \phi(G).\]
	Thus by continuity of the map \(\widetilde{\eta}_{\sigma}\), we have \[\widetilde{\eta}_{\sigma}(X_{\sigma}) \subset \phi(G),\] since \(\phi(G)\) is closed. This implies that \[\phi^{-1}\circ \widetilde{\eta}_{\sigma}:X_{\sigma} \rightarrow G \] gives an extension of \[ T \rightarrow G, \  t \mapsto \rho_{\sigma}(t) g \rho_{\sigma}(t)^{-1}\,\] to \(X_{\sigma}\). Therefore, we get \(g \in P_{\sigma}\). Hence the claim \eqref{extension parabolic} holds. By Proposition \ref{automorphism}, we have \[\text{Aut}_T(\mathcal{E}_G)= \bigcap_{\sigma \in \Xi^*}P_{\sigma} \text{ and } \text{Aut}_T(\mathcal{E}_{\phi})= \bigcap_{\sigma \in \Xi^*}P'_{\sigma}.\] Thus, by \eqref{extension parabolic}, we get 
	\begin{equation}\label{extension auto}
		\text{Aut}_T(\mathcal{E}_G)= \phi^{-1}(\text{Aut}_T(\mathcal{E}_{\phi})).
	\end{equation}
	
	Let \(K_0\) be a maximal torus in \(G\). Then there exists a maximal torus \(K'_0\) of \(G'\) such that \(\phi(K_0) \) is contained in \(K'_0\). Since the \(T\)-equivariant principal \(G'\)-bundle \(\mathcal{E}_{\phi}\) splits equivariantly, by Remark \ref{choice of splitting}, we may assume that \(\mathcal{E}_{\phi}\) admits an equivariant reduction of structure group to \(K'_0\). Then, by Proposition \ref{levi_reduction}, we have \[K'_0 \subset \text{Aut}_T(\mathcal{E}_{\phi}).\] Taking inverse image under \(\phi\) and using \eqref{extension auto}, we obtain \[K_0 \subset \text{Aut}_T(\mathcal{E}_G).\] Thus, by Proposition \ref{levi_reduction}, \(\mathcal{E}_G\) splits equivariantly. This completes the proof.
	\end{proof}

A vector bundle on $\mathbb{P}^n$ splits precisely when its restriction to some plane \(\mathbb{P}^2 \subset \mathbb{P}^n\) splits (\cite[Theorem 2.3.2]{OSS}). This is proved essentially using Horrock's criterion. Using Theorem \ref{extension split} together with equivariant version of Horrock's criterion (see \cite[Corollary 4.3.3]{Kly}) we have the following corollary (cf. \cite[Corollary 4.1]{BGH}).
 
 \begin{cor}
 	A \(T\)-equivariant principal \(G\)-bundle $\mathcal{E}$ on $\mathbb{P}^n$ admits an equivariant splitting if and only if there is a torus invariant plane \(\mathbb{P}^2 \subset \mathbb{P}^n\) such that $\mathcal{E}|_{\mathbb{P}^2 }$ splits equivariantly.
 \end{cor}

\begin{proof}
 If $\mathcal{E}$ splits equivariantly, then clearly its restriction to any torus invariant subvariety splits equivariantly. Conversely, assume that there is a torus invariant plane \(\mathbb{P}^2 \subset \mathbb{P}^n\) such that $\mathcal{E}|_{\mathbb{P}^2 }$ splits equivariantly. Consider an embedding $\iota : G \hookrightarrow {\rm GL}(V)$, where \(V\) is a finite dimensional vector space. Since $\mathcal{E}|_{\mathbb{P}^2 }$ splits equivariantly, the associated vector bundle \((\mathcal{E} \times_{\iota} V)|_{\mathbb{P}^2}\)  also splits equivariantly. Now arguing as in the proof of \cite[Theorem 2.3.2]{OSS} we get that \[H^i(\mathbb{P}^n, (\mathcal{E} \times_{\iota} V)(k))=0 \text{ for } i=1, \ldots, n-1 \text{ and } k \in \Z. \]
 Hence \(\mathcal{E} \times_{\iota} V\) splits equivariantly by \cite[Corollary 4.3.3]{Kly}. Then by Theorem \ref{extension split} it follows that  $\mathcal{E}$ splits equivariantly.
\end{proof}
 
We conclude this section with the following principal bundle analogue of a famous result \cite[Corollary 3.5]{Kan2} of Kaneyama on the existence of equivariant splitting for $T$-equivariant vector bundles of rank less than $ n$ on $\mathbb{P}^n$.

\begin{thm}\label{equivkan}
Let $G$ be a connected reductive linear algebraic group over $\C$ and fix an embedding of \(G\) in  ${\rm GL}(r,\C)$. Any $T$-equivariant principal \(G\)-bundle on $\mathbb{P}^n$ splits equivariantly if \(r<n\).
\end{thm}

\begin{proof}
Let $\mathcal{E}_G$ be a \(T\)-equivariant principal $G$-bundle on $\mathbb{P}^n$.  Then we have that $\mathcal{E}_G \times_G {\rm GL}(r, \C)$  splits equivariantly since $r<n$ (by \cite[Corollary 3.5]{Kan2}). Then by Theorem \ref{extension split}, $\mathcal{E}_G$ splits equivariantly. (Alternatively, by \cite[Theorem 1.2]{BGH}, $\mathcal{E}_G$ splits. Now by \cite[Proposition 3.2]{BP}, we have an equivariant splitting of $\mathcal{E}_G$.)
\end{proof}

\section{Stability of equivariant principal bundles}

Let $G$ be a connected reductive affine algebraic group defined over $\mathbb C$. The Lie algebra
of $G$ will be denoted by $\mathfrak g$.
The connected component of the center of $G$ containing the identity element will be denoted by $Z^0(G)$.

The Levi factor of a parabolic subgroup $P$ of $G$ is unique up to conjugations by elements of the unipotent
radical of $P$. For any character $\chi \,:\, P \,\longrightarrow\, \C^{*}$ which
is dominant with respect to some Borel subgroup of $G$ contained in $P$, the restriction
of $\chi$ to \(Z(G)\) is trivial (see \cite[Exercise 4, p. 162]{Hum}). 

Let $X$ be 
a complex nonsingular projective toric variety under the action of the torus $T$. Fix a polarization \(L\), i.e. an (equivariant) ample line bundle. The notions of equivariant semistability 
and equivariant stability extend to the \(T\)-equivariant principal bundles on $X$ (see
Definition \ref{estab}). In the case of 
\(T\)-equivariant torsion-free sheaves over a projective toric variety,
 the notion of equivariant semistability (respectively, 
equivariant stability) coincides with that of semistability (respectively, stability); see \cite[Theorem 2.1]{BDGP} or \cite[Proposition 4.13]{Kool}. In this section we deal with this question in the case of \(T\)-equivariant principal \(G\)-bundles. For the definition of stability for any principal \(G\)-bundle over any projective variety see \cite[Definition 3]{RR1}. Analogously, the notion of equivariant (semi)-stability for equivariant principal bundle is given as follows.
% when $G$ is a connected reductive affine algebraic group. 

\begin{defn}\label{estab}
Let $X$ be a complex nonsingular projective toric variety together with a polarization \(L\). An equivariant principal $G$-bundle $\mathcal{E}_G$ over $X$ is called equivariantly semistable $($respectively, 
equivariantly stable$)$ with respect to \(L\) if for every proper parabolic subgroup $P \,\subsetneq\, G$ and every equivariant reduction 
\(\mathcal{E}_P\) of \(\mathcal{E}_G\) over any \(T\)-invariant open subset \(U \subset X \) with \(\text{codim}(X 
\setminus U) \geq 2\), and for every nontrivial character $\chi \,:\, P \,\longrightarrow\, \C^{*}$ which
is dominant with respect to some Borel subgroup of $G$ contained in $P$, the associated
line bundle $\mathcal{E}_{P}(\chi)$ has 
non-positive $($respectively, negative$)$ degree with respect to the above fixed polarization.
\end{defn}

Using the existence and uniqueness of the canonical reduction of principal \(G\)-bundles (see \cite[Theorem 6]{HN_PB}, \cite{HN BH}), we see that the canonical reduction of a \(T\)-equivariant principal \(G\)-bundle is also \(T\)-equivariant. Hence, the notions of semistability and equivariant semistability coincide for \(T\)-equivariant principal \(G\)-bundles. Clearly a stable equivariant principal \(G\)-bundle is equivariantly stable. We will show in Theorem \ref{stab_indecom} that the converse also holds.

\begin{defn}[{\cite[Definition 3.4]{AB}}]\label{admi_red}
Let \(P\) be a parabolic subgroup of \(G\), and $\mathcal{E}_P$ be the reduction of the structure group of a 
principal \(G\)-bundle $\mathcal{E}_G$ to \(P\). The reduction $\mathcal{E}_P$ is called admissible if for any 
character $\chi$ on \(P\), which is trivial on \(Z^0(G)\), the associated line bundle $\mathcal{E}_{P}(\chi)$ has 
degree zero.
\end{defn}

\begin{defn}[{\cite[Definition 3.5]{AB}}]\label{pstab}
A semistable principal $G$-bundle $\mathcal{E}_G$ over \(X\) is said to be polystable if either $\mathcal{E}_G$ is 
stable or there exists a parabolic subgroup $P$ of $G$ and a reduction $\mathcal{E}_{L(P)}$ of $\mathcal{E}_{G}$ 
to the Levi factor $L(P)$ of $P$ such that
\begin{enumerate}[$($i$)$]
\item the principal \(L(P)\)-bundle $\mathcal{E}_{L(P)}$ is stable, and

\item the principal \(P\)-bundle $\mathcal{E}_{L(P)}(P) \,:=\,\mathcal{E}_{L(P)} \times_{L(P)} P$, obtained
by extending the structure group of $\mathcal{E}_{L(P)}$ using the natural inclusion of the Levi
factor \(L(P)\, \hookrightarrow\, P\), is an admissible reduction of $\mathcal{E}_{G}$ to \(P\). 
\end{enumerate}
\end{defn}

\begin{defn}[{\cite[Definition 2.1]{Krull-Sch-Red}}]
A principal \(G\)-bundle \(\mathcal{E}_G\) over \(X\) is called \(L\)-indecomposable if \(Z^0(G)\) is a
maximal torus of \({\rm Aut}(\mathcal{E}_G)\).
\end{defn}

The following result is a corollary of \cite[Proposition 3.2]{RR} and \cite[Proposition 2.4]{Krull-Sch-Red}.

\begin{prop}\label{stab_indecom}
	Any stable principal \(G\)-bundle on \(X\) is \(L\)-indecomposable.
\end{prop}

\begin{thm}\label{thmstable}
	Let $\mathcal{E}_G$ be a \(T\)-equivariant principal \(G\)-bundle on a nonsingular projective toric variety
\(X\). Then $\mathcal{E}_G$ is  stable if and only if $\mathcal{E}_G$ is equivariantly stable.
\end{thm}

\begin{proof}
Let $\mathcal{E}_{G}$ be an equivariantly stable \(T\)-equivariant principal $G$-bundle over $X$. Then the adjoint 
bundle $\text{ad}(\mathcal{E}_{G})\,=\, E_G\times^G{\mathfrak g}$ is a \(T\)-equivariant vector bundle whose 
fibers are Lie algebras identified with ${\mathfrak g}$ uniquely up to automorphisms of $\mathfrak g$ given by 
the adjoint action of $G$. The bundle $\text{ad}(\mathcal{E}_{G})$ is equivariantly semistable (by equivariant 
analogue of \cite[Proposition 2.10]{AB}). By \cite[Proposition 4.13]{Kool}, $\text{ad}(\mathcal{E}_{G})$ is 
semistable. Let $\mathcal{S} \subset \text{ad}(\mathcal{E}_G)$ be its socle which is the unique maximal polystable 
subsheaf (see \cite[Lemma 1.5.5]{HL}). Suppose $\mathcal{S}$ is a proper subsheaf of the adjoint bundle. Then by 
\cite[Lemma 1.5.9]{HL}, $\mathcal{S}$ is invariant under the action of $T$. We note that
the proof in \cite{HL} is for Gieseker semistable sheaves, but exactly same proof works for $\mu$-semistable
sheaves too.

Now by \cite[Proposition 2.12]{AB}, this will produce a $T$-invariant parabolic subalgebra bundle $\mathcal P$ of 
${\rm ad}(\mathcal{E}_G)$ over a $T$-invariant open subset $U \,\subset\, X$ with $\text{codim}(X \setminus U)
\,\geq \,2$ and 
$\text{deg}(\mathcal P)\,=\,0$. This will contradict equivariant stability of $\mathcal{E}_G$ by \cite[Lemma 
2.11]{AB}. Hence ${\rm ad}(\mathcal{E}_G)$ is polystable. Now by \cite[Corollary 3.8]{AB}, $\mathcal{E}_G$ is 
polystable.

Suppose the polystable principal $G$-bundle $\mathcal{E}_G$ is not stable. Then there exists a parabolic subgroup 
$P$ of $G$, and a reduction $\mathcal{E}_{L(P)}$ of $\mathcal{E}_{G}$ to the Levi factor $L(P)$ of $P$ 
satisfying the two conditions of Definition \ref{pstab}.

We claim that there is a parabolic subgroup $Q$ of $G$, and an 
equivariant reduction $\mathcal{E}_{L(Q)}$ of $\mathcal{E}_{G}$ to the Levi factor $L(Q)$ of $Q$, which also 
satisfies the condition \((2)\) of Definition \ref{pstab}.
	
To prove the above claim, let \(\text{Aut}^0(\mathcal{E}_G)\) denote the connected component containing the 
identity element of \(\text{Aut}(\mathcal{E}_G)\).  We can choose a maximal torus $\widetilde{T}$ of 
\(\text{Aut}^0(\mathcal{E}_G)\) on which \(T\) acts trivially as follows (cf. \cite[proof of Theorem 2.1]{BDGP}). 
Consider the semidirect product \(\text{Aut} (\mathcal{E}_G) \rtimes T\), where the action of \(T\) on 
\(\text{Aut} (\mathcal{E}_G)\) is given as follows \[(t \Phi)(e)\,=\,t \Phi(t^{-1} e)\]
for all $\Phi\,\in\, \text{Aut} (\mathcal{E}_G)$, all $t \,\in\, T$ and all $e \,\in \,\mathcal{E}_G$.
Let $$\widetilde{T'} \,\subseteq\, \text{Aut} 
(\mathcal{E}_G) \rtimes T$$ be a maximal torus containing the subgroup \(T\) of $\text{Aut} (\mathcal{E}_G) \rtimes 
T$. Then \[\widetilde{T}\,:=\,\widetilde{T'} \cap \text{Aut}(\mathcal{E}_G)\,\subset\,\text{Aut}(\mathcal{E}_G) \] is 
the desired maximal torus of \(\text{Aut}^0(\mathcal{E}_G)\) on which \(T\) acts trivially.

The above maximal torus $\widetilde{T}$ gives rise to an equivariant reduction $\mathcal{E}_H$ of $\mathcal{E}_G$ to a 
Levi subgroup \(H\) of \(G\) (see \cite[Proposition 3.3]{BP_Levi}). Furthermore, this reduction $\mathcal{E}_H$ is 
\(L\)-indecomposable by \cite[Theorem 3.2]{Krull-Sch-Red}. Now by Proposition \ref{stab_indecom}, the 
\(L(P)\)-bundle \(\mathcal{E}_{L(P)}\) is also \(L\)-indecomposable. Then using \cite[Proposition 
3.3]{Krull-Sch-Red}, there exists \(g_0 \,\in\, G\) and $\Phi \,\in \,\text{Aut}^0(\mathcal{E}_G)$ such that
\begin{equation}\label{eqivstab}
H\,=\,g_0^{-1} \, L(P) \, g_0\ \ \text{ and }\ \ \Phi(\mathcal{E}_{L(P)}) \, g_0\,=\,\mathcal{E}_H\, .
\end{equation}
Set \(Q\,=\,g_0^{-1} \, P \, g_0\); then \(Q\) is a parabolic subgroup with its Levi factor \(L(Q)
\,=\,H\). From \eqref{eqivstab} it is clear that the \(T\)-equivariant principal \(Q\)-bundle $\mathcal{E}_H(Q)$ is
an equivariant reduction of structure group of $\mathcal{E}_G$ to \(Q\) which is also an admissible reduction.
Hence, for any character $\chi$ of \(Q\) which is trivial on  \(Z^0(G)\), we have  
	\begin{equation}\label{admi_1}
		\text{deg}(\mathcal{E}_H(Q)(\chi))\,=\, 0.
	\end{equation}
This proves our claim.
	
As mentioned before, any dominant character of \(Q\) is necessarily trivial on \(Z(G)\). Thus, we have an 
equivariant reduction of structure group $\mathcal{E}_H(Q)$ of $\mathcal{E}_G$ to the parabolic subgroup \(Q\)
such that \eqref{admi_1} holds for any 
dominant character $\chi$ of \(Q\). This contradicts that $\mathcal{E}_G$ is 
equivariantly stable. Hence $\mathcal{E}_G$ is stable.
\end{proof}

\section{Examples}
In this section we give explicit computations for tangent frame bundles of nonsingular complete toric varieties to illustrate our results.

\subsection{The tangent frame bundle}\label{tangentframe}

Let $\mathcal{T}_X$ be the tangent bundle on a nonsingular complete toric variety \(X=X(\Xi)\) of dimension \(n\). Let $\sigma=\text{Cone}(v_1, \ldots, v_n)$ be an \(n\)-dimensional cone in $\Xi$. Let \(u_1, \ldots, u_n \) be the dual basis of \(M\) corresponding to the basis \(v_1, \ldots, v_n\) of \(N\). The affine toric variety defined by \(\sigma\) is \(X_{\sigma}=\text{Spec } \C[x_1, \ldots, x_n]\), where \(x_i=\chi^{u_i}, ~i=1, \ldots, n\). For any \(x \in X_{\sigma}\), the fiber of $\mathcal{T}_X$ at $x$, denoted by \(\left( \mathcal{T}_X\right)_x \), is a complex vector space with a basis \(\{\frac{\partial}{\partial x_1}\big|_x, \ldots, \frac{\partial}{\partial x_n}\big|_x\}\). Then we have a trivialization for $\mathcal{T}_X$ over \(X_{\sigma}\),
\begin{equation}\label{psi_x}
	\begin{split}
		\psi_{\sigma} : \mathcal{T}_X|_{X_{\sigma}}  \stackrel{\cong} \longrightarrow X_{\sigma} \times \C^n, ~	w  \mapsto (x, (c_1, \ldots, c_n)), 
	\end{split}
\end{equation}
\noindent
$\text{where } x \in X_{\sigma}, 	~ w =\sum\limits_{i=1}^n c_i \frac{\partial}{\partial x_i}\big|_x \in \left( \mathcal{T}_X\right)_x \text{ and } c_1,  \ldots,  \ c_n \in \C.$ The tangent bundle \(\mathcal{T}_X\) is naturally a \(T\)-equivariant vector bundle so that 
\begin{equation}\label{action_on_section}
	t \, \frac{\partial}{\partial x_i}\bigg|_x= \chi^{u_i}(t) \frac{\partial}{\partial x_i}\bigg|_{tx}, \text{ for } i=1, \ldots, n.
\end{equation}
We have the following action of \(T\) on $\C^n$ induced by the trivialization $\psi_{\sigma}$ (see \eqref{psi_x}),
\[t(c_1, \ldots, c_n)=(tx, (\chi^{u_1}(t)c_1, \ldots, \chi^{u_n}(t) c_n)), \text{ where } x \in X_{\sigma},\ t \in T \text{ and }  c_1,  \ldots ,  c_n \in \C.\]	

\noindent
The frame bundle associated to \(\mathcal{T}_X\) is a principal \({\rm GL}(n, \C)\)-bundle $$\pi: Fr(\mathcal{T}_X) \rightarrow X,$$ where \(Fr(\mathcal{T}_X)=\bigsqcup\limits_{x \in X} Fr(\left( \mathcal{T}_X\right) _x)\) and the natural projection $\pi$ maps \(Fr(\left( \mathcal{T}_X\right) _x)\) to \(x\). Here \(Fr(\left( \mathcal{T}_X\right) _x)\) denotes the set of all ordered bases of the vector space \(\left( \mathcal{T}_X\right) _x\). We will represent an ordered basis \(w_1, \ldots, w_n\) of \(Fr(\left( \mathcal{T}_X\right) _x)\) by a row vector \(w=[w_1, \, \cdots, \, w_n]\). The local trivialization $\psi_{\sigma}$ of \(\mathcal{T}_X\)  (see \eqref{psi_x}) induces a local trivialization for the frame bundle as follows:
\begin{equation}\label{eg2}
	\begin{split}
		\widetilde{\psi}_{\sigma} : Fr(\mathcal{T}_X)|_{X_{\sigma}}  & \stackrel{\cong} \longrightarrow X_{\sigma} \times Fr(\C^n)\\
		[w_1, \, \cdots, \, w_n]  \in Fr(\left( \mathcal{T}_X\right) _x) & \longmapsto (x, [(pr_2 \circ \psi_{{\sigma}})(w_1), \, \cdots, \, (pr_2 \circ \psi_{{\sigma}})(w_n) ]). 
	\end{split}
\end{equation}
where $pr_2: X_{\sigma} \times \mathbb{C}^n \to \mathbb{C}^n$ is the natural projection.
\noindent	
There is a natural action of \(T\) on \(Fr(\mathcal{T}_X)\) given by 
\begin{equation}\label{action_on_frame}
	t [w_1,\, \cdots,\, w_n]=[t w_1, \, \cdots, \, t w_n], 
\end{equation}
where \(t \in T\) and \([w_1, \, \cdots, \,w_n] \in Fr(\left( \mathcal{T}_X\right) _x)\). This induces an action of \(T\) on $Fr(\C^n)$ via the trivialization \(\widetilde{\psi}_{\sigma} \) in \eqref{eg2} as follows:
\begin{equation}\label{eg3}
	\begin{split}
		&t (x, \, [(c^1_1, \ldots, c^1_n), \, \cdots,  \,(c^n_1, \ldots, c^n_n)]) \\
		&=(tx, \, [(\chi^{u_1}(t)c^1_1, \ldots , \chi^{u_n}(t) c^1_n), \, \cdots, \, (\chi^{u_1}(t)c^n_1, \ldots , \chi^{u_n}(t)c^n_n)]),
	\end{split}
\end{equation}
for \( x \in X_{\sigma}, ~ w_j=\sum\limits_{k=1}^n c^j_k \frac{\partial}{\partial x_k}\big|_x \in Fr(\left( \mathcal{T}_X\right) _x) \text{ and } t \in T.\) Define a homomorphism 
\begin{equation}\label{tangent_action}
	\begin{split}
		\rho_{\sigma}:  T \rightarrow {\rm GL}(n, \C) \text{ given by }   t  \mapsto  \text{diag}(\chi^{u_1}(t), \, \ldots, \, \chi^{u_n}(t)), \, t \in T. 
	\end{split}
\end{equation}
Then, \eqref{eg3} can be written as 
\begin{equation*}\label{corr_homo}
	t (x,\, [\underline{c}^1, \, \cdots, \, \underline{c}^n])=(tx, \, \rho_{\sigma}(t) [\underline{c}^1, \, \cdots, \, \underline{c}^n]), 
\end{equation*}
where  $\underline{c}^j$ is the transpose of $(c^j_1, \, \ldots, \, c^j_n)$.
\noindent
Now, consider the section 
\begin{equation}\label{frame_sec}
	s_{\sigma}: X_{\sigma} \rightarrow Fr(\mathcal{T}_X) \text{ given by } x \mapsto \left[\frac{\partial}{\partial x_1}\bigg|_x, \, \ldots,\, \frac{\partial}{\partial x_n}\bigg|_x\right], \text{ where } x \in X_{\sigma}. 
\end{equation}
Note that \({\rm GL}(n, \C)\) acts on \(Fr(\C^n)\) on the right by matrix multiplication as follows, 
\begin{equation}\label{frame_action}
	w \cdot g=[w_1,\, \cdots,\, w_n] [g^i_j]=\left[\sum\limits_{i} w_i g^i_1,\, \cdots, \,\sum\limits_{i} w_i g^i_n\right].
\end{equation}
Therefore, from \eqref{action_on_section}, \eqref{action_on_frame} and \eqref{frame_sec}, we have 
\begin{equation}\label{equivframe1}
	t s_{\sigma}(x) =  t \left[\frac{\partial}{\partial x_1}\bigg|_x, \, \cdots,\, \frac{\partial}{\partial x_n}\bigg|_x\right]=  \left[\chi^{u_1}(t) \frac{\partial}{\partial x_1}\bigg|_{tx}, \, \cdots,\, \chi^{u_n}(t) \frac{\partial}{\partial x_n}\bigg|_{tx}\right].
\end{equation}
On the other hand, from \eqref{tangent_action}, \eqref{frame_sec} and \eqref{frame_action}, we obtain  
\begin{equation}\label{equivframe2}
	\begin{split}
		s_{\sigma}(tx) \cdot \rho_{\sigma}(t) &=\left[\frac{\partial}{\partial x_1}\bigg|_{tx}, \, \cdots, \,  \frac{\partial}{\partial x_n}\bigg|_{tx}\right] \, \cdot \, \text{diag}(\chi^{u_1}(t), \, \ldots, \, \chi^{u_n}(t))\\
		& =\left[\chi^{u_1}(t) \frac{\partial}{\partial x_1}\bigg|_{tx}, \, \cdots, \, \chi^{u_n}(t) \frac{\partial}{\partial x_n}\bigg|_{tx}\right].
	\end{split}
\end{equation}
\noindent
Now comparing \eqref{equivframe1} and \eqref{equivframe2}, we have  \[t s_{\sigma}(x)=s_{\sigma}(tx) \cdot \rho_{\sigma}(t).\] Hence,  \(s_{\sigma}\), as defined in \eqref{frame_sec}, is a distinguished section with associated homomorphism \(\rho_{\sigma}\) (defined in \eqref{tangent_action}) for each cone \(\sigma \in \Xi(n)\).

\subsection{Equivariant automorphisms of the tangent frame bundle}

Let  $\lambda : \C^* \rightarrow {\rm GL}(n, \C)$ be a $1$-psg given by \[\lambda(z)=\text{diag}(z^{r_1}, \ldots, z^{r_n}),\] where \(z \in \C^* \text{ and } r_1, \, \ldots, \, r_n\) are integers. Then the associated parabolic subgroup is given by 
\begin{equation}\label{parabolic}
	P(\lambda)=\{(a_{ij}) \in {\rm GL}(n, \C)~|~ a_{ij}=0 \text{ whenever } r_i < r_j \}
\end{equation} (see \cite[Page 62]{GIT} ). Recall the generators $v_1, \ldots, v_n$ of $\sigma$ and the dual basis $u_1, \ldots, u_n$ of $M$ from Section \ref{tangentframe}.
Note that for \(z \in \C^*\),
$$\chi^{u_j}(\lambda^{v_l}(z))= \left\{ \begin{array}{ll}
	z& \text{ if } j=l,\\
	1 & \text{ otherwise. }  
\end{array} \right. 
$$
Then, using \eqref{tangent_action}, we have the $1$-psg \(\rho_{\sigma}\circ\lambda^{v_l} : \C^* \rightarrow
{\rm GL}(n, \C)\) defined by \begin{equation}\label{lth}
	(\rho_{\sigma}\circ\lambda^{v_l})(z)=\text{diag}(1, \ldots, 1,\underbrace{z}_{\text{\(l^{th}\)}},1, \ldots, 1),
\end{equation} where \(z \in \C^*, ~l=1, \, \ldots, \, n\). Thus, from \eqref{parabolic}, we have
\begin{equation}\label{affine_auto}
	\bigcap\limits_{\alpha \in \sigma(1)} P^{\alpha}=\bigcap\limits_{l=1}^n P(\rho_{\sigma} \circ \lambda^{v_l})=D(n, \C),
\end{equation}
where \(D(n, \C)\) denotes the group of all invertible diagonal matrices. Hence, from Theorem \ref{Aut}, $$\text{Aut}_T(Fr(\mathcal{T}_X))=\bigcap\limits_{\alpha \in \Xi(1)} P^{\alpha} \subseteq  \bigcap\limits_{\alpha \in \sigma(1)} P^{\alpha} = D(n, \C).$$ 
Therefore, we obtain
\begin{equation}\label{Aut_sc}
	\{a I_n ~|~ a \in \C^*\} \subseteq \text{Aut}_T(Fr(\mathcal{T}_X)) \subseteq D(n, \C).
\end{equation}
We now give explicit description of automorphism group of \(Fr(\mathcal{T}_X)\) for all nonsingular projective toric varieties \(X\) with Picard number $\leq 2$.
\bigskip

\noindent
\subsubsection{Projective space}\label{projspace} The only nonsingular projective toric variety with Picard number $1$ is the projective space. Let \(X=\mathbb{P}^n\) and let $\Xi$ denote the fan of \(\mathbb{P}^n\) in the lattice \(N=\Z^n\). Let \(e_1, \ldots, e_n\) denote the standard basis of \(\Z^n\) and set \(e_0=-e_1-\ldots-e_n\). Then the rays of \(\Xi\)  are \(e_0, e_1, \ldots, e_n\) and the maximal cones are \(\sigma_i=\text{Cone}(e_0, \ldots, \widehat{e}_i, \ldots, e_n)\), for \(i=0, \ldots, n\) (henceforth, by \(\widehat{e}_i\) we mean that \(e_i\) has been omitted from the relevant collection). Let \(e_1^*, \ldots, e_n^* \) be the dual basis of \(M\) corresponding to \(e_1, \ldots, e_n\). The affine toric variety corresponding to the maximal cone \(\sigma_0\) is \(X_{\sigma_0}=\text{Spec}(\C[x_1, \ldots, x_n])\), where \(x_i=\chi^{e_i^*}, \text{ for } i=0, \ldots, n\). To compute the automorphism group, we need to fix \(e \in Fr(\mathcal{T}_X)_{x_0}\) (see Corollary \ref{Aut}). Let \(e=[\frac{\partial}{\partial x_1}\big|_{x_0}, \, \cdots, \, \frac{\partial}{\partial x_n}\big|_{x_0}]\).

From \eqref{affine_auto} (taking \(\sigma=\sigma_0\)),  we get \begin{equation}\label{proj_affine_auto}
	\bigcap\limits_{\alpha \in \sigma_0(1)} P^{\alpha}=D(n, \C).
\end{equation}
Now, to compute \(P^{e_0}\) consider the cone \(\sigma_1=\text{Cone}(e_0,e_2, \ldots,e_n)\). The dual cone of $\sigma$ is $$\sigma_1^{\vee}=\text{Cone}(-e_1^*,-e_1^*+e_2^*, \ldots,-e_1^*+e_n^*)\,.$$  The corresponding affine toric variety is \(X_{\sigma_1}=\text{Spec}(\C[z_1, \ldots, z_n])\), where the coordinates satisfy the following relations: \begin{equation}\label{coordschange} z_1=x_1^{-1}, \,  z_2=x_1^{-1}x_2~, \, \ldots, \, z_n=x_1^{-1}x_n.\end{equation} 
Consider the distinguished section given in accordance with \eqref{frame_sec} by 
\begin{equation}\label{projsection1}
	s_{\sigma_1}: X_{\sigma_1} \rightarrow Fr(\mathcal{T}_X), \, x \mapsto \left[\frac{\partial}{\partial z_1}\bigg|_x, \, \cdots, \, \frac{\partial}{\partial z_n}\bigg|_x\right], 
\end{equation}
where  $x \in X_{\sigma_1}$. Denote the homomorphism corresponding to \(s_{\sigma_1} \) by $\rho_{\sigma_1}$. Then, by \eqref{lth}, we have
\begin{equation}\label{Pneq1}
	(\rho_{\sigma_1} \circ \lambda ^{e_0}) (z)=\text{diag}(z, 1, \ldots, 1), \text{ where }z \in \C^*.
\end{equation}
Using \eqref{coordschange}, we have the following transformation rules,
\begin{align*}
	\frac{\partial}{\partial z_1} &=-\frac{1}{z_1^2} \frac{\partial}{\partial x_1}-\frac{z_2}{z_1^2} \frac{\partial}{\partial x_2}- \ldots-  \frac{z_n}{z_1^2} \frac{\partial}{\partial x_n},\\
	& \\
	\frac{\partial}{\partial z_2}&=\frac{1}{z_1} \frac{\partial}{\partial x_2}, ~ \ldots~ , ~\frac{\partial}{\partial z_n}=\frac{1}{z_1} \frac{\partial}{\partial x_n}.
\end{align*}
Then, from \eqref{projsection1}, we have
$$s_{\sigma_1}(x_0)=\left[-\frac{\partial}{\partial x_1}\bigg|_{x_0}  - \frac{\partial}{\partial x_2}\bigg|_{x_0}-  \ldots  -\frac{\partial}{\partial x_n}\bigg|_{x_0}, ~ \frac{\partial}{\partial x_2}\bigg|_{x_0}, \, \cdots, \, \frac{\partial}{\partial x_n}\bigg|_{x_0}\right].$$
We consider the distinguished section \(s_{\sigma_1} \cdot g_{\sigma_1}\), where 
$$
g_{\sigma_1}=\left( \begin{array}{ccccc} 
	-1 & 0 & 0 & \cdots & 0 \\
	-1 & 1 & 0 & \cdots & 0 \\
	-1 & 0 & 1 & \cdots & 0 \\
	\vdots & \vdots & \vdots & \ddots& 0 \\
	-1 & 0 & 0 & \cdots & 1\ 
\end{array} \right). \\
$$
Let the homomorphism corresponding to \(s_{\sigma_1} \cdot g_{\sigma_1}\) be denoted by $\widehat{\rho}_{\sigma_1}$. Note that \((s_{\sigma_1} \cdot g_{\sigma_1})(x_0)=e\). Then, by Lemma \ref{distconj} and \eqref{Pneq1}, for \(z \in \C^*\), we have
\begin{align}\label{SLcase}
	(\widehat{\rho}_{\sigma_1} \circ \lambda ^{e_0}) (z)& =g_{\sigma_1}^{-1} ~ \text{diag} (z, 1, \ldots, 1) ~ g_{\sigma_1}
	=\left( \begin{array}{ccccc} 
		z & 0 & 0 & \cdots & 0 \\
		z-1 & 1 & 0 & \cdots & 0 \\
		z-1 & 0 & 1 & \cdots & 0 \\
		\vdots & \vdots & \vdots & \ddots& 0 \\
		z-1 & 0 & 0 & \cdots & 1\ 
	\end{array} \right), ~ z \in \C^*. \\
\end{align}
Thus, we obtain
\begin{equation}\label{proj_parabolic}
	\begin{split}
		P^{e_0}&= P(\widehat{\rho}_{\sigma_1} \circ \lambda ^{e_0})\\
		&=\{(a_{ij}) ~|~ \lim\limits_{z \rightarrow 0} (\widehat{\rho}_{\sigma_1} \circ \lambda ^{e_0})(z) (a_{ij}) (\widehat{\rho}_{\sigma_1} \circ \lambda ^{e_0})(z)^{-1} \text{ exists in } {\rm GL}(n, \C)\}\\
		& =\{(a_{ij}) ~|~ \sum\limits_{k=1}^n a_{ik}=\sum\limits_{j=1}^n a_{1j} , \ 2 \leq i \leq n \, ; \, (\sum\limits_{k=1}^n a_{kk})(a_{22}-a_{12}) \, \cdots \, (a_{nn}-a_{1n}) \neq 0\}.
	\end{split}
\end{equation}
Finally, from Theorem \ref{Aut}, and equations \eqref{proj_affine_auto} and \eqref{proj_parabolic}, we have 
\begin{equation}\label{autproj}
	\begin{split}
		\text{Aut}_T(Fr(\mathcal{T}_{\mathbb{P}^n}))&= \bigcap\limits_{k=0}^n P^{e_k}\\
		&= \{(a_{ij}) ~|~ a_{ij}=0 \text{ for } i \neq j \,; \, \sum\limits_{k=1}^n a_{ik}=\sum\limits_{j=1}^n a_{1j} , \, 2 \leq i \leq n ;\\
		& ~ ~ (\sum\limits_{k=1}^n a_{kk})(a_{22}-a_{12}) \, \cdots \, (a_{nn}-a_{1n}) \neq 0\}\\
		&= \{a I_n ~|~ a \in \C^*\}.
	\end{split}
\end{equation}

\noindent
\subsubsection{Picard number two case}\label{pictwo} Now we turn to nonsingular projective toric varieties \(X\) with Picard number 2. Then, \(X = \mathbb{P}(\mathcal{O}_{\mathbb{P}^s}  \oplus \mathcal{O}_{\mathbb{P}^s}(a_1) \oplus \cdots \oplus \mathcal{O}_{\mathbb{P}^s}(a_r)  ),\)  where \(s, r \geq 1\), \(s+r = \text{dim}(X)\), and \(0 \leq a_1 \leq \ldots \leq a_r\), by a result of Kleinschimidt (see \cite{Kleinschmidt}). We recall the fan structure of \(X\) from \cite[Example 7.3.5]{Cox}. Let $\Xi$ be the fan of \(X\) in the lattice \(N=\Z^s \times \Z^r\). Let \(\{v'_1, \ldots, v'_s\}\) and \(\{e_1', \ldots, e_r'\}\) be standard bases of \(\Z^s\) and \(\Z^r\), respectively. Set 
\begin{equation*}
	\begin{split}
		& v_i=(v'_i, {\bf{0}}) \in N \text{ for } 1 \leq i \leq s \text{ , }\\ & e_i=({\bf 0}, e_i') \in N  \text{ for }  1 \leq i \leq r,\\   &v_0=-v_1-\cdots-v_s+a_1e_1+ \cdots+a_r e_r \,, \text{ and } \\&   e_0=-e_1- \cdots-e_r \, .
	\end{split}
\end{equation*}
The rays of $\Xi$ are given by \( v_0, v_1 \ldots, v_s, e_0, e_1, \ldots, e_r \) and the maximal cones are given by
\[\text{Cone}(v_0, \ldots, \widehat{v}_j, \ldots, v_s, \, e_0, \ldots, \widehat{e}_i, \ldots, e_r ), 
\text{ for } j=0, \ldots, s \text{ and } i=0, \ldots, r.\]
Consider the maximal cone $\sigma=\text{Cone}(v_1, \ldots, v_s, e_1, \ldots, e_r)$. Let \(u_1, \ldots, u_s,u_{s+1}, \ldots, u_{s+r}\) be the dual basis of \(M\) corresponding to the basis \(v_1, \ldots, v_s, e_1, \ldots, e_r\) of \(N\). Then the dual cone is $\sigma^{\vee}=\text{Cone}(u_1, \ldots , u_{s+r})$,  and hence the corresponding affine toric variety is $X_{\sigma}=\text{Spec }\C[x_1, \ldots, x_{s+r}],$ where \(x_i=\chi^{u_i}\) for \(i=1, \ldots, s+r\). To compute the automorphism group, fix \(e=[\frac{\partial}{\partial x_1}\big|_{x_0}, \, \cdots, \, \frac{\partial}{\partial x_{s+r}}\big|_{x_0}] \in Fr(\mathcal{T}_X)_{x_0}\) (see Corollary \ref{Aut}). Thus arguing as before, from \eqref{affine_auto}, we have 
\begin{equation}\label{pic2_1}
	\bigcap\limits_{\alpha \in \sigma(1)} P^{\alpha}=D(s+r, \C).
\end{equation}
Note that from Theorem \ref{Aut}, we have
\begin{equation}\label{pic2eq1}
	\text{Aut}_T(Fr(\mathcal{T}_X))=\bigcap\limits_{\alpha \in \Xi(1)} P^{\alpha} = \left( \bigcap\limits_{\alpha \in \sigma(1)} P^{\alpha}\right)  \bigcap P^{e_0} \bigcap P^{v_0}.
\end{equation} 
\noindent
Let us consider the cone $$\tau=\text{Cone}(v_1, \ldots, v_s, e_0, e_2, \ldots, e_r).$$ The dual cone is given by $$\tau^{\vee}=\text{Cone}(u_1, \ldots, u_s, -u_{s+1}, -u_{s+1}+u_{s+2}, \ldots, -u_{s+1}+u_{s+r} ).$$ Denote the corresponding affine toric variety is $X_{\tau}=\text{Spec }\C[y_1, \ldots, y_{s+r}].$ 	These coordinates satisfy the following:
\begin{equation}\label{pic2_2}
	\begin{split}
		&x_1=y_1, ~\ldots~,~ x_s=y_s,\\
		& x_{s+1}=\frac{1}{y_{s+1}}, ~x_{s+2}=\frac{y_{s+2}}{y_{s+1}},~ \ldots~,~ x_{s+r}=\frac{y_{s+r}}{y_{s+1}}.
	\end{split}
\end{equation}
\noindent
Consider the distinguished section given by \[s_{\tau}: X_{\tau} \rightarrow Fr(\mathcal{T}_X),~ x \mapsto \left[\frac{\partial}{\partial y_1}\bigg|_x, \, \cdots ,\, \frac{\partial}{\partial y_{s+r}}\bigg|_x\right] \]
(cf. \eqref{frame_sec}). Denote the  corresponding homomorphism  by $\rho_{\tau}$. Then by \eqref{lth}, for \(z \in \C^*\) we have
\begin{equation}\label{pic2eq2}
	(\rho_{\tau} \circ \lambda ^{e_0})(z)  =\left( \begin{array}{cccc} 
		I_s & O \\
		O &  D(z)\
	\end{array} \right), 
\end{equation}
where \(I_s\) is the identity matrix of order \(s\) and $D(z)=\text{diag}(z, 1, \ldots, 1)$ is a diagonal matrix of order $r$.
\noindent
Using chain rule, from \eqref{pic2_2} we have,
\begin{align*}
	&\frac{\partial}{\partial y_1}=\frac{\partial}{\partial x_1}, ~\ldots~,~ \frac{\partial}{\partial y_s}=\frac{\partial}{\partial x_s},\\
	& \frac{\partial}{\partial y_{s+1}}=-\frac{1}{y_{s+1}^2} \frac{\partial}{\partial x_{s+1}}-\frac{y_{s+2}}{y_{s+1}^2} \frac{\partial}{\partial x_{s+2}}- \ldots-\frac{y_{s+r}}{y_{s+1}^2} \frac{\partial}{\partial x_{s+r}},\\
	& \frac{\partial}{\partial y_{s+2}}=\frac{1}{y_{s+1}} \frac{\partial}{\partial x_{s+2}}, ~\ldots~,~ \frac{\partial}{\partial y_{s+r}}=\frac{1}{y_{s+1}} \frac{\partial}{\partial x_{s+r}} \,.
\end{align*}
Therefore,
$$s_{\tau}(x_0)=\left[\frac{\partial}{\partial x_1}\bigg|_{x_0}~, ~ \ldots~ ,~ \frac{\partial}{\partial x_s}\bigg|_{x_0}, ~ -\frac{\partial}{\partial x_{s+1}}\bigg|_{x_0}-\ldots-\frac{\partial}{\partial x_{s+r}}\bigg|_{x_0}~, ~ \frac{\partial}{\partial x_{s+2}}\bigg|_{x_0} ~,~ \ldots ~,~ \frac{\partial}{\partial x_{s+r}}\bigg|_{x_0}\right].$$
We consider the distinguished section \(s_{\tau} \cdot g_{\tau}\), where \[g_{\tau}=\left( \begin{array}{cccc} 
	I_s & O \\
	O &  J_r\
\end{array} \right), \] 
\(I_s\) is the identity matrix of order \(s\), and \(J_r\) is the \(r \times r\) matrix given by
\begin{equation}\label{Jr}
	J_r=\left( \begin{array}{ccccc} 
		-1 & 0 & 0 & \cdots & 0 \\
		-1 & 1 & 0 & \cdots & 0 \\
		-1 & 0 & 1 & \cdots & 0 \\
		\vdots & \vdots & \vdots & \ddots& 0 \\
		-1 & 0 & 0 & \cdots & 1\ 
	\end{array} \right)_{r \times r}. 
\end{equation}	
Note that \((s_{\tau} \cdot g_{\tau})(x_0)=e.\) 
Denote the homomorphism corresponding to \(s_{\tau} \cdot g_{\tau}\) by $\widehat{\rho}_{\tau}$. Then, by Lemma \ref{distconj} and \eqref{pic2eq2}, for \(z \in \C^*\) we have
\begin{equation}\begin{split}\label{pic2eqn8}
		(\widehat{\rho}_{\tau} \circ \lambda ^{e_0})(z)  & =g_{\tau}^{-1} \left( \begin{array}{cccc} 
			I_s & O \\
			O &  D(z)\
		\end{array} \right) g_{\tau}\\
		& =\left( \begin{array}{cccc} 
			I_s & O \\
			O &  J_rD(z)J_r\
		\end{array} \right).
	\end{split}
\end{equation}
Let us define 
\begin{equation}\label{pic2matrixb}
	B_r(z):=J_rD(z)J_r=\left( \begin{array}{ccccc} 
		z & 0 & 0 & \cdots & 0 \\
		z-1 & 1 & 0 & \cdots & 0 \\
		z-1 & 0 & 1 & \cdots & 0 \\
		\vdots & \vdots & \vdots & \ddots& 0 \\
		z-1 & 0 & 0 & \cdots & 1\ 
	\end{array} \right)_{r \times r}. 
\end{equation}
Then, from \eqref{pic2eqn8}, we get 
\begin{equation}\label{pic2parabolic}
	(\widehat{\rho}_{\tau} \circ \lambda ^{e_0})(z) =\left( \begin{array}{cccc} 
		I_s & O \\
		O &  B_r(z)\
	\end{array} \right), 
\end{equation}
where \(z \in \C^*. \) Now, we have to compute 
\begin{equation*}
	\begin{split}
		P(\widehat{\rho}_{\tau} \circ \lambda ^{e_0}) \cap D(s+r, \C)= \{ (a_{ij}) \in  D(s+r, \C) ~|~& \lim\limits_{z \rightarrow 0} (\widehat{\rho}_{\tau} \circ \lambda ^{e_0})(z) (a_{ij}) (\widehat{\rho}_{\tau} \circ \lambda ^{e_0})(z)^{-1}\\
		&  ~ \text{ exists in } {\rm GL}(s+r, \C) \} \,.
	\end{split}
\end{equation*}
From \eqref{pic2parabolic}, we have
\begin{align*}
	(\widehat{\rho}_{\tau} \circ \lambda ^{e_0})(z)~ \text{diag}(a_{11}, \ldots, a_{s+r\, s+r}) ~ (\widehat{\rho}_{\tau} \circ \lambda ^{e_0})(z)^{-1} 
	&=\left( \begin{array}{cc} 
		A_1 & O \\
		O &  B_r(z) A_2 B_r(z)^{-1}\
	\end{array} \right),
\end{align*}
where \(A_1=\text{diag}(a_{11}, \ldots, a_{s s})\) and \(A_2=\text{diag}(a_{s+1\, s+1}, \ldots, a_{s+r\, s+r})\). Thus, 
\begin{equation}\label{e0cap} \begin{array}{ll}
		P^{e_0}\cap D(s+r, \C) &=P(\widehat{\rho}_{\tau} \circ \lambda ^{e_0}) \cap D(s+r, \C)\\ 
		& =\left\{\left( \begin{array}{cccc} 
			A_1 & O \\
			O &  a I_r\
		\end{array} \right) ~|~ a \in \C^*\, ,~
		A_1=\text{diag}(a_{11}, \ldots, a_{s s}) \in {\rm GL}(s, \C) \right\},
	\end{array}
\end{equation} 
by a similar argument as in \eqref{autproj}.

Finally, we consider the cone $$\gamma=\text{Cone}(v_0, v_2, \ldots, v_s, e_1, \ldots, e_r).$$ The dual cone of $\gamma$ is  $$\gamma^{\vee}=\text{Cone}(-u_1, -u_1+u_2, \ldots, -u_1+u_s, a_1u_1+u_{s+1}, \ldots, a_ru_1+u_{s+r})\,.$$ Denote the corresponding affine toric variety by $X_{\gamma}=\text{Spec }\C[z_1, \ldots, z_{s+r}].$ We have the following relations among the $z$ and the $x$ coordinates,
\begin{equation}\label{zxeqs}
	\begin{split}
		&x_1=\frac{1}{z_1}, ~x_2=\frac{z_2}{z_1}, ~\ldots~,~ x_s=\frac{z_s}{z_1},\\ &x_{s+1}=z_1^{a_1} z_{s+1}, ~\ldots~,~ x_{s+r}=z_1^{a_r} z_{s+r}.
	\end{split}
\end{equation}
By \eqref{frame_sec}, we have a distinguished section \ \[s_{\gamma}: X_{\gamma} \rightarrow Fr(\mathcal{T}_X),~ x \mapsto \left[\frac{\partial}{\partial z_1}\bigg|_x, \, \cdots , \, \frac{\partial}{\partial z_{s+r}}\bigg|_x\right] \,. \]
Let $\rho_{\gamma}$ denote the corresponding homomorphism. Then for \(z \in \C^*\), we have
\begin{equation}\label{pic2eq5}
	(\rho_{\gamma} \circ \lambda ^{e_0})(z)  =\left( \begin{array}{cccc} 
		D_s(z) & O \\
		O & I_r \
	\end{array} \right), 
\end{equation}
where \(I_r\) is the identity matrix of order \(r\) and $D_s(z)=\text{diag}(z, 1, \ldots, 1)$ is a diagonal matrix of order \(s\). 
Applying chain rule on \eqref{zxeqs}, we have
\begin{equation*}
	\begin{split}
		& \frac{\partial}{\partial z_1}=-\frac{1}{z_1^2} \frac{\partial}{\partial x_1}-\frac{z_2}{z_1^2} \frac{\partial}{\partial x_2}-\ldots-\frac{z_s}{z_1^2} \frac{\partial}{\partial x_s}+a_1 z_1^{a_1-1} z_{s+1} \frac{\partial}{\partial x_{s+1}}+ \ldots+ a_r z_1^{a_r-1} z_{s+r} \frac{\partial}{\partial x_{s+r}}, \\
		& \frac{\partial}{\partial z_2}= \frac{1}{z_1} \frac{\partial}{\partial x_2},~ \ldots~,~ \frac{\partial}{\partial z_s}= \frac{1}{z_1} \frac{\partial}{\partial x_s}, \\
		& \frac{\partial}{\partial z_{s+1}}= z_1^{a_1} \frac{\partial}{\partial x_{s+1}}, ~\ldots~,~ \frac{\partial}{\partial z_{s+r}}= z_1^{a_r} \frac{\partial}{\partial x_{s+r}}\, .
	\end{split}
\end{equation*}
Thus, $$s_{\gamma}(x_0)=\left[-\frac{\partial}{\partial x_{1}}\bigg|_{x_0}-\ldots-\frac{\partial}{\partial x_{s}}\bigg|_{x_0}+a_1 \frac{\partial}{\partial x_{s+1}}\bigg|_{x_0}+\ldots +a_r\frac{\partial}{\partial x_{s+r}}\bigg|_{x_0}, ~\frac{\partial}{\partial x_2}\bigg|_{x_0}, ~ \cdots,  ~ \frac{\partial}{\partial x_{s+r}}\bigg|_{x_0}\right].$$
We consider the distinguished section \(s_{\gamma} \cdot g_{\gamma}\), where \(g_{\gamma}=\left( \begin{array}{cc} 
	J_s & O \\
	A &  I_r\
\end{array} \right), \\\) and 
$$
A=\left( \begin{array}{cccc} 
	a_1  & 0 & \cdots & 0 \\
	a_2  & 0 & \cdots & 0 \\
	\vdots  & \vdots & \ddots& 0 \\
	a_r & 0 & \cdots & 1\ 
\end{array} \right)_{r \times s}. \\
$$
Note that \((s_{\gamma} \cdot g_{\gamma})(x_0)=e\). Let $\widehat{\rho}_{\gamma}$ denote the corresponding homomorphism.	Then from Lemma \ref{distconj} and \eqref{pic2eq5}, for \(z \in \C^*\), we have
\begin{equation}\label{pic2eqn6}
	\begin{split}
		(\widehat{\rho}_{\gamma} \circ \lambda ^{v_0})(z)   =g_{\gamma}^{-1} \left( \begin{array}{cccc} 
			D_s(z) & O \\
			O &  I_r\
		\end{array} \right) g_{\gamma} =\left( \begin{array}{cccc} 
			J_sD_s(z)J_s & O \\
			AD_s(z)J_s+A &  I_r\
		\end{array} \right).
	\end{split}
\end{equation}
As in \eqref{pic2matrixb}, set \(B_s(z)=J_sD_s(z)J_s \). Define \[C(z):=AD_s(z)J_s+A=\left( \begin{array}{cccc} 
	a_1(1-z)  & 0 & \cdots & 0 \\
	a_2 (1-z) & 0 & \cdots & 0 \\
	\vdots  & \vdots & \ddots& 0 \\
	a_r(1-z) & 0 & \cdots & 1\ 
\end{array} \right). \]
Thus from \eqref{pic2eqn6}, we have 
\begin{equation}\label{pic2eqn7}
	(\widehat{\rho}_{\gamma} \circ \lambda ^{v_0})(z)=\left( \begin{array}{cccc} 
		B_s(z) & O \\
		C(z) &  I_r\
	\end{array} \right) .
\end{equation}
Recall that 
\begin{align*}
	P^{v_0}=P(\widehat{\rho}_{\gamma} \circ \lambda ^{v_0})=&\{(a_{ij}) ~|~ \lim\limits_{z \rightarrow 0} (\widehat{\rho}_{\gamma} \circ \lambda ^{v_0})(z) (a_{ij}) (\widehat{\rho}_{\gamma} \circ \lambda ^{v_0})(z)^{-1} \text{ exists in } {\rm GL}(n, \C)\}.
\end{align*}
Now, consider \((a_{ij}) \in P^{e_0}\cap D(s+r, \C) \).  Recall from \eqref{e0cap} that \((a_{ij})\) is of the  form, 
\begin{equation}\label{matrix}
	(a_{ij})=\left( \begin{array}{cccc} 
		A_1 & O \\
		O &  a I_r\
	\end{array} \right), \text{ where } A_1=\text{diag}(a_{11}, \ldots, a_{s, s}) \text{ and } a \in \C^*.
\end{equation}
Then from \eqref{pic2eqn7} and \eqref{matrix}, we have
\begin{align*}
	(\widehat{\rho}_{\gamma} \circ \lambda ^{v_0})(z) (a_{ij}) (\widehat{\rho}_{\gamma} \circ \lambda ^{v_0})(z)^{-1}  &=\left( \begin{array}{cccc} 
		B_s(z) & O \\
		C(z) &  I_r\
	\end{array} \right)  \left( \begin{array}{cccc} 
		A_1 & O \\
		O &  a I_r\
	\end{array} \right) \left( \begin{array}{cccc} 
		B_s(z)^{-1} & O \\
		-C(z)B_s(z)^{-1} &  I_r\
	\end{array} \right)\\
	&=\left( \begin{array}{cccc} 
		B_s(z) A_1 B_s(z)^{-1}& O \\
		C(z)A_1B(z)^{-1}+ \lambda \widetilde{C}(z) &  a I_r\
	\end{array} \right),
\end{align*}
where \(\widetilde{C}(z)=-C(z)B(z)^{-1}= \frac{1}{z} \left( \begin{array}{cccc} 
	a_1(1-z)  & 0 & \cdots & 0 \\
	a_2 (1-z) & 0 & \cdots & 0 \\
	\vdots  & \vdots & \ddots& 0 \\
	a_r(1-z) & 0 & \cdots & 1\ 
\end{array} \right).\\\) 
Now, $$\lim\limits_{z \rightarrow 0} (\widehat{\rho}_{\gamma} \circ \lambda ^{v_0})(z) (a_{ij}) (\widehat{\rho}_{\gamma} \circ \lambda ^{v_0})(z)^{-1} \text{ exists in } {\rm GL}(s+r, \C)$$ if and only if
\begin{equation}\label{pic2lim1}
	\lim\limits_{z \rightarrow 0} B_s(z) A_1 B_s(z)^{-1} \text{ exists in } {\rm GL}(s, \C) 
\end{equation}
and 
\begin{equation}\label{pic2lim2}
	\lim\limits_{z \rightarrow 0} C(z)A_1B_s(z)^{-1}+ \lambda \widetilde{C}(z) \text{ exists in } M_{r \times s}( \C).
\end{equation}
As in the computation of $\text{Aut}_T(Fr(\mathcal{T}_{\mathbb{P}^n}))$ (cf. \eqref{autproj}),	the limit in \eqref{pic2lim1} exists if and only if 
\begin{equation}\label{matrix2}
	A_1=b I_s, \text{ for some } b \in \C^* . 
\end{equation} 
Suppose that the limit in \eqref{pic2lim1} exists. This implies that \(A_1=b I_s\) for some $b \in \C^*$. Then, from \eqref{pic2lim2}, we see that
\begin{equation}\label{lim}
	C(z)A_1B(z)^{-1}+ \lambda \widetilde{C}(z)=\frac{1}{z} \left( \begin{array}{cccc} 
		(a_1 b-a a_1)(1-z)  & 0 & \cdots & 0 \\
		(a_2 b-a a_2) (1-z) & 0 & \cdots & 0 \\
		\vdots  & \vdots & \ddots& 0 \\
		(a_r b-a a_r)(1-z) & 0 & \cdots & 1\ 
	\end{array} \right).
\end{equation}
Hence, the limit of \eqref{lim} as \(z \rightarrow 0\) exists if and only if 
$$(b-a) a_1 \,=\, \ldots \,= \, (b-a) a_r=0.$$ 
This implies that $b=a$ if \((a_1, \ldots, a_r)\neq (0, \ldots, 0)\), and otherwise, $b$ and $a$ are independent elements in \(\C^*\).
Thus, we have 
\begin{align*}
	&\text{Aut}_T(Fr(\mathcal{T}_{X}))= \bigcap\limits_{k=0}^n P^{e_k}= 	\begin{cases}
		\{a I_{s+r} ~|~ a \in \C^*\} , &{\rm if}\ (a_1, \ldots, a_r)\neq (0, \ldots, 0), \  \\
		\left\{ \left( \begin{array}{cccc} 
			b I_s & O \\
			O&  a I_r\
		\end{array} \right) ~ |~ a, b \in \C^* \right\},  & {\rm if}\ (a_1, \ldots, a_r)= (0, \ldots, 0).
	\end{cases}
\end{align*}

\begin{rmk}{\rm 
		One can also use the Klyachko filtrations of the tangent bundle to compute its equivariant automorphisms and derive the above results in a relatively easier manner. 
		%For example, over $\mathbb{P}^n$, the automorphism group is the collection of linear automorphisms of \(\C^n\) fixing the subspaces \(\text{Span}(e_i)\), \(i=0, \ldots, n\). Hence, \(\text{Aut}_T(\mathcal{T}_{\mathbb{P}^n})= \{a I_n ~|~ a \in \C^*\}.\)	
		In the next subsection, we give an example that is beyond the scope of Klyachko filtrations. 
	}
\end{rmk}

\subsection{Extension of structure group}\label{extsg}

Let $\phi: G \rightarrow G'$ be a homomorphism of algebraic groups. Let $\mathcal{P}$ be a principal \(G\)-bundle on \(X\). Denote by \(\mathcal{P}_{\phi}: = \mathcal{P} \times_{G} G'\) be the principal \(G'\)-bundle obtained from $\mathcal{P}$ by extending the structure group via the homomorphism $\phi$. Then we have a homomorphism between the bundle automorphism groups 
\begin{equation*}
\begin{split}
		\tilde{\phi}: \text{Aut}^G(\mathcal{P}) & \rightarrow \text{Aut}^{G'}(\mathcal{P}_{\phi})\\
		\Phi & \mapsto ([(e, g')] \mapsto [(\Phi(e), g')]), \, \text{for } e \in \mathcal{P}, g' \in G'.
\end{split}
\end{equation*}
Note that if $\phi$ is injective, then the induced map \(\tilde{\phi}\) is also injective. To see this, let \(\Phi_1, \Phi_2 \in \text{Aut}^G(\mathcal{P})\) such that \(\tilde{\phi}(\Phi_1)=\tilde{\phi}(\Phi_2)\). Then for any \(e \in \mathcal{P}\) and \(g' \in G'\), we have \([(\Phi_1(e), g')]=[(\Phi_2(e), g')]\). Hence there exists \(g \in G\) such that
\begin{equation*}
	\begin{split}
		& (\Phi_1(e) g, \phi(g)^{-1} g')=((\Phi_2(e), g')).\\
	\end{split}
\end{equation*}
This forces \(g=1_G\) as $\phi$ is injective and hence we get \(\Phi_1=\Phi_2\).

If in addition, $\mathcal{P}$ is a toric principal bundle, then similar result holds for the equivariant bundle automorphidm group \(\text{Aut}_T(\mathcal{P})\).

Let \(X=\mathbb{P}^2\) and \(\mathcal{T}\) denote the tangent bundle on \(X\). Let \(\mathcal{P}=Fr(\mathcal{T})\) denote the associated frame bundle. Consider the injective homomorphism $\phi: {\rm GL}(2, \C) \rightarrow {\rm SL}(3, \C)$ given by 
\[A \longmapsto \left( \begin{array}{cccc} 
	A & O \\
	O &  \text{det}(A)^{-1}\
\end{array} \right).\]

Let \(\mathcal{P}_{\phi_1}:= \mathcal{P} \times_{{\rm GL}(2, \C)} {\rm SL}(3, \C)\) be the extension of structure group to \({\rm SL}(3, \C)\) via the map $\phi$. Then by Lemma \ref{asso}, the admissible collection associated to \(\mathcal{P}_{\phi_1}\) is given by \(\{\phi \circ \rho_{\sigma}, \phi(P(\tau, \sigma))\}\), where \(\left\lbrace \rho_{\sigma}, P(\tau, \sigma)\right\rbrace \) is the admissible collection of the tangent frame bundle considered in Section \ref{tangentframe}. Let $\sigma=\text{Cone}(e_1, e_2)$ (recall the description of the fan of $\mathbb{P}^2$ from Section \ref{projspace}). To get the parabolic subgroups \(P^{e_i}, \, i=1,2\) consider the 1-psgs given by
\begin{equation*}\label{1psgSl} 
	(\phi \circ \rho_{\sigma} \circ \lambda^{e_1})(z)=\text{diag}(z, 1, \frac{1}{z}) \text{ and } (\phi \circ \rho_{\sigma} \circ \lambda^{e_2})(z)=\text{diag}(1, z, \frac{1}{z}), \, z \in \C^*. 
\end{equation*}
Note that if $\lambda : \C^* \rightarrow {\rm SL}(3, \C)$ is a $1$-psg given by \[\lambda(z)=\text{diag}(z^{r_1}, z^{r_2}, z^{r_3}),\] where \(z \in \C^* \text{ and } r_1, \, r_2, \, r_3\) are integers. Then the associated parabolic subgroup is given by 
\begin{equation}\label{SLparabolic}
	P(\lambda)=\{(a_{ij}) \in {\rm SL}(3, \C)~|~ a_{ij}=0 \text{ whenever } r_i < r_j \}.
\end{equation}
Hence, we have
\begin{equation}\label{capSL}
	P^{e_1} \cap P^{e_2} =\left\lbrace  \left( \begin{array}{cccc} 
		a_{11} & 0 & a_{13} \\
		0 & a_{22} & a_{23}  \\
		0 & 0 & a_{33}\
	\end{array} \right) \in  {\rm SL}(3, \C) ~|~ a_{11}, a_{13}, a_{22}, a_{23}, a_{33} \in \C, \, a_{11}a_{22}a_{33}=1 \right\rbrace .
\end{equation}
Now we compute \(P^{e_0}\). We consider the cone \(\sigma_1=\text{Cone}(e_0, e_2)\) (cf. Section \ref{projspace}). Then the distinguished section on \(X_{\sigma_1}\) is given by
\begin{equation}\label{Sl1}
	s'_{\sigma_1}(x)=\left[\frac{\partial}{\partial z_1}\bigg|_x, \, \frac{\partial}{\partial z_2}\bigg|_x, \, I_3\right] \in \mathcal{P} \times_{{\rm GL}(2, \C)} {\rm SL}(3, \C).  
\end{equation}
Using \eqref{coordschange}, we have the following transformation rules,
\begin{align*}
	\frac{\partial}{\partial z_1} =-\frac{1}{z_1^2} \frac{\partial}{\partial x_1}-\frac{z_2}{z_1^2} \frac{\partial}{\partial x_2}, ~
	\frac{\partial}{\partial z_2}=\frac{1}{z_1} \frac{\partial}{\partial x_2}.
\end{align*}
Then, from \eqref{Sl1}, we have
$$s'_{\sigma_1}(x_0)=\left[-\frac{\partial}{\partial x_1}\bigg|_{x_0}  - \frac{\partial}{\partial x_2}\bigg|_{x_0}, ~ \frac{\partial}{\partial x_2}\bigg|_{x_0}, \, I_3\right].$$
We consider the distinguished section \(s'_{\sigma_1} \cdot g'_{\sigma_1}\), where 
$$
g'_{\sigma_1}=\phi(g_{\sigma_1})=\left( \begin{array}{ccccc} 
	-1 & 0 & 0  \\
	-1 & 1 & 0 \\
	0 & 0 & -1  \
\end{array} \right), \\
$$
where $g_{\sigma_1}$ is defined in \eqref{SLcase}. Note that 
\begin{equation*}
	\begin{split}
		(s'_{\sigma_1} \cdot g'_{\sigma_1})(x_0)& =s'_{\sigma_1}(x_0) \cdot g'_{\sigma_1}=\left[-\frac{\partial}{\partial x_1}\bigg|_{x_0}  - \frac{\partial}{\partial x_2}\bigg|_{x_0}, ~ \frac{\partial}{\partial x_2}\bigg|_{x_0}, \, \phi(g_{\sigma_1}) \right]\\
		& =\left[ \left[-\frac{\partial}{\partial x_1}\bigg|_{x_0}  - \frac{\partial}{\partial x_2}\bigg|_{x_0}, ~ \frac{\partial}{\partial x_2}\bigg|_{x_0}\right] \cdot g_{\sigma_1}, \, I_3 \right]\\
		& =\left[ s_{\sigma_1}(x_0), \, I_3 \right].
	\end{split}
\end{equation*}
 Let the homomorphism corresponding to \(s'_{\sigma_1} \cdot g'_{\sigma_1}\) be denoted by $\widehat{\rho}'_{\sigma_1}$.  Then, by Lemma \ref{distconj} and \eqref{lth}, for \(z \in \C^*\), we have
\begin{align*}
	(\widehat{\rho}'_{\sigma_1} \circ \lambda ^{e_0}) (z)& =g_{\sigma_1}'^{-1} ~ \text{diag} (z, 1, \frac{1}{z}) ~ g'_{\sigma_1}
	=\left( \begin{array}{ccccc} 
		z & 0 & 0   \\
		z-1 & 1 & 0  \\
	0	 & 0 & \frac{1}{z}  \
	\end{array} \right), ~ z \in \C^*. \\
\end{align*}
Thus, we obtain
\begin{equation}\label{proj_parabolic}
	\begin{split}
		P^{e_0}&= P(\widehat{\rho}'_{\sigma_1} \circ \lambda ^{e_0})\\
		&=\{(a_{ij}) ~|~ \lim\limits_{z \rightarrow 0} (\widehat{\rho}'_{\sigma_1} \circ \lambda ^{e_0})(z) (a_{ij}) (\widehat{\rho}'_{\sigma_1} \circ \lambda ^{e_0})(z)^{-1} \text{ exists in } {\rm SL}(3, \C)\}.
	\end{split}
\end{equation}
Now for \((a_{ij}) \in P^{e_0} \cap P^{e_1}\), we get that
\begin{equation}\label{capSL1}
	\begin{split}
(\widehat{\rho}'_{\sigma_1} \circ \lambda ^{e_0})(z) (a_{ij}) (\widehat{\rho}'_{\sigma_1} \circ \lambda ^{e_0})(z)^{-1} & =\left( \begin{array}{ccccc} 
			z & 0 & 0   \\
			z-1 & 1 & 0  \\
			0	 & 0 & \frac{1}{z}  \
		\end{array} \right) 
	\left( \begin{array}{cccc} 
		a_{11} & 0 & a_{13} \\
		0 & a_{22} & a_{23}  \\
		0 & 0 & a_{33}\
	\end{array} \right)
	\left( \begin{array}{ccccc} 
		\frac{1}{z} & 0 & 0   \\
		\frac{1}{z}-1 & 1 & 0  \\
		0	 & 0 & z  \
	\end{array} \right)\\
&=\left( \begin{array}{cccc} 
	a_{11} & 0 & a_{13} z\\
	a_{11}-\frac{a_{11}}{z}+\frac{a_{22}}{z}-a_{22} & a_{22} & a_{13}(z-1)+a_{23}  \\
	0 & 0 & a_{33}\
\end{array} \right).
\end{split}
\end{equation}
The limit as \(z\) tends to zero exists if and only if \(a_{11}=a_{22}\) and \(a_{33}=\frac{1}{a_{11} \, a_{22}}\).
Finally, from Theorem \ref{Aut}, and equations \eqref{capSL} and \eqref{capSL1}, we have 
\begin{equation}\label{autproj}
	\begin{split}
		\text{Aut}_T(\mathcal{P}_{\phi})&= \bigcap\limits_{k=0}^3 P^{e_k}\\
		&= \left\lbrace  \left( \begin{array}{cccc} 
			a_{11} & 0 & a_{13} \\
			0 & a_{11} & a_{23}  \\
			0 & 0 & \frac{1}{a_{11}^2}\
		\end{array} \right) ~|~ a_{11} \in \C^*, \,  a_{13}, a_{23} \in \C \right\rbrace .
	\end{split}
\end{equation}
Hence the bundle automorphism group \(\text{Aut}_T(\mathcal{P}_{\phi})\), after extension of structure group, becomes strictly bigger.

\end{document}